\documentclass{amsproc}
 \usepackage{tabu}
\newtheorem{theorem}{Theorem}[section]

\usepackage{algorithm}
\usepackage{algorithmic}
\theoremstyle{definition}

\newtheorem{example}[theorem]{Example}

\usepackage{algorithm}
\usepackage{algorithmic}
\theoremstyle{remark}

\numberwithin{equation}{section}
\usepackage{graphicx}
\usepackage{latexsym}
\usepackage{subfigure}
\usepackage{amsmath}
\usepackage{amssymb}
\usepackage{amsfonts}
\usepackage{verbatim}
\usepackage{mathrsfs}
\usepackage[latin1]{inputenc}
\usepackage{color}
\usepackage[colorlinks,citecolor=blue,urlcolor=blue]{hyperref}
\usepackage{caption}
\newtheorem{tm}{Theorem}[section]
\newtheorem{rk}{Remark}[section]

\newtheorem{prop}{Proposition}[section]
\newtheorem{lm}{Lemma}[section]
\newtheorem{cor}{Corollary}[section]
\newtheorem{ex}{Example}[section]
\usepackage{enumitem}

\newcommand{\R}{\mathbb R}

\newcommand{\OOO}{\mathcal O}

\newcommand{\Rd}{\R^d}

\newcommand{\<}{\langle}
\renewcommand{\>}{\rangle}
\numberwithin{figure}{section}

\newcommand{\zhou}[1]{{\color{blue} (#1 (Zhou))}}

\begin{document}

\title[Multiple shooting method for Wasserstein geodesic equation]
{
A continuation multiple shooting method for Wasserstein geodesic equation}

\author{Jianbo Cui}
\address{School of Mathematics, Georgia Tech, Atlanta, GA 30332, USA}
\curraddr{}
\email{jcui82@math.gatech.edu}
\thanks{}
\author{Luca Dieci}
\address{School of Mathematics, Georgia Tech, Atlanta, GA 30332, USA}
\curraddr{}
\email{dieci@math.gatech.edu}
\thanks{}
\author{Haomin Zhou}
\address{School of Mathematics, Georgia Tech, Atlanta, GA 30332, USA}
\curraddr{}
\email{hmzhou@math.gatech.edu}
\thanks{}

\subjclass[2010]{}


\date{\today}

\dedicatory{}

\begin{abstract}
In this paper, we propose a numerical method to solve the classic $L^2$-optimal transport problem.
Our algorithm is based on use of multiple shooting, in combination with a 
continuation procedure, to solve the boundary value problem associated to the transport problem.
We exploit the viewpoint of Wasserstein Hamiltonian flow with initial and target densities, and
our method is designed to retain the underlying Hamiltonian structure.
Several numerical examples are presented to illustrate the performance of the method.

\end{abstract}

\maketitle

\section{Introduction}

Optimal transport (OT) has a long and rich history, and it finds applications in various fields, such as image processing, 
machine learning and  economics  (e.g., see \cite{MBMOV16,San15}). 
The first mass transfer problem, a civil engineering problem, was considered by Monge in 1781.  
A modern treatment of this problem, in term of probability densities, was studied by Kantorovich in 
\cite{Kan04}.  In this light, the optimal transport problem consists in moving a certain probability density into another,
while minimizing a given cost functional.  Depending on whether (one or both of) the densities are
continuous or discrete, one has a fully discrete, or a semi-discrete, or a continuous  OT problem.  In this work,
we consider a continuous OT problem subject to the cost given by the squared $L^2$ norm.  This is 
the  most widely studied continuous OT problem, and
the formulation we adopt in this paper is based on an optimal control formulation in a fluid mechanics framework, 
known as {\emph{Benamou-Brenier formula}}, 
established in \cite{BB00}.  The starting point is to cast the OT problem in a  variational form as  
\begin{equation}\label{min-e}\begin{split}
\inf_{v}\{\int_0^1\<v ,v \>_{\rho}dt\ : \, 
\partial_t \rho  + \nabla\cdot(\rho  v )=0, \rho(0)=\mu, \rho(1)=\nu \},
\end{split}\end{equation}
where $ \<v,v\>_{\rho}:=\int_{\mathbb R^d} |v|^2 \rho dx$ with 
smooth velocity field  $v(t,x)\in \mathbb R^d$, and $\mu$ and $\nu$ are
probability density functions satisfying $\int_{\mathbb R^d}|x|^2\mu(x) dx,\int_{\mathbb R^d}|x|^2\nu(x) dx<+\infty$. 
This ensures the existence and uniqueness of the optimal map $M^*$ for the equivalent Monge-Kantorovich problem of 
\eqref{min-e}, i.e., $\inf_{M} \int_{\mathbb R^d}|M(x)-x|^p \mu(x) dx$ with $M: \mathbb R^d \to \mathbb R^d$ transferring $\mu$ to $\nu$ 
(see e.g., \cite[Theorem 1.22]{San15}). 
Moreover, the optimal map has the form $M^*(x)=\nabla \psi(x)=x+\nabla \phi(x)$, $\mu$-$a.s.,$ with a convex function $\psi(x)$.  
From \cite{BB00}, we have that $\nabla \phi(x)=v(0,x)$ and that the characteristic line $(X(t,x),v(t,X(t,x)))$ satisfies 
\begin{align*}
\partial_t  \rho(t,X(t,x))+\nabla \cdot (\rho(t,X(t,x)) v(t,X(t,x)))=0,\\
\partial_t v(t,X(t,x))+\nabla (\frac 12|v(t,X(t,x))|^2)=0. 
\end{align*}
When $X(t,x)=x+t v(0,x)$ is invertible, we obtain that 
$\rho(t)=X(t,\cdot)^\# \rho(0)$ and that $v(t,x)=v(0,X^{-1}(t,x))=\nabla \psi(0,X^{-1}(t,x)).$ 
We refer to \cite{MR1177479,MR1440931,San15} and references therein for results about 
regularity of $M^*$ and $\psi.$
The optimal value in \eqref{min-e} is known as
the $L^2$-Wasserstein distance square between $\mu$ and $\nu$, and written as  $g_W^2(\mu, \nu)$.
The formulation \eqref{min-e} is interpreted as finding the optimal vector field $v$ to transport the
given density function $\mu$ to the density $\nu$ with the minimal amount of kinetic energy. 
(We emphasize that the ``time variable'' $t$ has no true physical meaning, and it serves the role of
a homotopy parameter.)

By introducing  the new variable $S$ satisfying $v=\nabla S$, 
the critical point of \eqref{min-e} satisfies (up to a spatially independent function $C(t)$)
the following system in the unknowns $(\rho, S)$:
\begin{equation}\label{GeodEqn1}\begin{cases}
& \partial_t \rho + \nabla \cdot ({\rho }\nabla S)=0 \\
& \partial_t S +\frac 12 |\nabla S|^2 =0,
\end{cases}\end{equation}
subject to boundary conditions $\rho(0)=\mu, \rho(1)=\nu$. 
This is the well-known geodesic equation between
two densities $\mu$ and $\nu$ on the Wasserstein manifold \cite{Vil09}, and can also be viewed as a
Wasserstein Hamiltonian flow with Hamiltonian $ H(\rho,S)=\frac 12\int_{\mathbb R^d}|\nabla S|^2\rho dx$ when $C(t)=0$, \cite{CLZ20}.
If $S^0=S\Big|_{t=0}$ is known, 
the optimal value $g_W(\mu, \nu)$, the $L^2$-Wasserstein distance between $\mu$ and $\nu$, 
equals $\sqrt{2  H(\mu,S^0)}$. 



\begin{rk}\label{SnotUnique}
Obviously, $S$ is defined only up to an arbitrary constant.  As a consequence, the $(\rho, S)$ formulation 
\eqref{GeodEqn1} of the boundary value problem cannot have a unique solution.  
Because of this fact, we will in the end
reverse to using a formulation based on $\rho$ and $v$, but the Hamiltonian structure of \eqref{GeodEqn1} will
guide us in the development of appropriate semi-discretizations of the problem in the $(\rho, v)$ variables.
\end{rk}

In recent years, there have been several numerical studies concerned with approximating solutions of OT problems,
and many of them are focused on the continuous problem considered in this work, that is on computation of
the Wasserstein distance $g_W$ and the underlying OT map. 
A key result in this context is that the optimal map is the gradient of a convex function $u$, 
which is the solution of  the so-called Monge-Amp\'ere equation, a non linear elliptic PDE subject to
non-standard boundary conditions.  We refer to 
\cite{BB99,BFO14,Froese, GLSY16, OlikerPrussner, PBTIT15,WBBC16}, 
for a sample of numerical work on the solution of the Monge-Amp\'ere equation.
For different approaches, in the case of continuous,
discrete, and semi-discrete OT problems, and for a variety of cost functions, we refer to
\cite{CHYGT18, Cuturi, LucaJD, LYO18, ObermanRuan, PPO14, RCLO18, TWK18}.

However, numerical approximation of the solution of the geodesic equation has received little attention, and this
is our main scope in this computational paper. There are good reasons to consider solving the geodesic equation: at once one can
recover the Wasserstein distance, the OT map, and the ``time dependent'' vector field producing the optimal trajectory. 
At the same time, there are also a number of obstacles that make the numerical solution of the 
Wasserstein geodesic equation very challenging: 
the density $\rho$ needs to be non-negative, mass conservation is required, and 
retaining the underlying symplectic structure is highly desirable too. 
Another hurdle, which is not at all obvious, is that the Hamiltonian system \eqref{GeodEqn1} with initial values on 
the Wasserstein manifold often develops singularities in finite time (see e.g. \cite{CDZ20}). 
These challenges must be overcome when designing numerical schemes for the boundary value problem
\eqref{GeodEqn1}.

In this paper, we propose to compute the solution of \eqref{GeodEqn1} by combining a multiple shooting method, in conjunction with
a continuation strategy, for an appropriate semi-discretization of \eqref{GeodEqn1}.
First, we consider a spatially discretized version of \eqref{GeodEqn1},  which will give a (large) boundary value problem of ODEs.
To solve the latter, we will use a multiple shooting method, 
whereby the interval $[0,1]$ is partitioned into several subintervals, $[0,1]=\cup_{i=0}^{K-1}[t_i,t_{i+1}]$, initial guesses for the density
and the velocity are provided at each $t_i$, $i=0,\dots, K-1$, initial value problems are solved on $[t_i,t_{i+1}]$,
and eventually enforcement of continuity and boundary conditions will result in a large nonlinear
system to solve for the density $\rho$ and velocity $v$ at each $t_i$.
To solve the nonlinear system, we use
Newton's method, and --to enhance its convergence properties-- we will adopt a continuation method to
obtain good initial guesses for the Newton's iteration. 

Multiple shooting is a well studied technique for solving two-point boundary value problems of ordinary differential equations 
(TPBVPs of ODEs), and we refer to \cite{Keller:NumSoltpbvps} for an early derivation of the method, and to 
\cite{AscherMattheijRussell:SolnBVPs}
for a comprehensive review of techniques for solving TPBVPs of ODEs, and relations
(equivalence) between many of them.  
Our main reason for adopting multiple shooting is its overall simplicity, and the ease with which we can adopt appropriate time
discretizations of symplectic type (on sufficiently short time intervals) in order to avoid finite time singularities when solving
\eqref{GeodEqn1} subject to given initial conditions.

The rest of paper is organized as follows. In Section \ref{semidisc}, we briefly review the continuous OT problem and introduce a spatial
discretization to convert \eqref{GeodEqn1} into Hamiltonian ODEs.  At first, we propose the semi-discretization for the $(\rho, S)$ variables,
but then in Section \ref{MSsection} we will  revert it to the $(\rho, v)$ variables, which are those with which we end up working.
The multiple shooting method, and the continuation strategy, are also presented in Section \ref{MSsection} .   Results of 
numerical experiments are presented in Section \ref{NumExs}.

\section{Spatially discrete OT problems}\label{semidisc}

In this section, we introduce the spatial discretization of \eqref{GeodEqn1}.   
First of all, we need to truncate $\mathbb{R}^d$ to a finite computational domain, which for us will be a $d$-dimensional rectangular
box in $\mathbb{R}^d$: ${\OOO}=[x_L,x_R]^d$. 
We note that truncating $\Rd$ to a domain like $\OOO$ is effectively placing some natural condition on the type of
densities $\mu$ and $\nu$ we envision having, namely they need to decay sufficiently fast outside of the
box $\OOO$ (\cite{Givoli}).  Then, we propose the spatial discretization of \eqref{GeodEqn1}, by 
following the theory of OT problem on a finite graph similarly to what we did in \cite{CDZ20}.

Next, we let $G=(V,E)$ be a uniform lattice graph with equal spatial step-size $\delta x=\frac {x_R-x_L} n$ in each dimension. 
Here $V$ is the vertex set with $N=(n+1)^d$ nodes labeled by multi-index $i=(i_k)_{k=1}^d\in V, i_k\le n+1.$ 
$E$ is the edge set: $ij \in E$ if $j\in N(i)$ (read, $j$ is a neighbor of $i$), where 
\begin{align*}
N(i)=\cup_{k=1}^d N_k(i), \quad
N_k(i)=\Big\{(i_1,\cdots,i_{k-1},j_k,i_{k+1},\cdots,i_d)\big| |i_k-j_k
|=1 \Big\}.
\end{align*} 
A vector field $v$ on $E$  is a skew-symmetric matrix. The inner product of two vector fields  $u,v$ is defined by 
$$\<u,v\>_{\theta(\rho)}:=\frac 12\sum_{(j,l)\in E}u_{jl}v_{jl}\theta_{jl}(\rho),$$ 
where $\theta$ is a weight function depending on the probability density. 
In this study, we select it as the average of density on 
neighboring points, i.e., 
\begin{equation}\label{thetaij}
\theta_{ij}(\rho):=\frac {\rho_i+\rho_j} 2, \quad \text{if}\quad j\in N(i).
\end{equation}
For more choices, 
we refer to \cite{CDZ20} and references therein. 

The discrete divergence of the flux function $\rho v$  is defined as 
$$div_G^{\theta}(\rho v):=-(\sum_{l\in N(j)} \frac 1{\delta x^2}v_{jl}\theta_{jl}).$$
Using the discrete divergence and inner product, a discrete version of the
Benamou-Brenier formula is introduced in \cite{CDLZ19},
\begin{align*}
W^2(\mu,\nu)=\inf_{v}\Big\{\int_{0}^1\<v,v\>_{\theta(\rho)}dt \,\ : \,
\frac{d\rho}{dt}+div_G^{\theta}(\rho v)=0, \; \rho(0)=\mu,\; \rho(1)=\nu\Big\}.
\end{align*} 
By the Hodge decomposition on graph, it is proved that the optimal vector field $v$ can be expressed as the gradient 
of potential function $S$ defined on the node set $V$, i.e. $v=\nabla_G S:=(S_j-S_l)_{(j,l)\in E}$, $\rho_t$-a.s. 
Similarly, its critical point satisfies the discrete Wasserstein Hamiltonian flow (cfr. with \eqref{GeodEqn1})
\begin{equation}\label{dhs}\begin{split}
&\frac {d\rho_i}{d t}=\sum_{j\in N(i)}\frac 1{(\delta x)^2}(S_i-S_j)\theta_{ij}(\rho)=\frac {\partial \mathcal H}{\partial S_i},\\
&\frac {d S_i}{dt}=-\frac 12\sum_{j\in N(i)}\frac 1{(\delta x)^2} (S_i-S_j)^2 \frac {\partial \theta_{ij}(\rho)}{\partial \rho_i}=-\frac {\partial \mathcal H}{\partial \rho_i}+C(t)
\end{split}\end{equation}
with boundary values $\rho(0)=
\mu$ and $\rho(1)=\nu.$
Here the discrete Hamiltonian is 
$$\mathcal H(\rho,S)=\frac 14\sum_{i=1}^N\sum_{j\in N(i)}\frac {|S_i-S_j|^2}{(\delta x)^2}\theta_{ij}(\rho).$$
We observe that \eqref{dhs} is a semi-discrete version of the Wasserstein Hamiltonian flow,
preserving the Hamiltonian and symplectic structure of the original system \eqref{GeodEqn1}. 
Likewise, the Wasserstein distance
$W(\mu,\nu)$ can be approximated by $\sqrt{2\mathcal H(\mu,S^0)}$, 
where $S^0$ is the initial condition of the spatially discrete $S$. 
 Finally, define the density set by 
$$\mathcal P(G)=\Big\{\rho=(\rho_i)_{i\in V}\Big| \sum_{i\in V}\rho_i (\delta x)^d =1, \rho_i\ge 0, i\in V \Big\},$$
where $\rho_i$ represents the density on node $i$. The interior of $\mathcal P(G)$ is denoted by $\mathcal P_{o}(G).$ 
 
In this study, \eqref{dhs} is the underlying spatial discretization for our numerical method (but see 
\eqref{rhov} below), in large part because of the following result which gives 
some important properties of \eqref{dhs}, and whose proof is in \cite[Proposition 2.1]{CDZ20}.
\begin{prop}\label{well-dhs}
Consider \eqref{dhs} with initial values $\mu$ and $S^0$ and let $T^*$ be the first time where the system develops a singularity.
Then, for any $\mu\in \mathcal P_o(G)$ and any
function $S^0$ on $V$, there exists a unique solution of \eqref{dhs} for all $t<T^*$,  
and it satisfies the following properties for all $t<T^*$.
\begin{enumerate}[label=(\roman*)]
\item Mass is conserved: 
$$\sum_{i=1}^N\rho_i(t) = \sum_{i=1}^N\mu_i^0.$$
\item Energy is conserved: 
$$\mathcal H(\rho(t),S(t))=\mathcal H(\mu,S^0).$$
\item Symplectic structure is preserved: 
$$d\rho(t)\wedge dS(t) =d\mu \wedge dS^0.$$
\item The solution is time reversible:  if $(\rho(t), S(t))$ is the solution of \eqref{dhs}, then 
$(\rho(-t), -S(-t))$ also solves it. 
\item A time invariant $\widetilde \rho\in \mathcal P_o(G)$ and $\widetilde S(t)=-v t$ form an interior stationary solution of \eqref{dhs} 
if and only if $\mathcal H(\rho,S)$ is spatially independent (we denote it as $\mathcal H(\rho) $ in this case), 
$\widetilde \rho$ is the critical point of $\min\limits_{\rho\in \mathcal P_o(G)}\mathcal H( \rho)$ and 
$v=\mathcal H(\widetilde \rho)$. 
\end{enumerate}
\qed
\end{prop}

\section{Algorithm}\label{MSsection}
In this section, we first present the ideas of shooting methods, then combine them with a continuation strategy  
to design our algorithm for approximating the solution of the OT problem \eqref{min-e}.

\subsection{Single shooting}
To illustrate the single shooting strategy,
consider \eqref{dhs} in the time interval $[0,1]$.  Assuming that it exists,
denote with $\rho(t,S^0),$ $t\in [0,1]$, the solution of \eqref{dhs} with initial values $(\mu, S^0)$. 
To satisfy the boundary value at $t=1$, one needs to  
find $S^0$ such that the trajectory starting at $(\mu, S^0)$ passes through $\nu$ at $t=1$, i.e.,
\begin{equation} \label{shooting1}
\rho(1,S^0)-\nu=0.
\end{equation}

To solve \eqref{shooting1}, root-finding algorithms must be used to update the current guess of $S^0$ to achieve better approximations.
For example, when using Newton's method, the updates are supposedly computed by
\begin{align*}
J(1,S^{(i)})\left(S^{(i+1)}-S^{(i)}\right)=-(\rho(1,S^{(i)})-\nu),\; i=0,1,\cdots,
\end{align*}
where $J(t,S)=\frac {\partial \rho(t,S)}{\partial S}$ is the Jacobian of $\rho(t,S)-\nu$ with respect to $S$. 
To ensure successful computations in Newton's method, finding a good initial guess for $S^0$ and having an invertible Jacobi matrix are crucial. 
But, as we anticipated in Remark \ref{SnotUnique},
the Jacobian matrix $J(t,S)$ is singular, as otherwise
a solution of \eqref{shooting1} ought to be isolated, which can't be true, since adding an arbitrary constant will still give a solution. 

To remedy this situation, we reverse to the $(\rho,v)$ formulation, and 
rewrite the Hamiltonian system \eqref{dhs} into an equivalent form in terms of $(\rho,v)$. 
More precisely, by letting $v_{ij} = S_j-S_i$ for $ij \in E$, \eqref{dhs} becomes 
\begin{equation}\begin{split}\label{rhov}
&\frac {d\rho_i}{d t}=-\sum_{j\in N(i)}\frac 1{(\delta x)^2} v_{ij} \theta_{ij}(\rho), \\ 
&\frac {d v_{ij}}{dt}=\frac 12\sum_{k\in N(j)}\frac 1{(\delta x)^2} v_{kj}^2 \frac {\partial \theta_{jk}(\rho)}{\partial \rho_j}-
\frac 12\sum_{k\in N(i)}\frac 1{(\delta x)^2} v_{ki}^2 \frac {\partial \theta_{ik}(\rho)}{\partial \rho_{i}}.
\end{split}\end{equation}
Since $v_{ij}$ is the difference between $S_j$ and $S_i$, a constant shift in $S$ has no impact on the values of $v = \{v_{ij}\}$. 
On the other hand, there are now many redundant equations in \eqref{rhov}, because $\{v_{ij}\}$ are not independent variables. 
For example, they must satisfy $v_{ij} = - v_{ji}$. Furthermore, there are total $N=(n+1)^d$ unknown values for $S$, while 
$2dn(n+1)^{d-1}$ unknowns for $v$ on the lattice graph $G$. Clearly, to determine $S$ up to a constant, only $N-1$ values for 
$v$ are needed. In other words, there must be only $N-1$ independent $v$-equations in \eqref{rhov} to be solved, and
the remaining ones are redundant and must be removed so that the resulting system leads to a non-singular Jacobian.

There are different ways to remove the redundancies. 
To illustrate this in a simple setting, let us consider the 1-dimensional case ($d=1$), in which the lattice graph $G$ 
has $n-1$ interior nodes and $2$ boundary nodes. Each interior node has two neighbors while a boundary node 
has only one neighbor. We have at least two options: either to keep all equations for $v_{i,i+1}$, $i=1,\cdots,(N-1)$, or
to keep the equations for $v_{i, i-1}$, $i=2,\cdots, N$. Adopting the first choice, we have the following equations to solve
\begin{equation}\label{rhov1d}\begin{split}
&\frac {d\rho_i}{d t}=\frac 1{(\delta x)^2} v_{(i-1)i} \theta_{(i-1)i}(\rho) - \frac 1{(\delta x)^2} v_{i(i+1)} \theta_{i(i+1)}(\rho),\\ 
&\frac {d v_{i(i+1)}}{dt}=\frac 14 \frac 1{(\delta x)^2} v_{(i-1)i}^2  - \frac 14\frac 1{(\delta x)^2} v_{i(i+1)}^2,
\end{split}\end{equation}
for all $i=1,\cdots,N-1$.  If we take no-flux boundary conditions for $(\rho,v)$, we have 
$v_{01}=0,\theta_{01}=0$.  Finally, mass conservation gives the condition
$\rho_N=\frac {1-\delta x\sum_{i=1}^{N-1} \rho_i}{\delta x}$. 

Denoting $v(0)=v^0=\{v^0_{i,i+1}\}_{i=1}^{N-1}=\{S_{i+1}^0-S_i^0\}_{i=1}^{N-1}$, and the solution of \eqref{rhov1d} with initial 
values $(\mu,v^0)$ as $\rho_t=\rho(t,v^0)$, $v_t=v(t,v^0)$, we can revise the single shooting strategy in terms of $(\rho,v)$ as finding 
the initial velocity $v^0$ such that $\rho(1,v^0))=\nu$.  By applying Newton's method, we obtain
\begin{align*}
\hat{J}(1,v^{(m)})\bigl(v^{(m+1)}-v^{(m)}\bigr)=-(\rho(1,v^{(m)})-\nu),\; m=0,1,\cdots,
\end{align*}
where $\hat{J}(1,v^{(m)})=\left[\frac {\partial \rho_t}{\partial v^0}\right]_{1,v^{(m)}}$ is the Jacobian of $\rho(t,v(0))-\nu$ with respect to $v(0)$,
evaluated at $t=1$, $v=v^{(m)}$.  For later reference, and since $\nu$ plays no role in the definition of $\hat J$, let us
define the function
$$\hat{J}(t,v^0)=\left[\frac {\partial \rho}{\partial v^0}\right]_{t,v}, \,\ t\ge 0\ .$$

Now, the single shooting strategy we just outlined is plagued by a common shortfall of single shooting techniques, namely that
the initial guess $v^{(0)}$ must be quite close to the exact solution.  In the present context, this is further exacerbated by the fact
that \eqref{GeodEqn1} may develop singularities in finite time (see e.g. \cite{CDZ20}), and as consequence the
choice of a poor initial guess may (and does) lead to finite time blow-up of the solution of the initial value problem.
To overcome this serious difficulty, we now give a result showing that the function $\hat{J}(t,v^0)$ remains
invertible for sufficiently short times, and later will exploit this result to justify adopting a multiple shooting strategy.

\begin{lm} \label{lemma1}
 Let $G$ be a 1-dimensional uniform lattice graph and 
let $t_1>0$ be sufficiently small. Assume that $(\rho,v)$ is the smooth solution of \eqref{rhov1d} satisfying $\mu>0.$
Then, the function $\hat{J}(t,v^0)$ is invertible for $t\in (0,t_1]$. 
\end{lm}
\begin{proof}
Direct calculation shows that the function $\hat{J}(t,v^0)=\frac {\partial }{\partial v^{0}}\rho(t,v^0)$ satisfies
\begin{align*}
\frac{d}{dt} \frac {\partial \rho_t}{\partial v^{0}}&=B_{11} \frac {\partial v_t}{\partial v^{0}}
+B_{12}  \frac {\partial \rho_t}{\partial v^{0}},\; \,\ \hat{J}(0,v^0)=0_{n\times n},\\
\frac{d}{dt} \frac {\partial v_t}{\partial v^{0}}&=B_{22}\frac {\partial v_t}{\partial v^{0}},\; \,\ 
\left[\frac {\partial v_{t}}{\partial v^{0}}\right]_{t=0}=I,
\end{align*}
where  
\begin{align*}
&(B_{11})_{ii}=-\frac {\rho_i+\rho_{i+1}}{2(\delta x)^2},\; i=1,\cdots,n-1,\\
&(B_{11})_{i,i-1}=\frac {\rho_{i}+\rho_{i-1}} {2(\delta x)^2}, i=2,\cdots, n,\\
&(B_{11})_{nn}=\frac {1-\sum_{i=1}^{n-1}\rho_{i}\delta x}{2(\delta x)^3},\\
&(B_{12})_{11}=-\frac {v_1}{2(\delta x)^2}, (B_{12})_{ii}(\rho,v)=-\frac {v_i}{2(\delta x)^2}+\frac {v_{i-1}}{2(\delta x)^2}, i=2,\cdots,n,\\
&(B_{12})_{i,i-1}=\frac {v_{i-1}}{2(\delta x)^2},
(B_{12})_{i,i+1}=-\frac {v_{i}}{2(\delta x)^2}, i=2,\cdots, n-1,\\
&(B_{12})_{n,i}=\frac {v_n}{2(\delta x)^2}, i=1,\cdots,n-2,\; (B_{12})_{n,n-1}=\frac {v_n}{2(\delta x)^2}+\frac {v_{n-1}}{2(\delta x)^2},\\
&(B_{22})_{i,i+1}=-\frac 1{2(\delta x)^2} v_{i+1}, i=1,\cdots, n-1,\;
(B_{22})_{i,i-1}=\frac 1{2(\delta x)^2} v_{i-1},i=2,\cdots, n.
\end{align*}
Since $B_{11}$ is a lower triangular matrix, it is invertible if and only if 
$$\min_{i\le n}(\theta_{i,i+1}(\rho))>0,$$ 
where $\theta_{ij}$ is defined in \eqref{thetaij} and hence $\theta_{i,i+1}(\rho)>0$ for as long as $\rho$ remains
positive.  Moreover, given the initial condition to the identity for
$\frac {\partial v_{t}}{\partial v^{0}}$, if 
$t_1>0$ is sufficiently small  the matrix $\frac {\partial v_t}{\partial v^{0}}$ remains invertible.
Furthermore,  since $\hat{J}(0,v^0)=0_{n\times n}$, 
we conclude that for $t>0$ sufficiently small 
\begin{equation*}
\hat{J}(t,v^0)
\approx t B_{11}+ \mathscr O(t^2),
\end{equation*}
which implies that $\hat{J}(t,v^0)$  is invertible for $t>0$, and sufficiently small.
\end{proof}

%

Once $v$ values become available, if desired we can reconstruct $S$ on the lattice graph $G$ from the relation $v_{ij} = S_i - S_j$.

We conclude this section by
emphasizing that the semi-discretization \eqref{rhov} is a 
spatial discretization of the Wasserstein geodesic equations written in term of $(\rho,v)$ \cite{CDZ20}. 
However, this semi-discretization has been arrived at by designing a semi-discretization scheme for the system \eqref{GeodEqn1}
in the $(\rho,S)$ variables, respecting the Hamiltonian nature of the problem, see \eqref{dhs} and Proposition \ref{well-dhs}.

\subsection{Multiple shooting method} 
As proved in Lemma \ref{lemma1}, in the 1-d case
the function $\hat{J}(t,v^0)$ is invertible for sufficiently short times; however, for the success of single
shooting, this ought to be invertible at $t=1$, a fact which is often violated.
In addition, our numerical experiments indicate poor stability behavior when using the single shooting
method to solve the Wasserstein geodesic equations \eqref{dhs}.
To mitigate these drawbacks, we propose to use multiple shooting.

We partition the interval $[0,1]$ into the union of sub-intervals $[t_k,t_{k+1}], k=0,\cdots,K-1$,
and let $\delta t=\max_{k}(t_{k+1}-t_k)$.
For example, we could take $t_k=k\delta t$ and $K\delta t=1$. 
To illustrate, 
we again take $G$ as the d-dimensional uniform lattice graph. 
In each subinterval $[t_k,t_{k+1}], k=0,\cdots,K-1$, \eqref{dhs} is converted into equations in terms of $(\rho,v)$, 
just like the ones in \eqref{rhov}, 
\begin{align*}
\frac {d\rho_i^{k+1}}{d t}&
=-\sum_{j\in N(i)}\frac 1{(\delta x)^2}v_{ij}^{k+1}\theta_{ij}(\rho),\\\nonumber
\frac {d v_{ij}^{k+1}}{dt}
&=\frac 12\sum_{l\in N(j)}\frac 1{(\delta x)^2} (v_{jl}^{k+1})^2 
\frac {\partial \theta_{lj}(\rho)}{\partial \rho_j}-\frac 12\sum_{m\in N(i)}\frac 1{(\delta x)^2} (v_{mi}^{k+1})^2 \frac {\partial \theta_{ik}(\rho)}{\partial \rho_{i}},
\end{align*}
where $i\in N$ is a multi-index for a grid point in d-dimensional lattice. 
The super script $k+1$ in $\rho$ and $v$ indicates that the corresponding variables are defined in the subinterval $[t_k, t_{k+1}]$. 
Then, the multiple shooting method requires finding the values of $\rho,v$ at temporal points $\{t_k\}_{k=0}^{K-1}$, i.e., 
$$(\widetilde v^0,\widetilde \rho^1,\widetilde v^1,\cdots,\widetilde \rho^{K-1},\widetilde v^{K-1})^{T},$$ 
such that the continuity conditions hold, that is, for $k = 0, \cdots, K-2,$
\begin{align*}
F_{2k+1}(\widetilde \rho^k,\widetilde v^k,\widetilde \rho^{k+1}) 
&=\rho^{k+1}(t_{k+1},\widetilde \rho^k, \widetilde v^k)-\widetilde \rho^{k+1}=0, \;\\
F_{2k+2}(\widetilde \rho^k,\widetilde v^k,\widetilde v^{k+1})
&=v^{k+1}(t_{k+1},\widetilde \rho^k, \widetilde v^k)-\widetilde v^{k+1}=0.
\end{align*}
When $k=0$ and $k=K-1,$ the given boundary values $\rho(0)=\mu$ and $\rho(1)=\nu$ yield that 
\begin{align*}
&F_1(\mu,\widetilde v^0,\widetilde \rho^{1})=\rho^{1}(t_1, \mu, \widetilde v^0)-\widetilde \rho^{1}=0,\\
&F_{2K-1}(\widetilde \rho^{K-1},\widetilde v^{K-1},\nu)=\rho^{K}(t_{K},\widetilde \rho^{K-1}, \widetilde v^{K-1})-\nu=0.
\end{align*}
As customary, 
we use Newton's method to find the root 
$(\widetilde v^0,\widetilde \rho^1,\widetilde v^1,\cdots,\widetilde \rho^{K-1},\widetilde v^{K-1})$ of  
$F=(F_{w})_{w=1}^{2K-1}=0.$ 
To this end, we first need to remove the redundant equations for the velocity field $v$. 
The number of unknown variables in $\rho$ is $N-1=(n+1)^d-1$, which is  one fewer than the total number of nodes in $G$, 
because the total probability must be one. The number of unknowns in $S$ is $N$. 
The vector field $v$ contains the differences in $S$, hence the total number of independent variables in $v$ is also $N-1$, 
due to the connectivity of $G$. The following lemma ensures that we can always find the $N-1$ components of $v$ from which 
one can generate all the components of $v$ on the lattice graph $G$.

\begin{lm}\label{lm-con-gen}
Given a connected $d$-dimensional lattice graph $G$ and a vector field $v$ which is generated by a potential $S$ on $G$, 
there exists a subset consisting of $N-1$ components of $v$, denoted by $\widehat v=(\widehat v_w)_{w=1}^{N-1}$,
such that any $v_{ij}$ can be expressed as combination of the entries of $\widehat v$, i.e.
\begin{align}\label{con-vij}
v_{ij}=\sum_{w=1}^{N-1} a_{w} \widehat v_{w},\quad \text{where} \quad a_{w}=1, \,\ \text{or} \,\  -1, \,\ \text{or}\,\ 0\ .
\end{align}
\end{lm}

\begin{proof}
Since $G$ is connected, there is always a path on the graph passing through all the nodes of $G$ and with
exactly $N-1$ edges.   We denote with $\widehat v_i$ the value of $v$ on the $i$-th edge along the path. 
By definition of $v_{ij}=S_{j}-S_i$, the values of $S$ can be reconstructed, up to a constant shift, along the path. 
Therefore, all entries of $v$ can be expressed as the above combination of the entries $(\widehat v_w)_{w=1}^{N-1}$.
\end{proof}

From the proof, we observe that the choice of $\widehat v$ is not unique, since
every path going through all nodes of $G$ using $N-1$ edges will give a system with no redundancy. The edges could be passed multiple times. 
Let us select one such choice and denote it by $(\widehat v_w)_{w=1}^{N-1}$.
For instance, in 2-dimensional lattice graph $G$, we choose the $\widehat v$ that generates the vector field (see Fig. \ref{generator}) as follows. 
Denote every node on $G$ by $(i,j)_{i,j=1}^{n+1}$. 
For fixed $i$, $(i,j)_{j=1}^{n+1}$ becomes 1-dimensional lattice graph in the $x_2$ direction. 
Following \eqref{rhov1d}, we choose $\widehat v_{w}=v_{(i,j)(i,j+1)}$ for   $w=n\times (i-1)+j,$ 
$j=1,\cdots,n, i=1, \cdots, n+1$, which gives $(n+1)\times n$ components of $\widehat v_w.$
Because of the connectivity of $G$ relative to the $x_1$ direction,  the last $n$ components  of 
$\widehat v_w$ are chosen by $\widehat v_{w}=v_{(j,1)(j+1,1)},$ for  $w=(n+1)\times n+j,$ $j=1,\cdots,n.$ For convenience, let us denote the velocity on the related edges in this path by 
$\{v_{i_wi_{w+1}}\}_{w=1}^{N-1}=\{\widehat v_{w}\}_{w=1}^{N-1}.$

\begin{figure}
	\centering
		\includegraphics[width=0.9\linewidth]{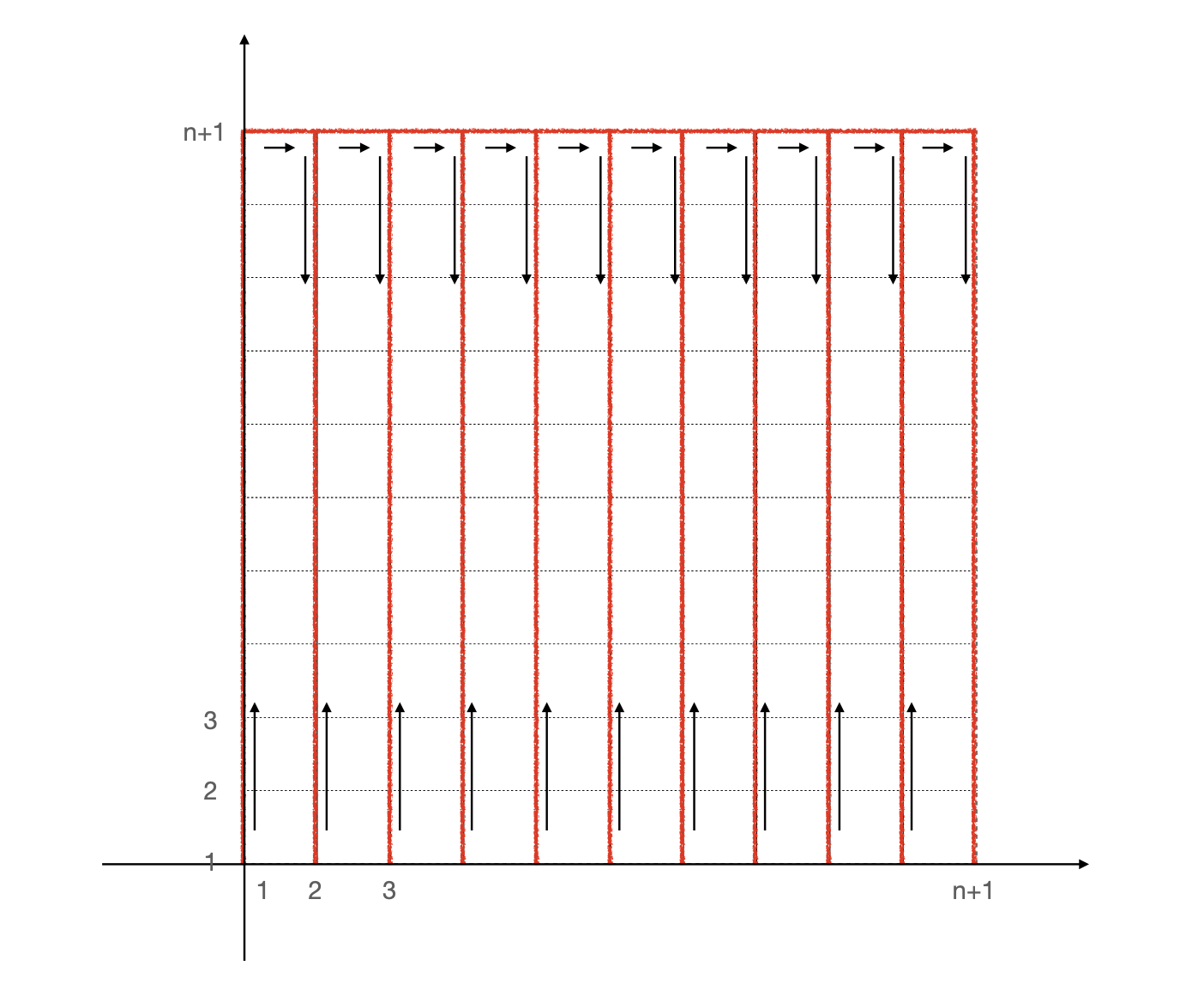}
	\caption{The edges (in red) of $\hat{v}$ that generates of the velocity in 2D lattice graph. The path is indicated by the arrows. Clearly, many edges are passed twice.}
	\label{generator}
\end{figure}

Then the reduced Wasserstein system \eqref{dhs} becomes
\begin{equation}\label{re-dhs}\begin{split}
\frac {d\rho_{i_w}^{k+1}}{d t} &= \sum_{j\in N(i_w)} v_{ji_w}^{k+1}\theta_{i_wj}(\rho),\\ 
\frac {d \widehat v_{i_w}^{k+1}}{dt}& =\frac 12\sum_{j\in N(i_w)}\frac 1{(\delta x)^2} (v_{i_w,j}^{k+1})^2 \frac {\partial \theta_{i_wj}(\rho)}{\partial \rho_{i_w}}\\
& \, -\frac 12\sum_{m\in N(i_{w+1})}\frac 1{(\delta x)^2} (v_{i_{w+1,m}}^{k+1})^2 \frac {\partial \theta_{i_{w+1}j}(\rho)}{\partial \rho_{i_{w+1}}},
\end{split}\end{equation}
where $v_{ij}$ satisfies \eqref{con-vij} and the unknowns are $(\rho,\widehat v)$ with 
\begin{align*}
&\rho^{k+1}(t_k,\rho(t_k),\widehat v(t_k))=\rho(t_k), \; \rho^{k+1}(t_{k+1},\rho(t_k),\widehat v(t_k))=\rho(t_{k+1}), \\
&\widehat v^{k+1}(t_k,\rho(t_k),\widehat v(t_k))=\widehat v(t_k), \; 
\widehat v^{k+1}(t_{k+1},\rho(t_k),\widehat v(t_k))=\widehat v(t_{k+1}).
\end{align*}

We apply the multiple shooting method to \eqref{re-dhs}, i.e., we look for the root 
$Z=(\widehat v^0,  \rho^1,\widehat v^1,\cdots, \rho^{K-1},\widehat v^{K-1})$ of $F$ defined by 
\begin{equation}\label{mul-shoot0}\begin{split}
F_{2k+1}(\rho^k,\widehat v^k, \rho^{k+1})
&=\rho^{k+1}(t_{k+1}, \rho^k, \widehat v^k)- \rho^{k+1}=0, \;\\ 
F_{2k+2}( \rho^k,\widehat v^k, \widehat v^{k+1})
&=\widehat v^{k+1}(t_{k+1}, \rho^k, \widehat v^k)-\widehat v^{k+1}=0, \; k\le K-2, \\ 
F_{2K-1}( \rho^{K-1},\widehat v^{K-1}, \rho^K)&=\rho^{K}(t_{K-1}, \rho^{K-1}, \widehat v^{K-1})-\nu=0,
\end{split}\end{equation}
where $\rho^0=\mu, \rho^K=\nu.$

Use of Newton's method to solve \eqref{mul-shoot0} gives
\begin{align}\label{mul-shoot}
A^{(m)}\Delta Z^{(m)}=-F^{(m)},
\end{align}
where $m$ is the iteration index, $\Delta Z^{(m)}=Z^{(m+1)}-Z^{(m)},$ 
$$Z^{(m)}=(v^{0,(m)},\rho^{1,(m)},v^{1,(m)},\cdots, v^{K-1,(m)},\rho^{K-1,(m)})^T,$$ 
$F^{(m)}=(F_1(Z^{(m)}),F_2(Z^{(m)}),\cdots, F_{2K-1}(Z^{(m)}))^T,$ and 
$A^{(m)}$ is the Jacobian of $F$, whose structure is as follows, where the $X$ correspond
to nonzero $(N-1) \times (N-1)$ matrices:
$$
\begin{pmatrix}
X & X & 0 & 0 & 0 &  &  &  &  &  \\
X & 0 & X & 0 & 0 &  &  &  &  &  \\
0 & X & X & X & 0 &   &  &  &  &  \\
0 & X & X & 0 & X &   &  &  &  &  \\
 &  &  & X & X & X  & 0 &  &  &  \\
 &  &  & X & X & 0  & X &  &  &  \\
 & & & & & & & \ddots & \ddots & \\
 &  &  & & & & X & X & X & 0  \\
 &  &  & & & & X & X & 0 & X  \\
 &  &  & & & & & & X & X  
\end{pmatrix}\ .
$$
Omitting the  superscript $m$ in the expressions of $A^{(m)}$, the blocks $A_{ij}, i,j=1,\cdots, 2K-1,$ are
easily seen to be the following.  For $i=2,\cdots, K-1,$
\begin{align*}
A_{2(i-1)+1,2(i-1)}&= \frac {\partial \rho^{i}(t_i,v^{i-1},\rho^{i-1})}{\partial v^{i-1}},\;  
A_{2(i-1)+1,2(i-1)+1}=\frac {\partial \rho^{i}(t_i,v^{i-1},\rho^{i-1})}{\partial \rho^{i-1}},\\
A_{2i,2(i-1)}&=\frac {\partial v^{i}(t_i,v^{i-1},\rho^{i-1})}{\partial v^{i-1}},\;  
A_{2i,2(i-1)+1}= \frac {\partial v^{i}(t_i,\rho^{i-1},\rho^{i-1})}{\partial \rho^{i-1}},\\
A_{2(i-1)+1,2i} &=-I, A_{2i,2i+1}=-I,\\
A_{11}&=\frac {\partial \rho^{1}(t_1,v^{0})}{\partial v^0}, A_{12}=-I,\\
A_{21}&=\frac {\partial v^{1}(t_1,v^{0})}{\partial v^0}, A_{23}=-I,
\end{align*}
and 
\begin{align*}
A_{2K-1,2K-2}&=\frac {\partial \rho^{K}(t_{K},v^{K-1},\rho^{K-1})}{\partial v^{K-1}}, 
A_{2K-1,2K-1}=\frac {\partial \rho^{K}(t_K,v^{K-1},\rho^{K-1})}{\partial \rho^{K-1}}.
\end{align*}
Below we show invertibility of $A^{(m)}$ for $\delta t$ sufficiently small.

\begin{tm}\label{MSconv}
Let $(\rho,v)$ be the unique solution of \eqref{rhov} and $Z^*=(v(0),\rho(t_1),$ $v(t_1),\cdots,\rho(t_{K-1}),v(t_{K-1}))^{T}$ be the exact solution
evaluated at the multiple shooting points.
Assume that the initial vector $Z^{(0)}$ is sufficiently close to $Z^*$, i.e., 
$|Z^{((0)}-Z^*|={\mathcal O}(\epsilon)$ for $\epsilon>0$ sufficiently small, 
$(\rho,v)$ is continuously differentiable in $[0,1]$ satisfying $(\rho,v)\in \mathcal C^2_b([0,1];\mathbb R^N)\times \mathcal C^2_b([0,1];\mathbb R^N\times \mathbb R^N)$ and $\min\limits_{t\in [0,T]}\min\limits_{i=1}^N \rho_i\ge c>0$, and that  
$\frac{\partial \rho(1,\rho^{0},v^{0})}{\partial v^{0}}$ is invertible.
Then, Newton's method of the multiple shooting method \eqref{mul-shoot} is quadratically convergent to $Z^*$ for $\delta t$ sufficiently small.
\end{tm}

\begin{proof}
By standard Newton's convergence theory, it will be enough to prove the      
invertibility of Jacobian matrix $A^{(0)}$ for appropriately small  
$\epsilon$ and $\delta t$.  
Rewrite $A^{(0)}$ in partitioned form 
$\begin{pmatrix}
A_{11}' & A_{12}'\\
O_{N-1,N-1} & A_{22}'
\end{pmatrix},$
where $A_{11}'$ is a $(2K-2)n \times n$ matrix, $A_{12}'$ is a $(2K-2)n\times (2K-2)n$ matrix, 
and $A_{22}'$ is  a $(N-1)\times (2K-2)(N-1)$ matrix. 
Using the property of determinant for the partitioned matrix and the fact that $\det(A_{12}')=1$, 
and writing $A$ in lieu of $A^{(0)}$, we have 
\begin{align*}
\det(A)&=\det \begin{pmatrix}
0_{N-1\times N-1} & A_{22}'\\
A_{11}' & A_{12}'
\end{pmatrix} \\
&=\det(A_{12}')\det(0_{N-1\times N-1}-A_{22}' (A_{12}')^{-1}A_{11}')\\
&=(-1)^{N-1} \det(A_{22}' (A_{12}')^{-1}A_{11}').
\end{align*}
So, we are left to show that $\det(A_{22}' (A_{12}')^{-1}A_{11}')\neq 0$. 
The structure of  $A_{12}'$  implies that 
\begin{equation}\label{inv}\begin{split}
A_{22}' (A_{12}')^{-1}A_{11}' & =(\frac {\partial \rho^{K,(0)}}{\partial \rho^{K-1,(0)}},\frac {\partial \rho^{K,(0)}}{\partial v^{K-1,(0)}}) \\
 & \prod_{i=2}^{K-1} \begin{pmatrix}
\frac {\partial \rho^{i,(0)}}{\partial \rho^{i-1,(0)}} & \frac {\partial \rho^{i,(0)}}{\partial v^{i-1,(0)}}  \\\nonumber
\frac {\partial v^{i,(0)}}{\partial \rho^{i-1,(0)}} & \frac {\partial v^{i,(0)}}{\partial v^{i-1,(0)}}
\end{pmatrix}
(\frac {\partial \rho^{1,(0)}}{\partial v^{0,(0)}},\frac {\partial v^{1,(0)}}{\partial v^{0,(0)}})^{T},
\end{split}\end{equation}
where $\rho^{i,(0)}=\rho^{i}(t_K,\rho^{i-1,(0)},v^{i-1,(0)})$, 
$v^{i,(0)}=v^{k}(t_K,\rho^{i-1,(0)},v^{i-1,(0)}),$ for $i=2,\cdots,K$, and 
$v^{1,(0)}=v^1(t_1,v^{0,(0)}),$ $\rho^{1,(0)}=\rho^1(t_1,v^{0,(0)})$.

Now, 
invertibility of the Jacobian matrix $A$ (or $A_{22}' (A_{12}')^{-1}A_{11}'$) follows from invertibility of the Jacobian
matrix at the exact solution $\frac{\partial \rho(t_K,\rho^{0},v^{0})}{\partial v^{0}}.$  To see this,
due to \eqref{inv}, the continuous differentiability of the exact solution, and the assumption that $|Z^{(0,(m))}-Z^*|=\mathscr O(\epsilon)$, 
we have that 
\begin{align*}
A_{22}' (A_{12}')^{-1}A_{11}'
=\frac{\partial \rho(t_K,\rho^{0},v^{0})}{\partial v^{0}}+\mathscr O(\epsilon)+\mathscr O(\delta t).
\end{align*}
Therefore, the  invertibility of $\frac{\partial \rho(t_K,\rho^{0},v^{0})}{\partial v^{0}}$ with $t_K=1$
implies the invertibility of the Jacobian matrix $A$. 
Combining with the assumption that $\epsilon$ and $\delta$ are sufficiently small, we obtain that $A^{(0)}$
is invertible in a neighborhood of $Z^*$, which, together with the boundedness assumption on $\rho,v$, implies 
the quadratic convergence of Newton's method. 
\end{proof}

\begin{rk}
Of course, the initial value problems for the multiple shooting method must be integrated numerically.  We have not accounted for this
in Theorem \ref{MSconv}.  In principle, many choices are available to integrate these initial value problems; we have
used the symplectic integrators  developed in \cite{CDZ20} for Wasserstein Hamiltonian flows,  without regularization 
by Fisher information.
\end{rk}
%
%

\subsection{Continuation multiple shooting strategy}\label{cmss}
In light of Theorem \ref{MSconv}, and notwithstanding the need for small $\delta t$, the
multiple shooting method requires the initial guess to be near the exact solution $Z^*$.
To make the method robust with respect to the initial guess, we adopt a standard
continuation strategy by introducing a density function $f(\mu,\nu, \lambda)$, 
which is smooth with respect to a homotopy parameter $\lambda \in [0,1]$ and satisfies
\begin{align}\label{homotopy}
f(\mu,\nu,0)=\mu,\quad f(\mu,\nu,1)=\nu.
\end{align}
The specific choice of $f$ in \eqref{homotopy}
depends on the initial and terminal distributions $\mu$ and $\nu$.   We illustrate below with
two typical situations.
\begin{itemize}
\item[(a)]\label{(a)}  
``Gaussian-type'' densities.
If $\mu(x)=K_0\exp(-c|x-b_0|^2)$ and $\nu(x)=K_1\exp(-c|x-b_1|^2)$, with 
$\int_{\mathcal O}\mu dx=\int_{\mathcal O}\nu dx=1$, we choose 
$$f(\mu,\nu,\lambda)(x)=K_{\lambda}\exp(-c|x-b_0-\lambda(b_1-b_0)|^2)\ $$ 
with $K_{\lambda}$ chosen so that $\int_{\mathcal O}f dx=1$. 
For $\mu=K_0\exp(-c_0|x-b_0|^2),\nu=K_1\exp(-c_1|x-b_1|^2)$, we choose
$$f(\mu,\nu,\lambda)(x)=K_{\lambda}\exp(-(c_0+\lambda(c_1-c_0))|x-b_0-\lambda(b_1-b_0)|^2)\ $$ 
with $K_{\lambda}$ chosen so that $\int_{\mathcal O}f dx=1$. 

\item[(b)]\label{(b)}
For general $\mu$ and $\nu$, we choose $f$ as the linear interpolant of $\mu$ and $\nu$, which is
automatically normalized.  That is, we take 
$$
f(\mu, \nu, \lambda) = (1-\lambda) \mu + \lambda \nu.
$$
\end{itemize}

\begin{rk}
For the success of our method, it is actually important that the densities be strictly positive (see Theorem \ref{MSconv}).
For this reason, and especially when the densities $\mu$ and $\nu$ are exponentially decaying (like Gaussians do),
we add a small positive number, which we call {\emph{shift}}, 
to the densities $\mu$ and $\nu$ and re-scale them so to keep the total probabilities equal to $1$.    
In the numerical tests in Section \ref{NumExs}, these are the values $r_0$ and $r_1$ we use.
\end{rk}

Using $f$, we consider the system  \eqref{re-dhs} 
with $\lambda$ dependent boundary conditions given by $\rho(0) = \mu$ and $\rho(1) = f(\mu, \nu,\lambda)$. 
Obviously, the problem with $\lambda_0=0$ is trivial to solve (the identity map), and it can be used
as initial guess for the solution at the value $\lambda_1=\Delta \lambda$.  
By gradually increasing $\lambda$ from $0$ to $1$, we eventually obtain the solution for \eqref{dhs} with boundary conditions 
$\mu$ and $\nu$, which is the original Wasserstein geodesic problem we wanted to solve.
This basic idea to use the solution with smaller value of $\lambda$
as the initial guess for the boundary value problem with larger value of $\lambda$ is well understood, and universal. 
In our context, it is important to note that it works because of  OT problem always has an optimal map as long as $\mu$ and 
$f(\mu,\nu,\lambda)$ satisfy $\int_{\mathbb R^d}|x|^2\mu dx,\int_{\mathbb R^d}|x|^2 f(\mu,\nu,\lambda) dx<+\infty$
(e.g., see \cite{San15}). In turns, this implies the existence of $v$ or $S$ 
(up to $\rho_t$-measure $0$ sets) for the BVP problem.
In particular, this fact guarantees that there is 
a finite sequence $\{\lambda_j\}_{j\le L}$,  $\lambda_L=1,$  and $Z_{\lambda_L}^*$ will be our 
approximation to the exact solution $(\rho,v)$ at the multiple shooting points.
\begin{align}\label{ini-guess}
Z_{\lambda_0}^0:=(v^{0,(0)},\rho^{1,(0)},\cdots, v^{K-1,(0)},\rho^{K-1,(0)})^{\mathcal T}.
\end{align}
For instance, we may take $v^{k,(0)},k\le K-1$, as constant vectors, $\rho^{k,(0)}, k\le K-1,$ from linear interpolation 
of $\rho^0=\mu$ and $\rho^1= f(\mu,\nu,\lambda_0),$ i.e.,
\begin{align*}
\rho^{k,(0)}=t_k\mu+(1-t_k) f(\mu,\nu,\lambda_0), k\le K-1.
\end{align*}

Finally, throughout all of our experiments,  we 
enforced the following stopping criterion for the Newton iteration:
\begin{equation}\label{NewtStops}
\frac {|F(Z^{(m+1)})-F(Z^{(m)})|}{F(Z^{(m)})}< 10^{-5}\ .
\end{equation}

We summarize the steps in the following algorithm.

\begin{algorithm}[htb] 
\caption{} 
\label{alg:Framwork} 
\begin{algorithmic}[1] 
\REQUIRE 
Multiple shooting points $t_k$, $k=0,\dots, K$, with $t_0=0$ and $t_K=1$.
Discrete densities $\mu$, $\nu$, on the spatial grid of size
$\delta x,$  continuation parameter $\lambda$, max-number of Newton's iterations
{\tt Maxits}.
\ENSURE 
The minimizer $Z^{*}$ at the multiple shooting points;
\STATE Follow \eqref{ini-guess} and produce a initial guess $Z^{(0)}_{\lambda_0}$; 
\STATE Until $\lambda_j=1$ or too many continuation steps, {\bf do}
\FOR{$m=1,2,\cdots, {\tt Maxits}$, while \eqref{NewtStops} not satisfied }
\STATE   Solve $J^{(m)}_{\lambda_j} d^{(m)}=- F(Z^{(m)}_{\lambda_j})$;\
\STATE $Z^{(m+1)}_{\lambda_j}=Z^{(m)}_{\lambda_j}+ d^{(m)}$;\
\ENDFOR

\STATE $\lambda_{j+1}=\lambda_j+\Delta \lambda$ (see Remark \ref{ContSteps}); \\
\STATE put $Z^{0}_{\lambda_{j+1}}=Z^*_{\lambda_j}$ as the new initial guess; 
\STATE $j+1 \to j$,  go back to step 2.
\end{algorithmic}
\end{algorithm}

\begin{rk}
Based on the output of Algorithm \ref{alg:Framwork}, the Wasserstein distance (or the Hamiltonian of \eqref{dhs}) can be easily obtained. 
From the first component $v^{0,*}$ of $Z^*$, we can reconstruct the initial values for $S^0$ as follows.
The first component $v^{0,*}=(\widehat v_w)_{w=1}^{N-1}$, $\{i_{w}i_{w+1}\}_{w=1}^{N-1}$ generates the initial vector field. We first     
define the potential $S$ on a fixed node $i_0$.  Due to the connectivity of $G$, using $S_{i_{w+1}}=v_{i_w,i_{w+1}}+S_{i_w},$ we get the other initial values of $S^0$.
Then the Wasserstein distance can be evaluated as $W(\mu,\nu)=\sqrt{2 H(\mu,S^0)}$.
\end{rk}

\begin{rk}[Barrier value for density]
On rare occasions, we observed that during the Newton's iteration the updates became negative, leading to a failure.
To avoid this phenomenon, we adopted a simple strategy, whereby we created a barrier for
the values of the densities, and reset to this barrier any value which went below it.   In our tests
in Section \ref{NumExs}, use of this artifical barrier was needed only for Examples \ref{Ex5} and \ref{Ex10}. 
To witness, in Example \ref{Ex5}, we used the barrier at $10^{-5}$, and in Example \ref{Ex10}
the barrier was set at $10^{-3}$.  Clearly with this strategy the total mass of the numerical solution is not exactly
equal to $1$, but the error incurred in the total mass is at the same level of the barrier value.
\end{rk}

\begin{rk}[Choosing continuation steps]\label{ContSteps}
We implemented a very simple and conservative continuation strategy.
In all of our tests, we first try to take $\lambda=1$, to see whether the continuation 
is really needed.   If the method does not work without continuation,
we begin with a value $\lambda_0$ of $\lambda$ for which  multiple shooting
works (e.g., we
usually take $\lambda_0=0.1$ as initial step), and choose a value 
$\Delta \lambda=\frac {1-\lambda_0}{L}$ with given $L$ (e.g., $L=10$ or $20$ is our
usual choice).  We then try to continue by taking steps of size $\Delta \lambda$,
though if the Newton's multiple shooting fails we decrease $\Delta \lambda$
by dividing the remaining interval by $L$ again and/or increase the value of $L$ by
doubling it.
In all tests of Section \ref{NumExs}, except  Examples \ref{Ex1} and \ref{Exnew0}, the 
continuation strategy was needed. 
\end{rk}

\begin{rk}[Choosing homotopy $f(\mu,\nu,\lambda)$]\label{HomoFtn}
Finally, for all tests
with Gaussian type densities $\mu,\nu$, we use the Gaussian interpolation (a) in subsection \ref{cmss} for $f(\mu,\nu,\lambda).$ 
For other examples, we use the linear interpolation (b) in subsection \ref{cmss} for $f(\mu,\nu,\lambda).$
To exemplify, in Example \ref{Ex5},   we take $f(\mu,\nu,\lambda)$ as the normalization of 
$\exp(-5(x_2-0.5-1.95\lambda)^2-5(x_1-1.5-0.95\lambda)^2)+\exp(-5(x_2-0.5-1.95\lambda)-5(x_1-1.5+0.95\lambda))^2+r$ and
obtain a sequence of $\lambda$'s starting from $\lambda_0=0.1,$ with $\Delta \lambda=0.9/20$.
\end{rk}

\section{Numerical experiments}\label{NumExs}
In this section, we apply Algorithm \ref{alg:Framwork} to approximate the solution of several OT problems. 
Throughout the experiments, 
the Jacobian in Newton's method is approximated by using a 1st order divided difference approximation of the derivatives. 
The spatial boundary conditions for the density functions are set to be homogeneous Neumann  boundary conditions for all
experiments except for Example \ref{Ex1}, which is subject to periodic boundary conditions.  
Except for this Example \ref{Ex1}, we do not have the exact solutions of our test problems, so
we display the evolution of the density from $\mu$ to $\nu$ as indication of the quality of the approximation.

\begin{example}\label{Ex1}
Here the 
spatial domain is the 2-torus $\mathbb  T^2=[0,1]\times[0,1],$ subject to periodic boundary conditions.
Following the approach in \cite{San15}, we 
define a smooth function $\phi(x_1,x_2)=\beta \sin(2\pi x_1)\sin(2\pi x_2)$,
with $\beta=\frac{1}{64}(2\pi)^{-2},$ take 
initial density $\mu(x_1,x_2)=\det(I-D^2\phi(x_1,x_2))$ and target density $\nu$ is the uniform distribution on $\mathbb T^2$. 
In this case, the exact initial velocity can be explicitly given:
$$v^0(x_1,x_2)=2\pi \beta (\cos(2\pi x_1)\sin(2\pi x_2),  \sin(2\pi x_1)\cos(2\pi x_2)),$$
and in Table 1 we measure the approximation error of our method, with respect to the spatial grid-size.
As it turns out, this was a very easy problem to solve, and single shooting with a
quasi-Newton approach (only one Jacobian matrix was computed and factored and then used across all
iterations) solved it adequately.  There was 
no need of adopting a continuation strategy, and we took 160 integration steps from $0$ to $1$.
About 90\% of the computation time was spent on calculating the Jacobian 
at the initial guess.
From Table 1, we observe 1st order convergence with respect to both 
$L^2$ and sup norms, i.e., $\|\widehat v^0-v^0\|_{l^{\infty}}, \|\widehat v^0-v^0\|_{L^2}$, where $\widehat v$ is the initial function
on the grids solved by single shooting method, and $l^{\infty},L^2$ denote the discrete sup norm and $L^2$ norm respectively.
This is in agreement with the semi-discretization scheme we used. 
		
\begin{table}
\begin{tabular}{|c|c|c|c|c|} 
\hline 
dx & Maximum Error & $L^2$-Error & Iterations \\
\hline  
1/16 &  0.00120 & 0.00068 & 4 \\
1/32  & 0.00057 & 0.00034 & 5  \\
1/64 & 0.00003 &  0.00017 & 6  \\
1/128 & 0.000019 & 0.000086 & 11  \\
\hline 
\end{tabular}
\caption{The error in the velocity for Example \ref{Ex1}}
\end{table}

\end{example}

\subsection{1D numerical experiments} 
Below we present results on 1-D OT problems, with one or both densities of Gaussian types.
Namely, the initial
and terminal distributions $\mu$ and $\nu$ are normalizations of 
\begin{equation}\label{1d-example}
\widehat \mu=\exp(-a_0(x-b_0)^2)+r_0,\;
\widehat \nu=\exp(-a_1(x-b_1)^2)+r_1,
\end{equation}
scaled so that $\int_\OOO \mu dx=\int_\OOO \nu dx=1.$  (Here, $\OOO$ is a subinterval
of the real line.)

\begin{example}\label{Ex2}
Here we look at the performance of the multiple shooting method when varying the 
(truncation of the real line to the) finite interval $\OOO$, and the shift number $r$.
The parameters of initial and terminal distributions $\mu,\nu$ in \eqref{1d-example}
are $a_0=a_1=15,$ $b_0=0.4,b_1=1.4$.
We take $K=60$ multiple shooting points, spatial  
step size $dx=3\times10^{-2}$, $N=300$ time steps per subinterval,
$r_0=r_1=0.0001$ in \eqref{1d-example}, and consider the intervals
$\mathcal O=[0,2]$ or $[-0.5,2.5].$   In Fig. \ref{evo-rhp}, we 
plot the evolution of density.  The top figures refer to $\OOO=[0,2]$ and
show distortion in the density evolution. The bottom row refers to $\OOO=[-0.5,2.5]$
and shows that the computation is more faithful when the truncated domain is large enough. 

\begin{figure}
	\centering
	\subfigure {
		\includegraphics[width=0.6\linewidth]{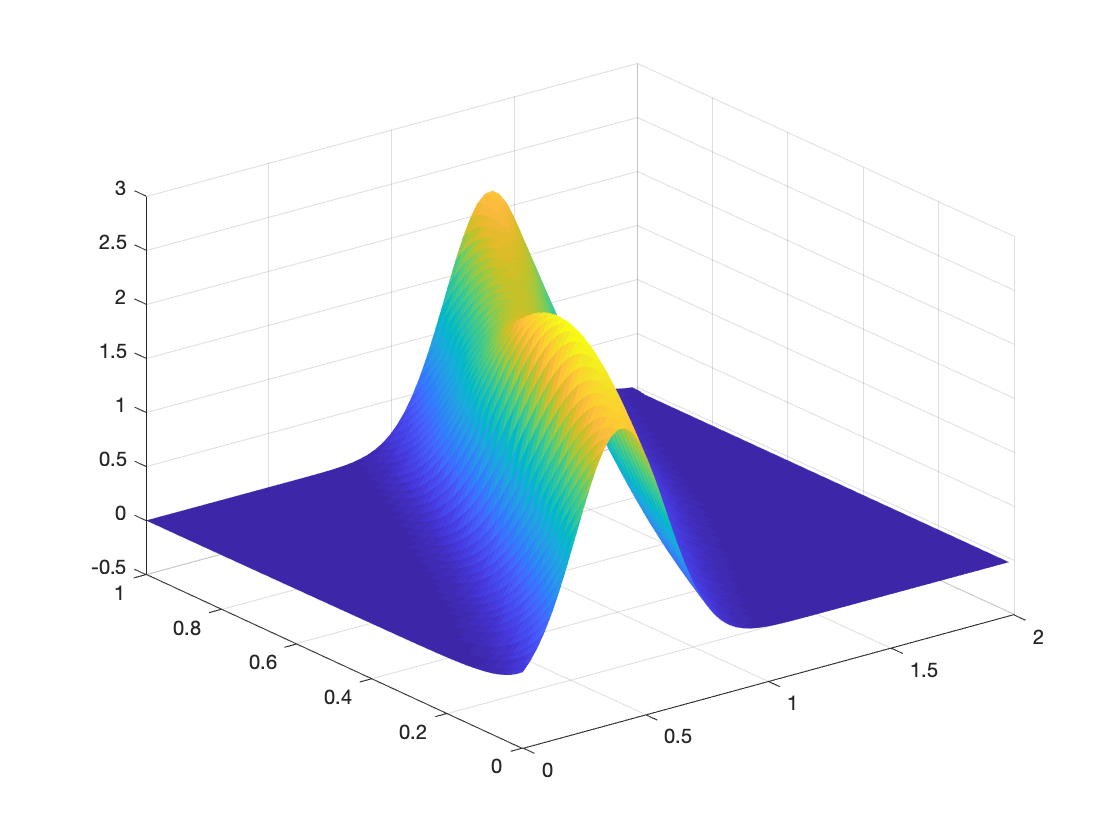}
		\includegraphics[width=0.6\linewidth]{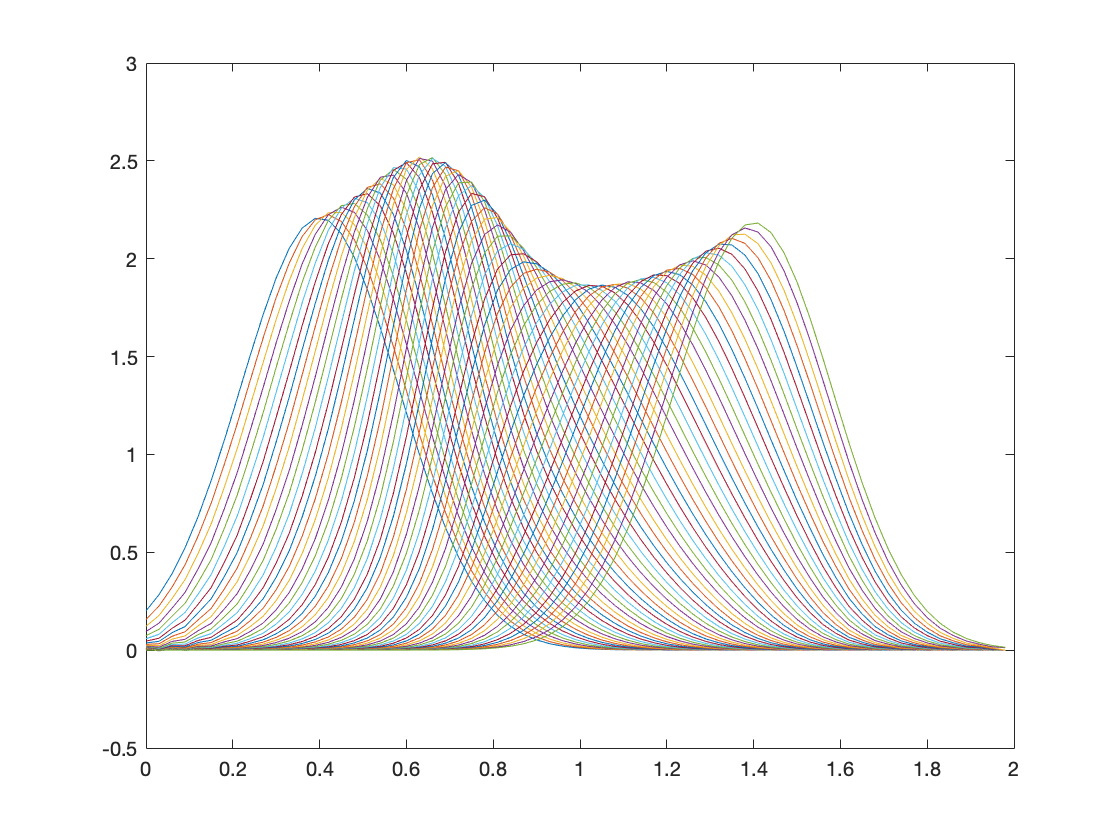}
	}
	\subfigure {
		\includegraphics[width=0.6\linewidth]{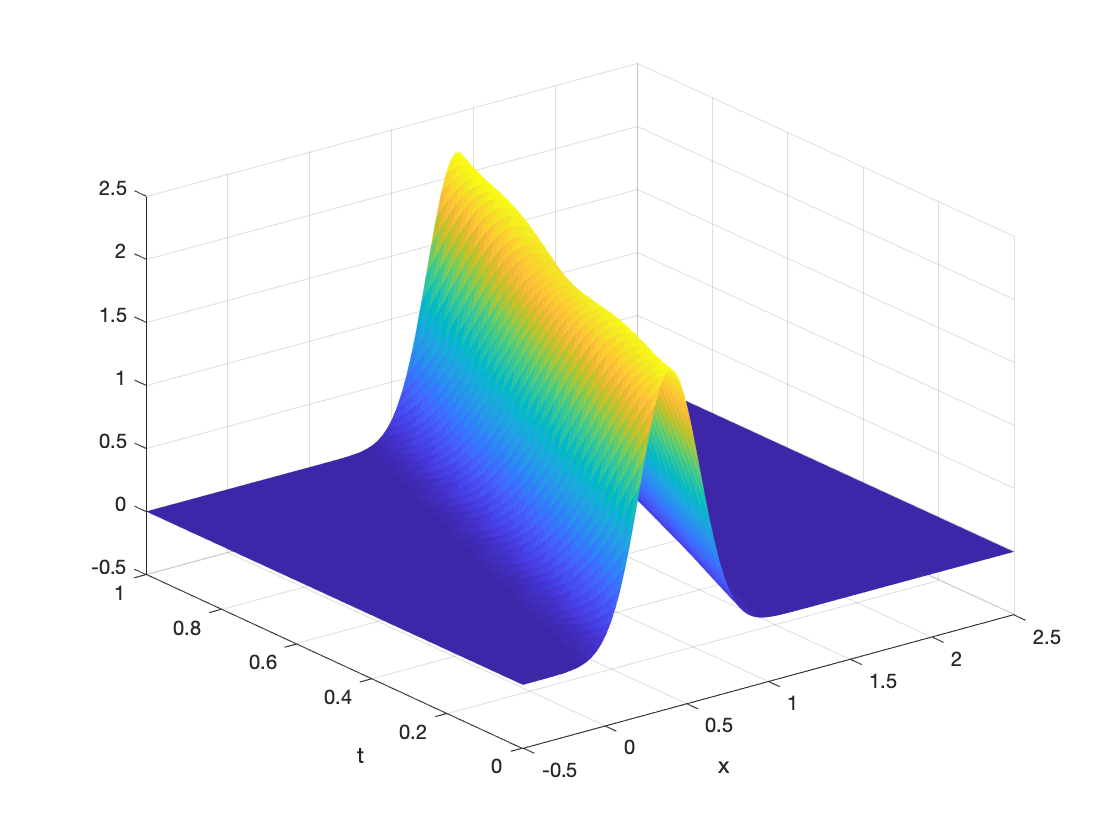}
		\includegraphics[width=0.6\linewidth]{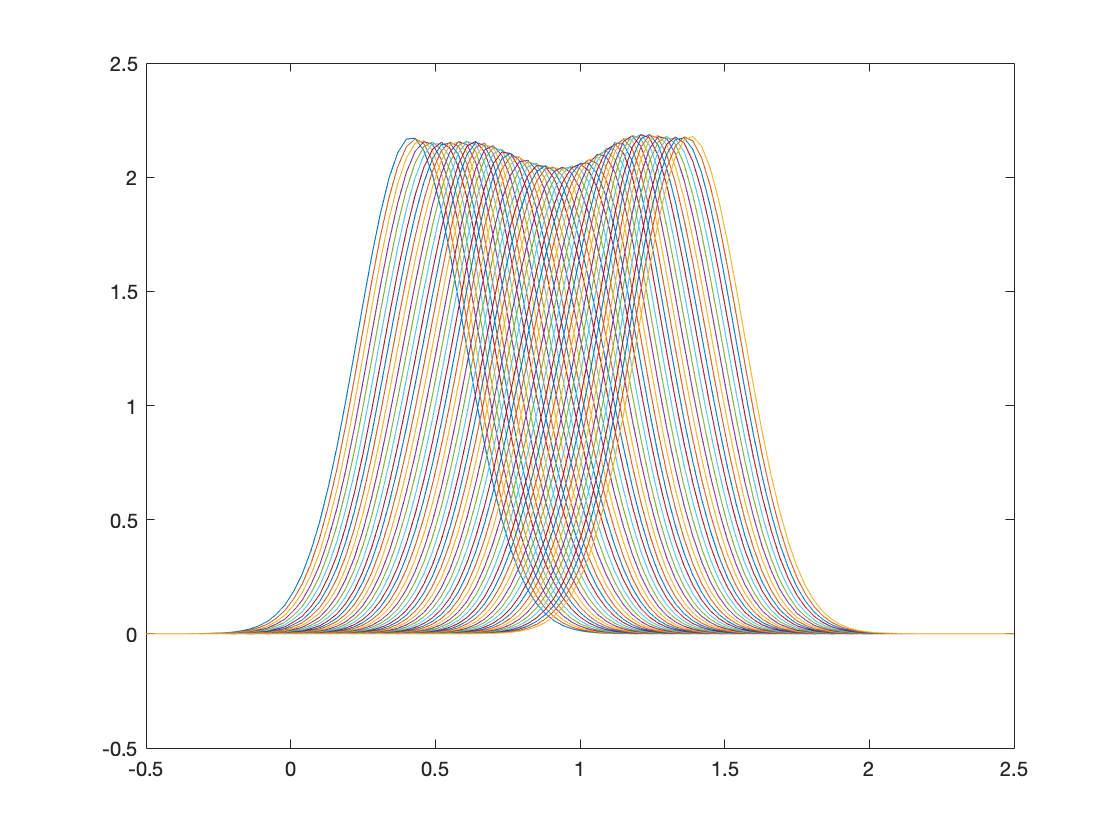}
	}
	\centering 
	\caption{Example \ref{Ex2}: evolution of $\rho(t)$ for truncated interval $[0,2]$ (top) and
	$[-0.5,2.5]$ (bottom).}
	\label{evo-rhp}
	
\end{figure}

\end{example}

\begin{example}\label{Ex3}
Here  $\mathcal O=[0,2]$, 
the initial distribution is the uniform distribution 
$\mu=\frac 12$ and the terminal distribution $\nu$ is the normalized Gaussian 
density as the $\hat{\nu}$ used in Example \ref{Ex2} with $a_1=25, b_1=1,r_1=0.$ 
The number of multiple shooting points is $K=60,$ the space stepsize 
$dx=5 \times10^{-2}$
and  we take $N=20$ integration steps for subinterval. 
Fig. \ref{shootp1central} shows the density evolution.

\begin{figure}
	\centering
	\subfigure {
		\includegraphics[width=0.6\linewidth]{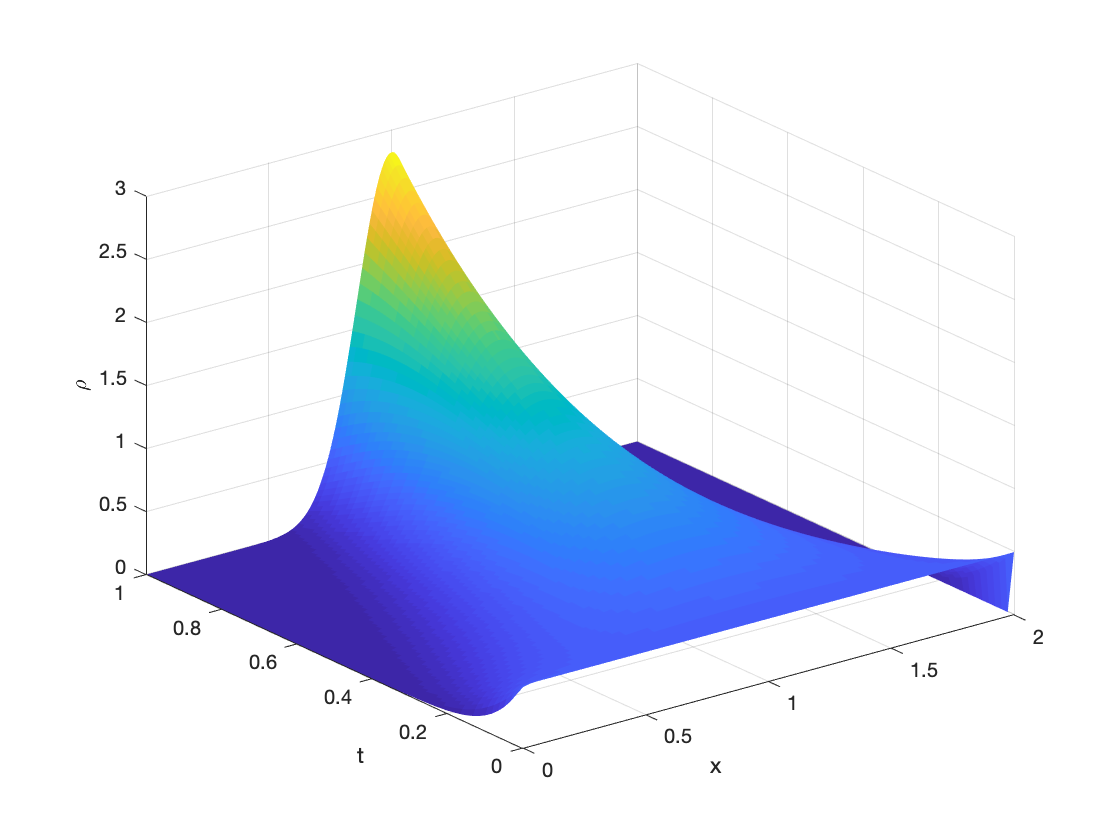}
		\includegraphics[width=0.6\linewidth]{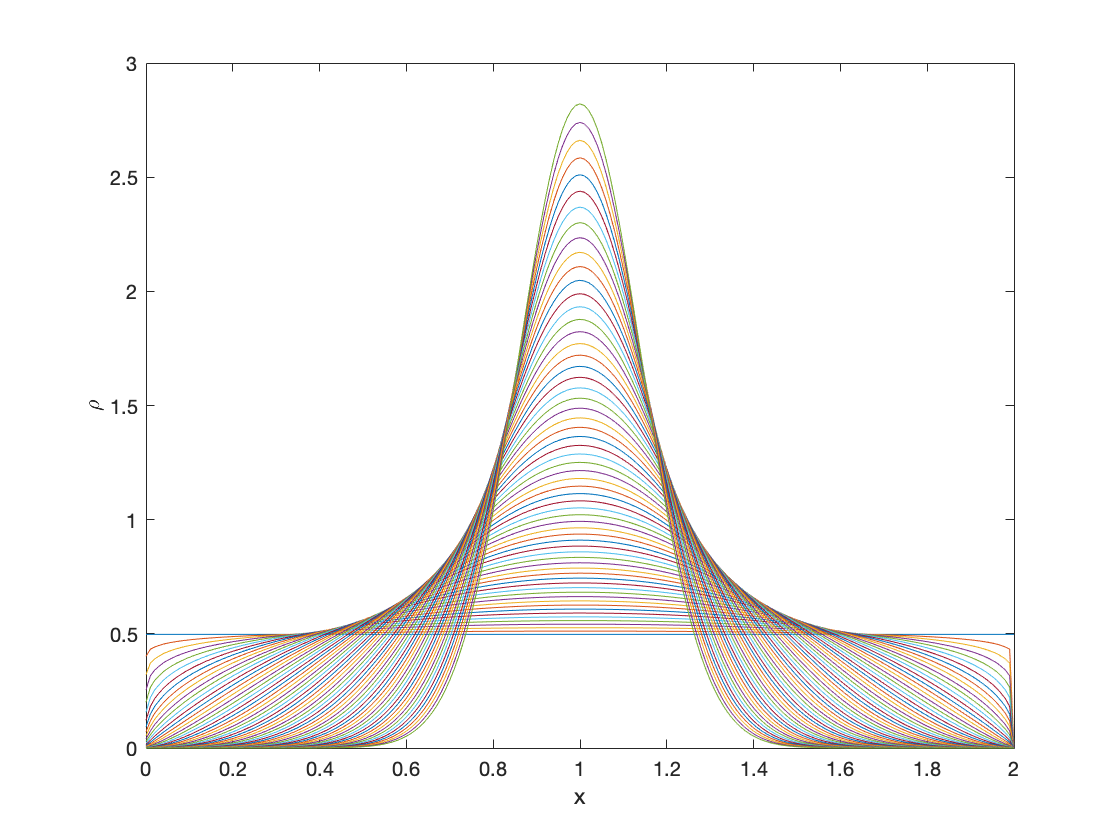}
	}
	
	\centering 
	\caption{the evolution of probability given $\mu$ and $\nu$ in Example 3}
	\label{shootp1central}
\end{figure}

\end{example}

\begin{rk}
In general, we observed that when we refine the spatial step size, the number of multiple shooting subintervals must 
increase in order to maintain non-negativity of the density at the temporal grids,
and a successful completion of our multiple shooting method, whereas the number of integration steps on
each subinterval is not as critical.  See Table 2 for results on Example \ref{Ex3},
which are typical of the general situation.
\end{rk}

\begin{table}
	
	\begin{tabular}{|c|c|c|c|c|} 
		\hline 
		$dx$ & $K$ & $N$ & success \\
		\hline  
		1/16 &  10 & 20 & $\surd$ \\
		1/32 & 10 &  40 & $\surd$  \\
		1/64 & 10 & 80 & $\surd$  \\
		1/128 &  10 & 160 & $\times$\\
		1/128 &  10 & 320 & $\times$\\
		\hline 
	\end{tabular}
	\begin{tabular}{|c|c|c|c|c|} 
		\hline 
		$dx$ & $K$ & $N$ & success \\
		\hline  
		1/16 &  10 & 20 & $\surd$ \\
		1/32  & 20 & 20 & $\surd$  \\
		1/64 & 20 &  20 & $\times$  \\
		1/64 & 40 &  20 & $\surd$  \\
		1/128 & 40 & 20 & $\surd$  \\
		\hline 
	\end{tabular}
	
	\caption{The relationship between $dx$, $K$ and $N$ in Example \ref{Ex3}.}
\end{table}

\begin{example}\label{Ex4}
This is similar to Example \ref{Ex2}, but the Gaussian has a much greater variance.
Let $\mathcal O=[-0.5,2.5]$, $dx=4\times10^{-2},$ $K=80$, $N=200$, and fix
the parameters of initial and terminal Gaussian distributions $\mu,\nu$ in \eqref{1d-example}
are $a_0=a_1=50,$ $b_0=0.4,b_1=1.4,$ $r_0=r_1=0.0001.$ 
The evolution of the density is shown in Fig. \ref{shootevop050}, and the sharper behavior of the
density evolution with respect to Figure \ref{evo-rhp} is apparent.

\begin{figure}
	\centering
	\subfigure {
		\includegraphics[width=0.6\linewidth]{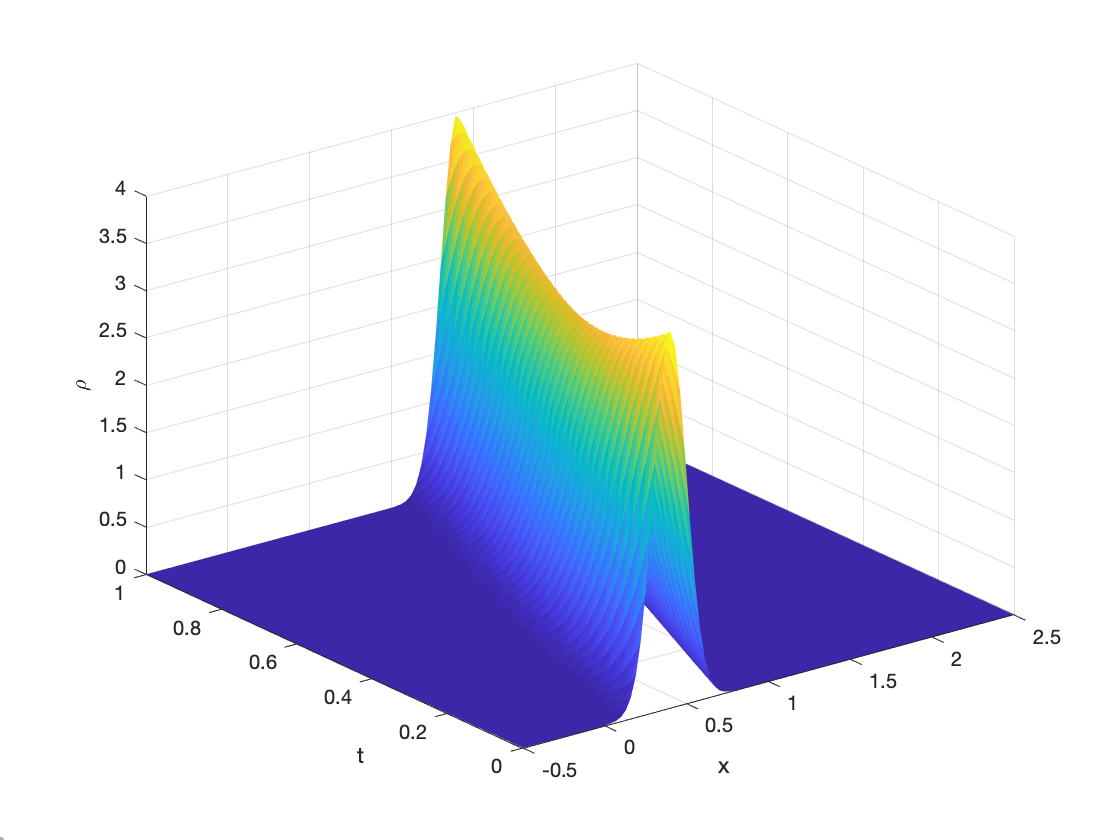}
		\includegraphics[width=0.6\linewidth]{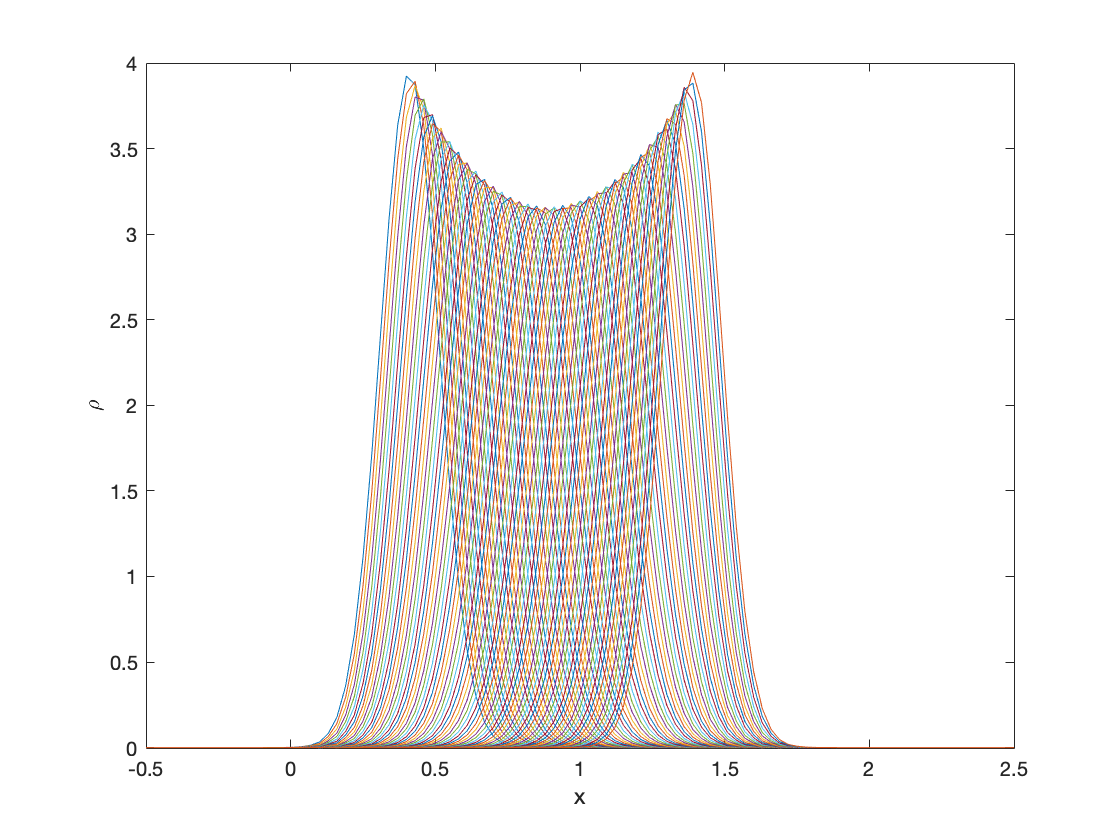}
	}
	\centering 
	\caption{Evolution of probability density in Example \ref{Ex4} }
	\label{shootevop050}
\end{figure}

\end{example}

\begin{example}\label{Exnew0}
This example is used to test Gaussian type distributions $\mu$ and $\nu$ with different variances.
Let $\mathcal O=[-0.5,2.5]$, $dx=4\times10^{-2},$ $K=80$, $N=40$, and let 
the parameters of initial and terminal Gaussian distributions $\mu,\nu$
are $a_0=15, a_1=10,$ $b_0=0.8,b_1=1.6,$ $r_0=r_1=0.0001.$  
The evolution of the density is shown in Figure  \ref{figureEx5}.
In this problem, we also exemplify the impact of the shifting number; as it can be
seen in Figure \ref{figureEx5},  if the shifting number is not sufficiently small
($r_0=r_1=0.01,$ in this case), one ends up with spurious oscillatory behavior
(presently, in $x=[0.4,0.8]$ and $[1.7,2.1]$).

\begin{figure}
	\centering
	\subfigure {
		\includegraphics[width=0.6\linewidth]{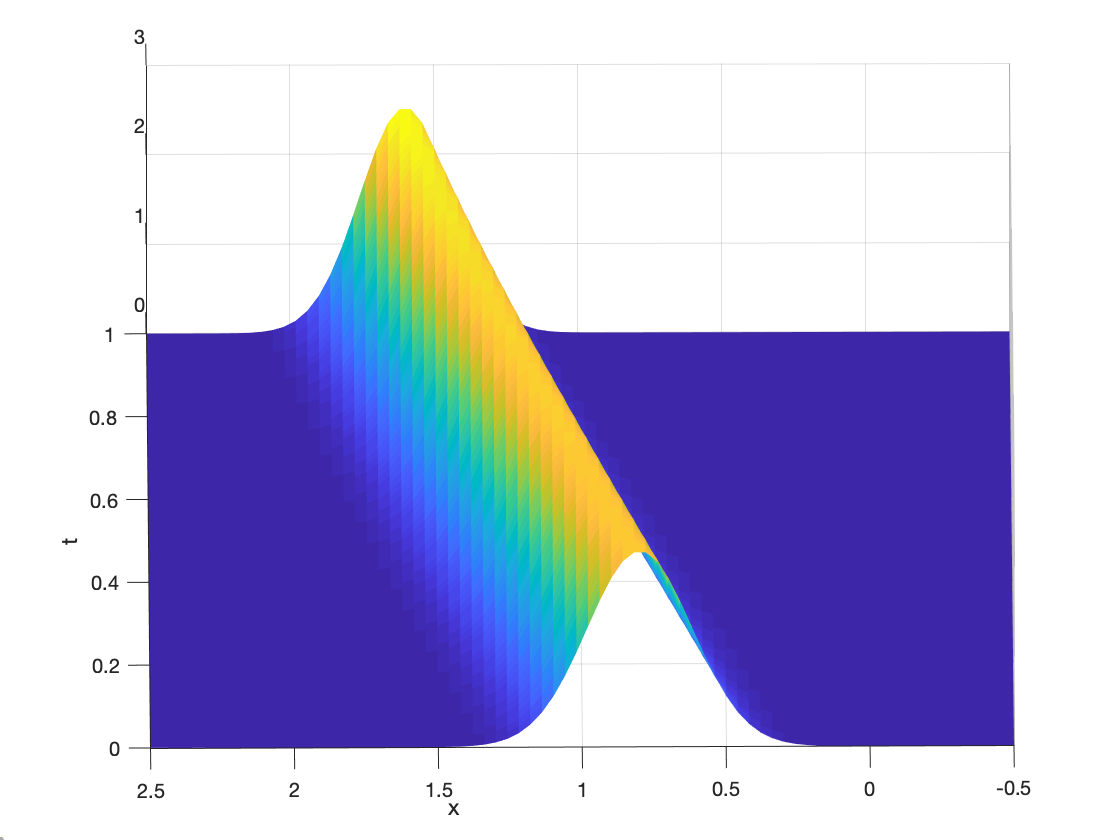}
		\includegraphics[width=0.6\linewidth]{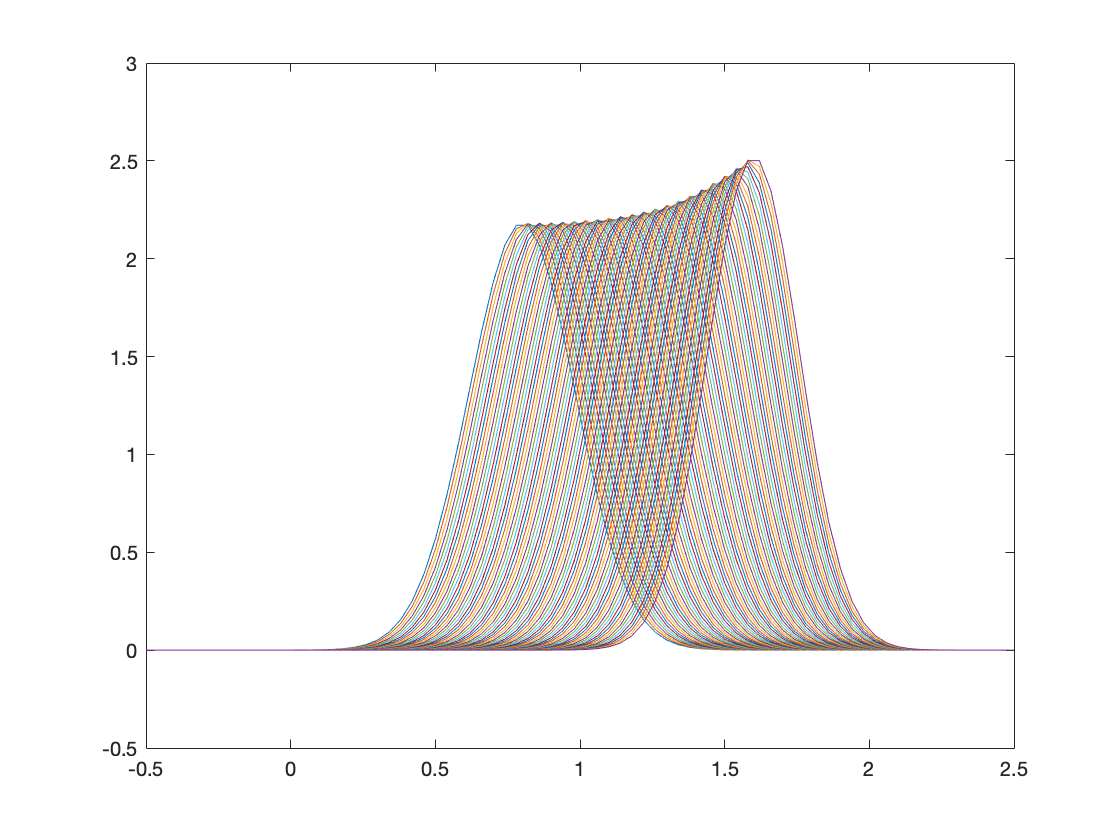}
	}
	\centering 
	\subfigure {
		\includegraphics[width=0.6\linewidth]{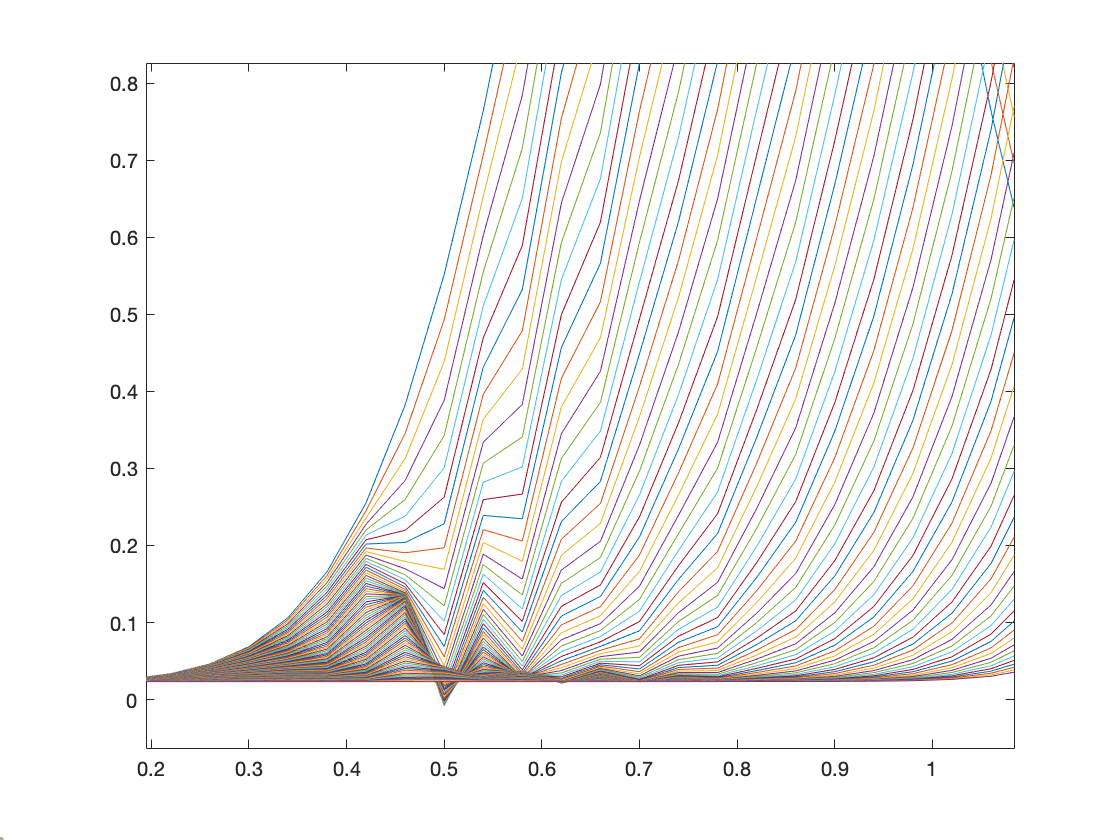}
		\includegraphics[width=0.6\linewidth]{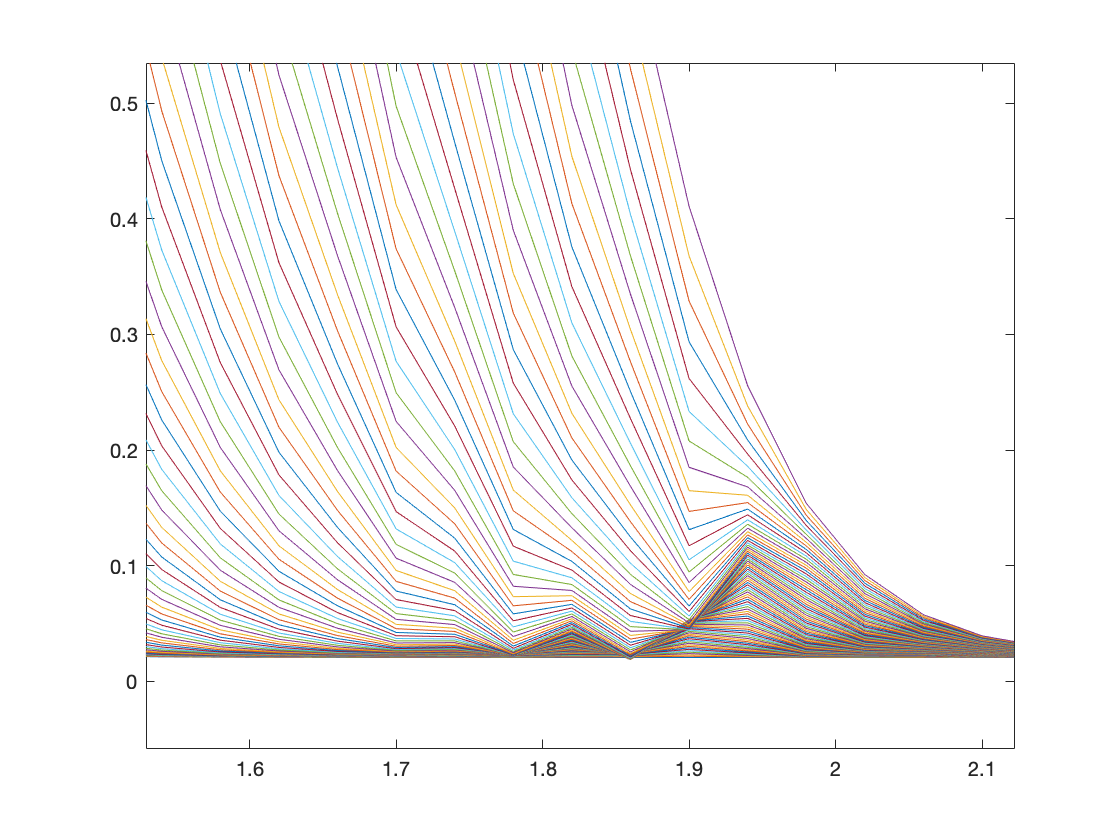}
	}
	\caption{Evolution of probability density in Example \ref{Exnew0} with $r=0.0001$
	(up) 
	oscillator behaviors of probability density when $r=0.01$(down)}
	\label{figureEx5}
\end{figure}

\end{example}

\subsection{2D numerical experiments}
Here, we give computational results for a computational domain $\mathcal O$ which represents a truncation
of $\mathbb R^2.$ 
In Examples \ref{Ex5}-\ref{Ex9}, we always
take $K=10$ multiple shooting subintervals, $\delta x=0.2$ as spatial step size, and $N=30$ 
integration steps on each subinterval $[t_i,t_{i+1}]$, $t_i=i/K, i=0,\cdots K-1$. 

In Examples \ref{Ex5}-\ref{Ex6},  the initial and/or terminal distributions, $\mu,\nu$, are 
normalizations of Gaussian type densities, namely
\begin{equation}\label{2d-example}\begin{split}
\widehat \mu & =\exp(-a_0(x_2-b_0)^2-c_0(x_1-d_0)^2)+r_0,\\
\widehat \nu & =\exp(-a_1(x_2-b_1)^2-c_1(x_1-d_1)^2)+r_1.
\end{split}\end{equation}

\begin{example}\label{Ex5}
Spatial domain is $\mathcal O=[-1,4]\times[-1,4]$.  
Initial density is the normalization of the Gaussian type density $\hat \mu$ in \eqref{2d-example},
with parameters $a_0=5,b_0=0.5,c_0=5,d_0=1.5,r_0=0.01.$
The terminal distribution is the normalization of $\widehat \nu$ below (a two-bump Gaussian)
\begin{align*}
\widehat \nu=\exp(-5(x_2-2.45)^2-5(x_1-2.45)^2)+\exp(-5(x_2-2.45)^2-5(x_1-0.55)^2)+0.01.
\end{align*}
In  Fig. \ref{ncontour-den}, we show the contour plots of the density at different times, from
which the formation of the two bumps is apparent. 
The surfaces of the density at $t=0.8$ and the two components of
initial velocity are shown in Fig. \ref{dent=0.8} and \ref{in-vel}, respectively.

\begin{figure}
\centering
\subfigure {
\begin{minipage}[b]{0.31\linewidth}
\includegraphics[width=1.15\linewidth]{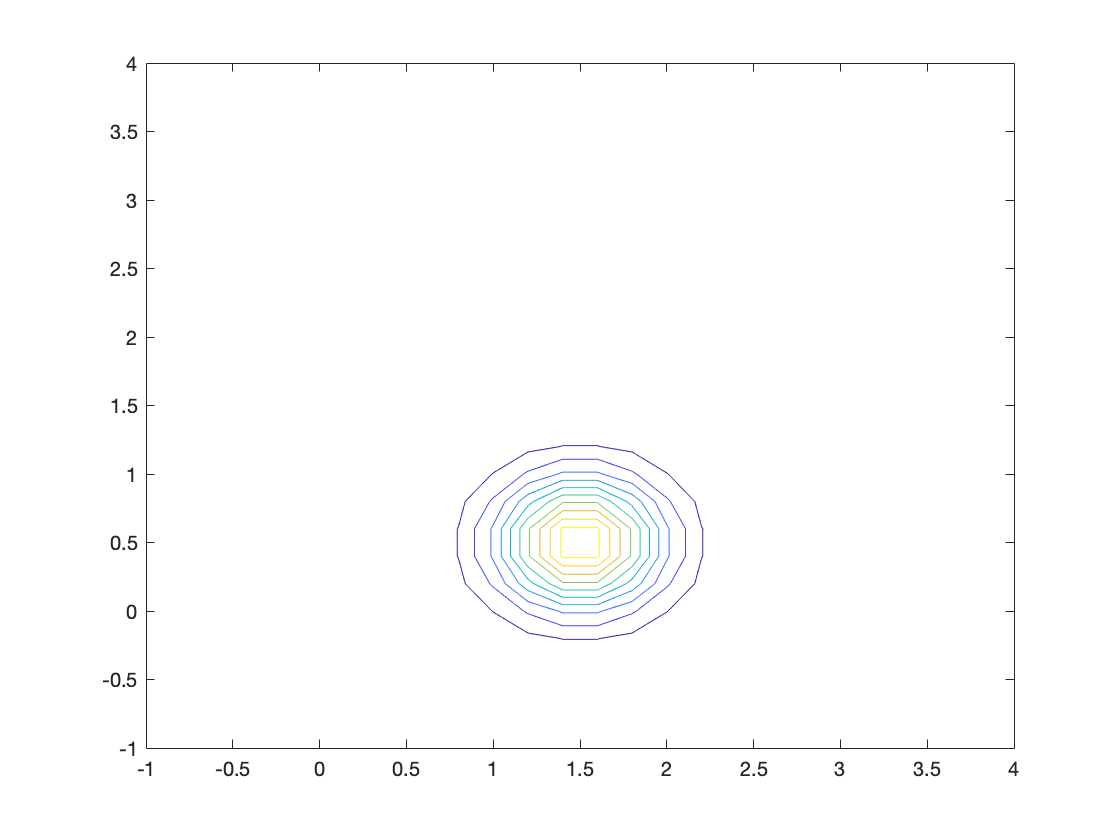}
\includegraphics[width=1.15\linewidth]{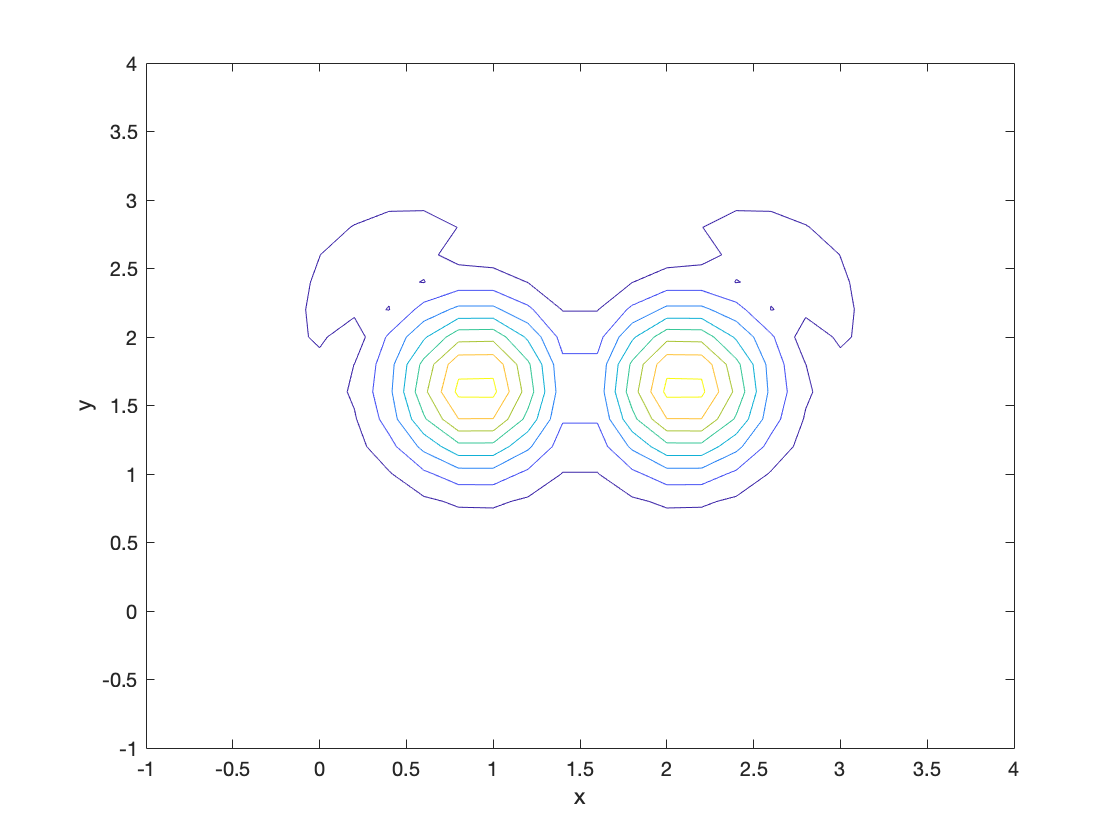}
\end{minipage}}
\subfigure{
\begin{minipage}[b]{0.31\linewidth}
\includegraphics[width=1.15\linewidth]{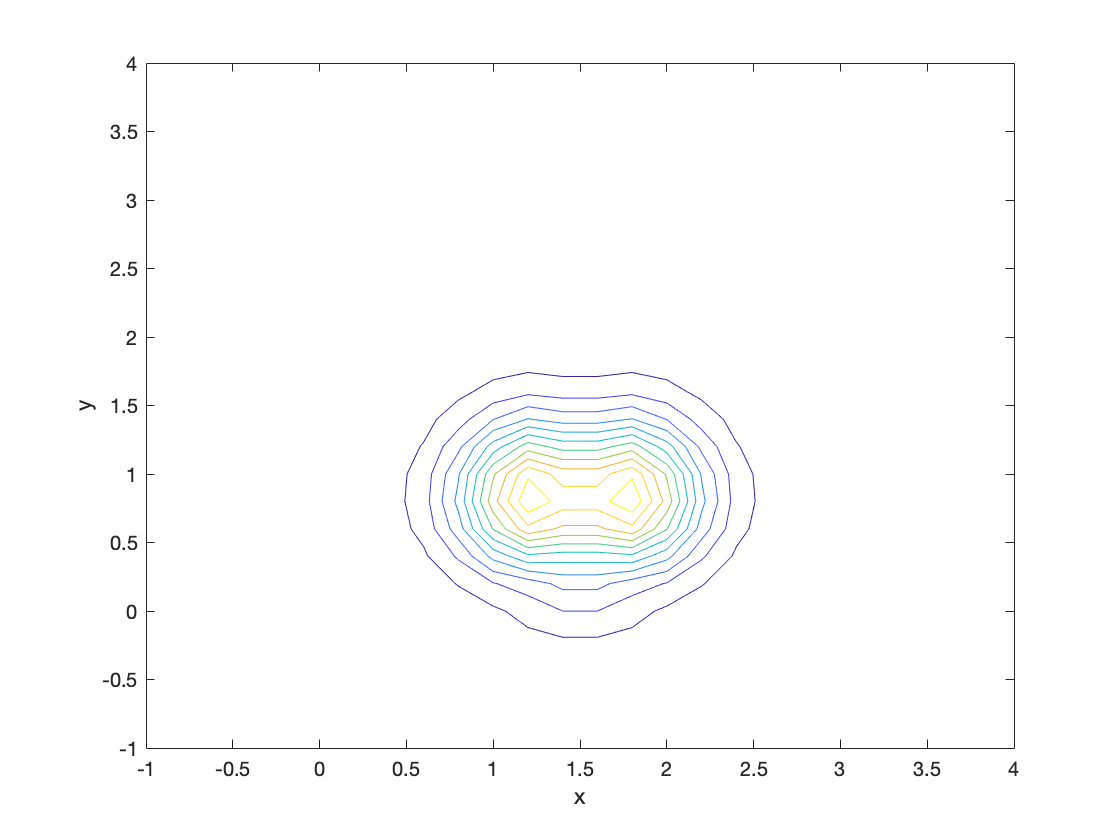}
\includegraphics[width=1.15\linewidth]{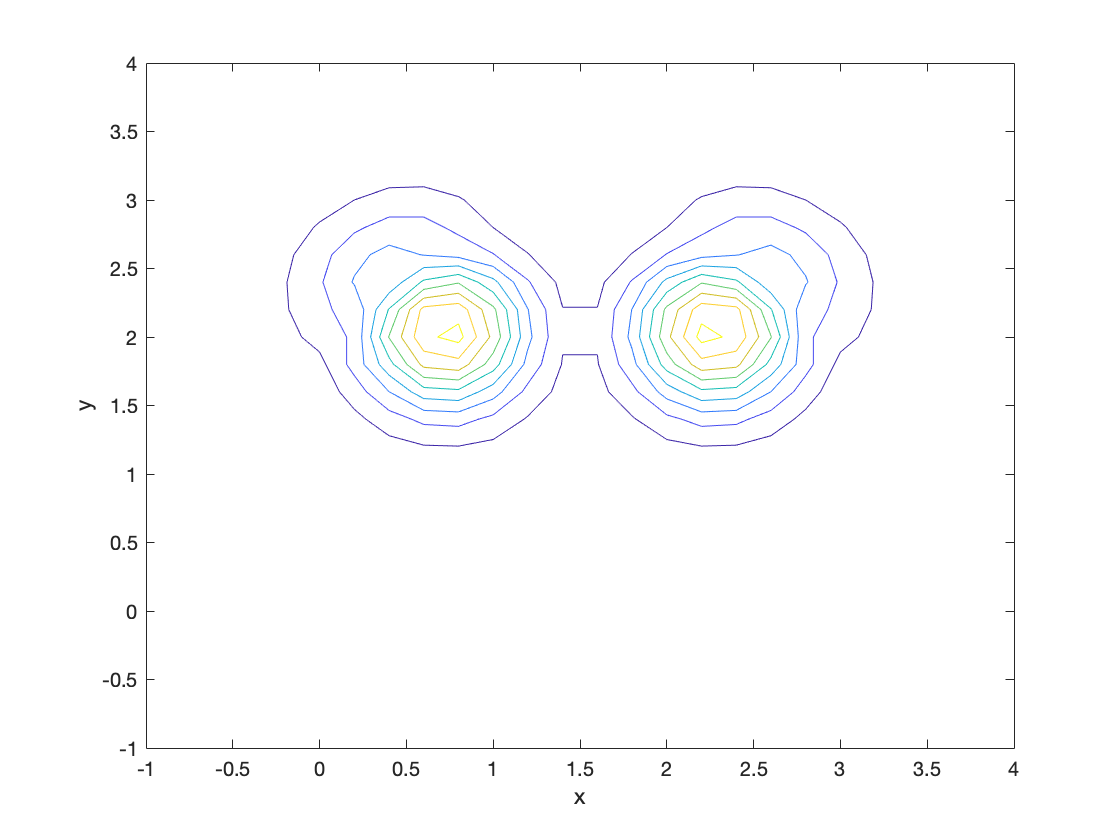}
\end{minipage}}
\subfigure{
\begin{minipage}[b]{0.31\linewidth}
\includegraphics[width=1.15\linewidth]{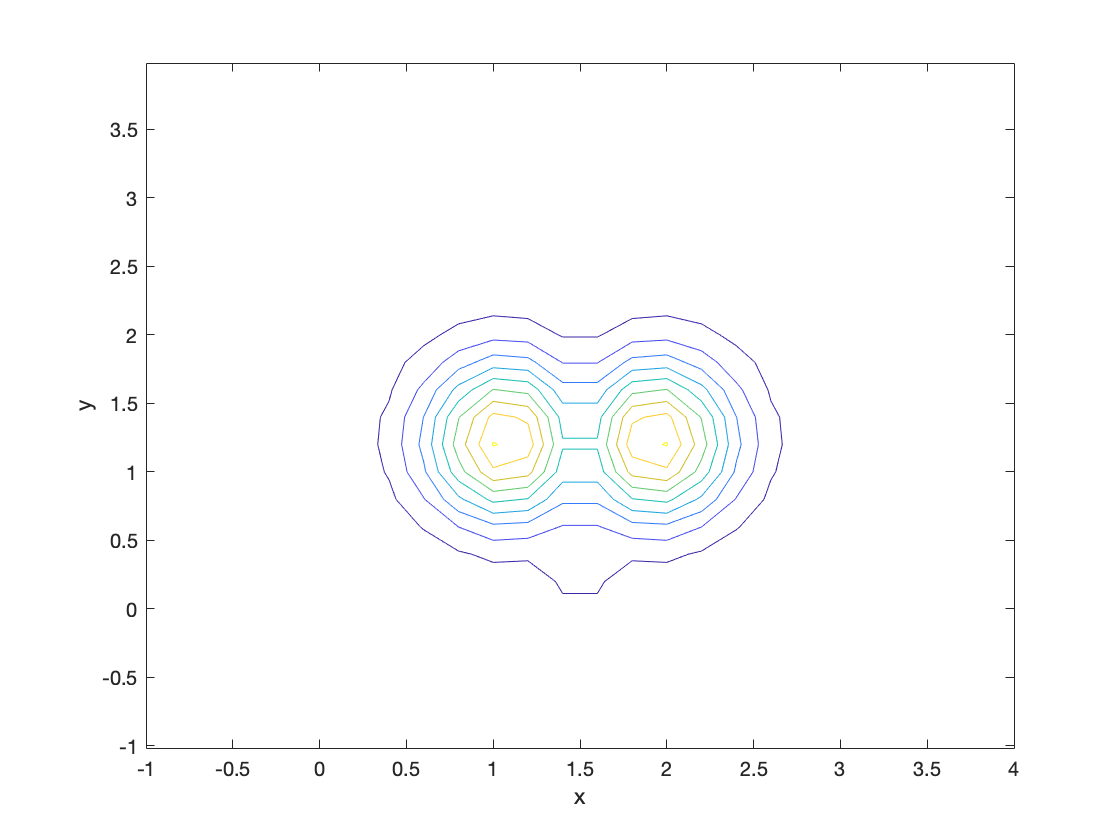}
\includegraphics[width=1.15\linewidth]{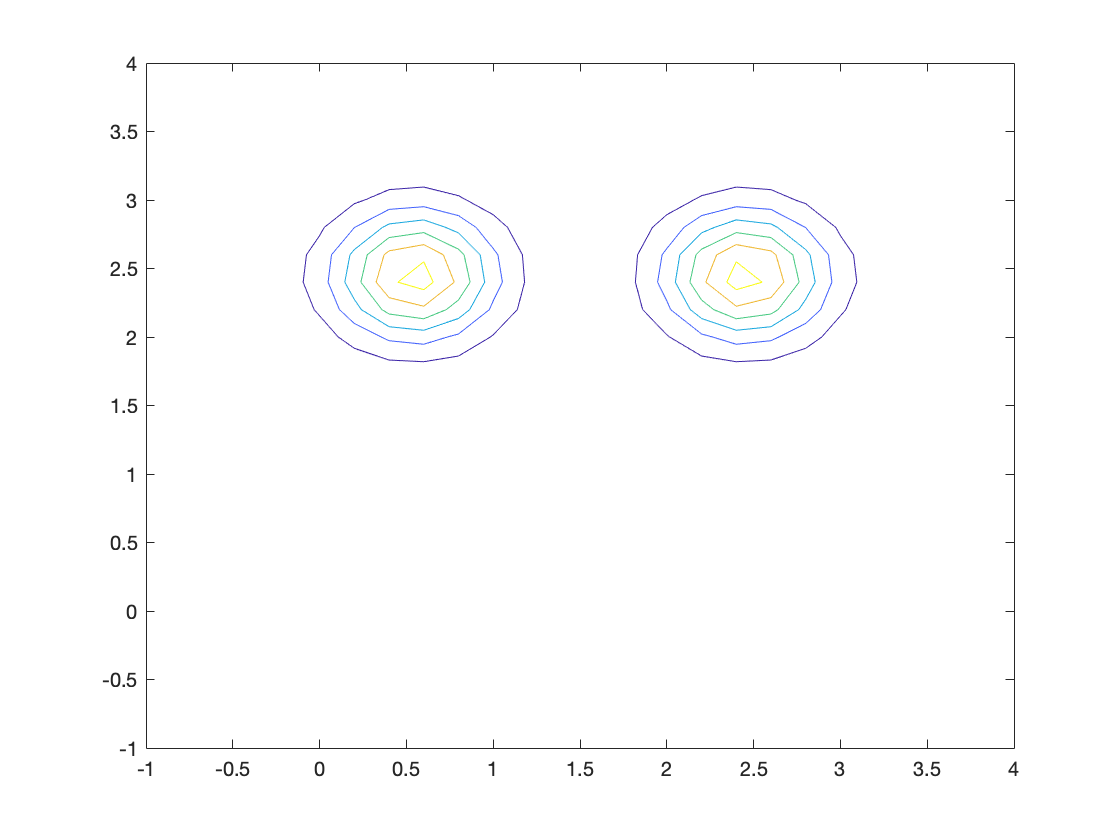}
\end{minipage}}

\centering 
\caption{Example \ref{Ex5}: contour plots of $\rho$ at $t=0, 0.2, 0.4,0.6,0.8,1$.}
\label{ncontour-den}
\end{figure}

\begin{figure}
\centering
\subfigure {
\begin{minipage}[b]{0.75\linewidth}
\includegraphics[width=1.15\linewidth]{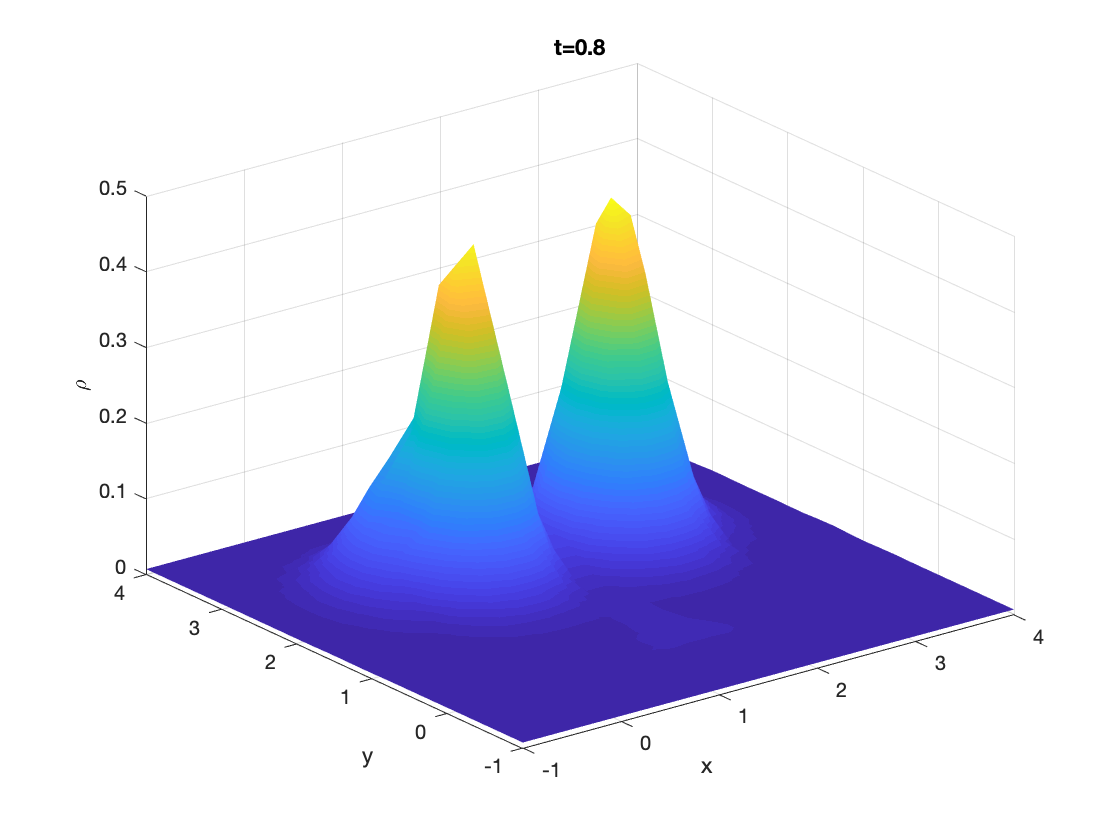}
\end{minipage}}
\centering 
\caption{Example \ref{Ex5}: the surface $\rho$ at $t=0.8$.}
\label{dent=0.8}
\end{figure}

\begin{figure}
\centering
\subfigure {
\begin{minipage}[b]{0.7\linewidth}
\includegraphics[width=1.25\linewidth]{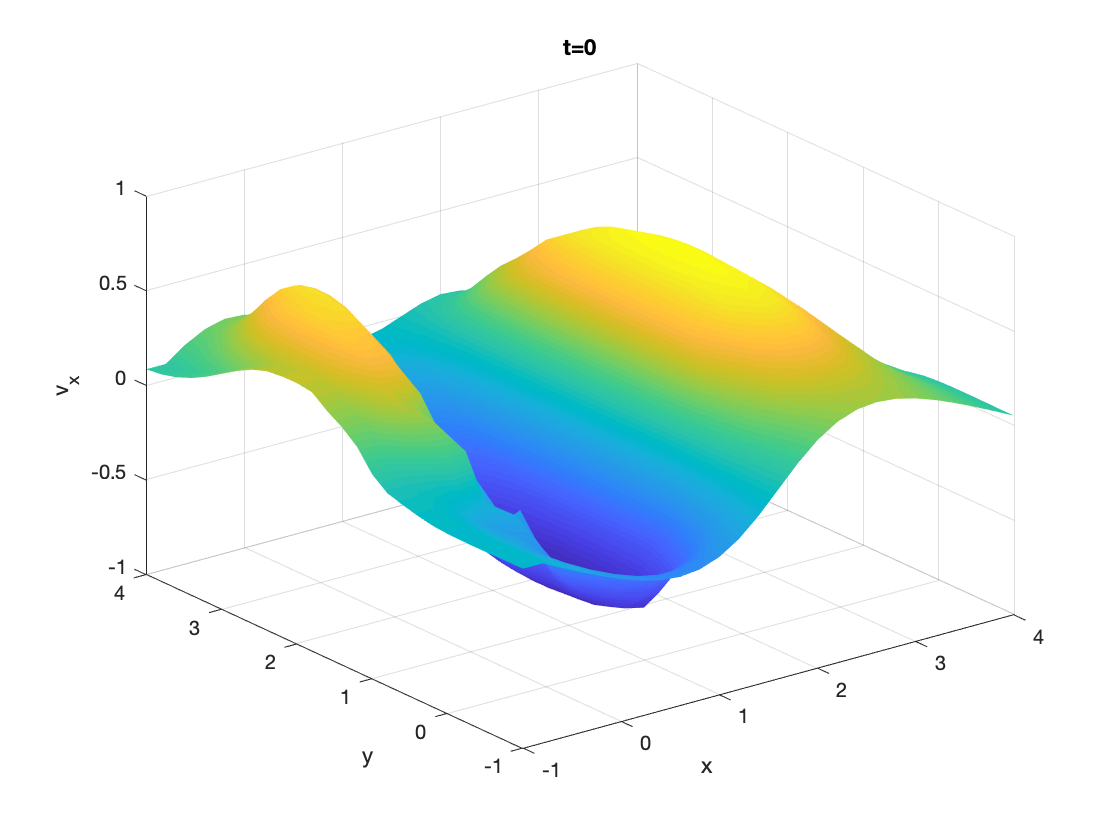}
\includegraphics[width=1.25\linewidth]{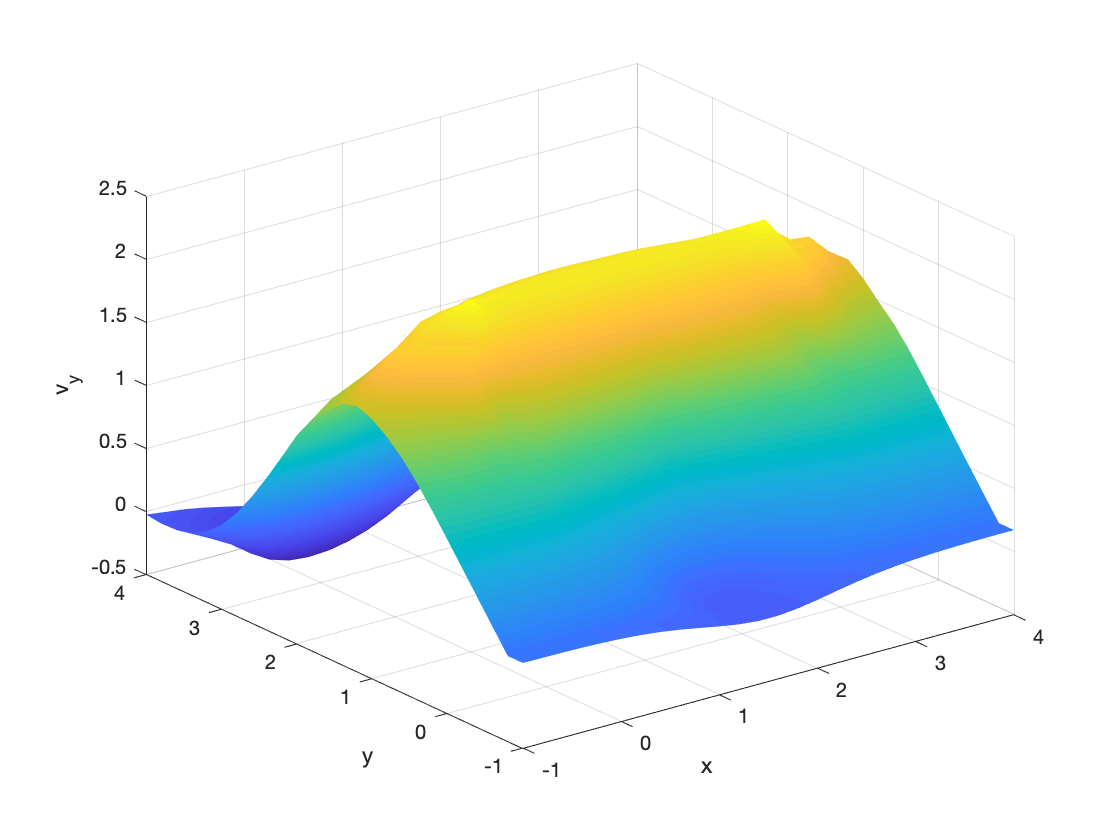}
\end{minipage}}

\centering 
\caption{Example \ref{Ex5}: the two components of the initial velocity.}
\label{in-vel}
\end{figure}

\end{example}

\begin{example}\label{Ex6}
Spatial domain is $\mathcal O=[-1,3]\times[-1,3]$. 
Initial and terminal densities from \eqref{2d-example} with parameters 
$a_0=2.5,a_1=5,b_0=0.5,b_1=1.5,c_0=5,c_1=10,d_0=0.3,d_1=1.3,r_0=r_1=0.001.$   
Contour plots of the density evolution are in 
Fig. \ref{contour-den1}.

\begin{figure}
\centering
\subfigure {
\begin{minipage}[b]{0.31\linewidth}
\includegraphics[width=1.15\linewidth]{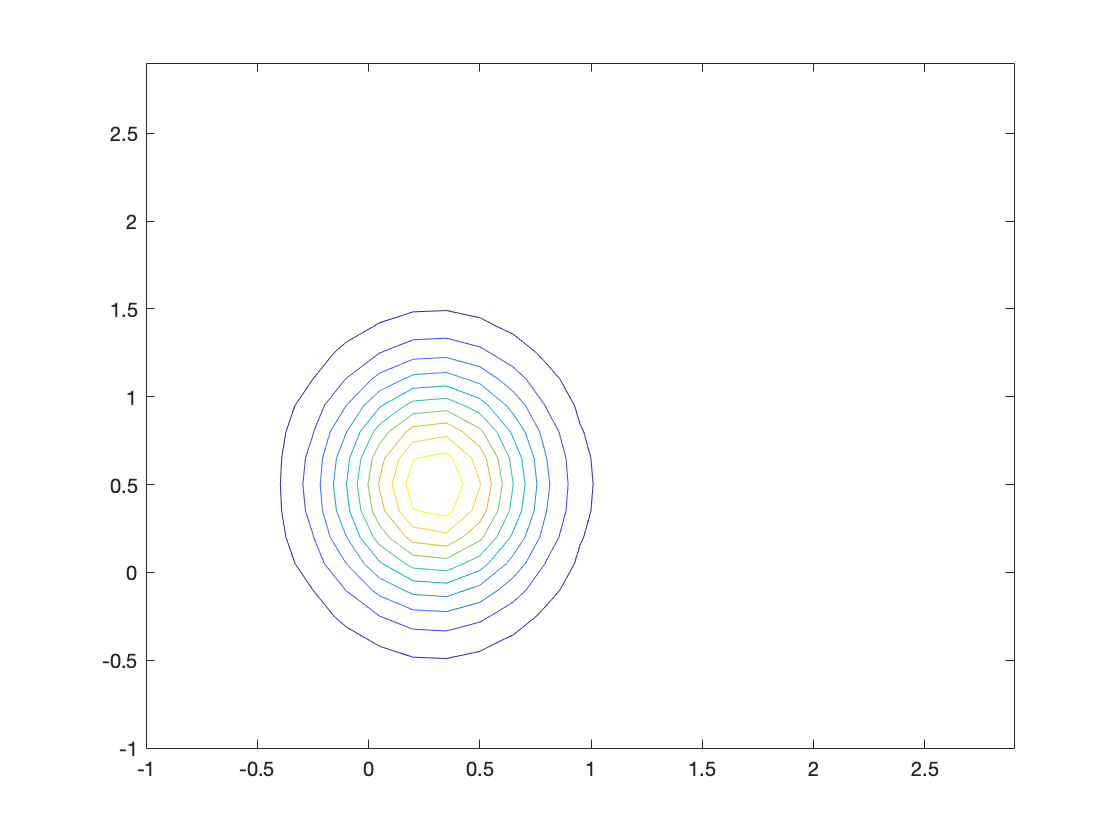}
\includegraphics[width=1.15\linewidth]{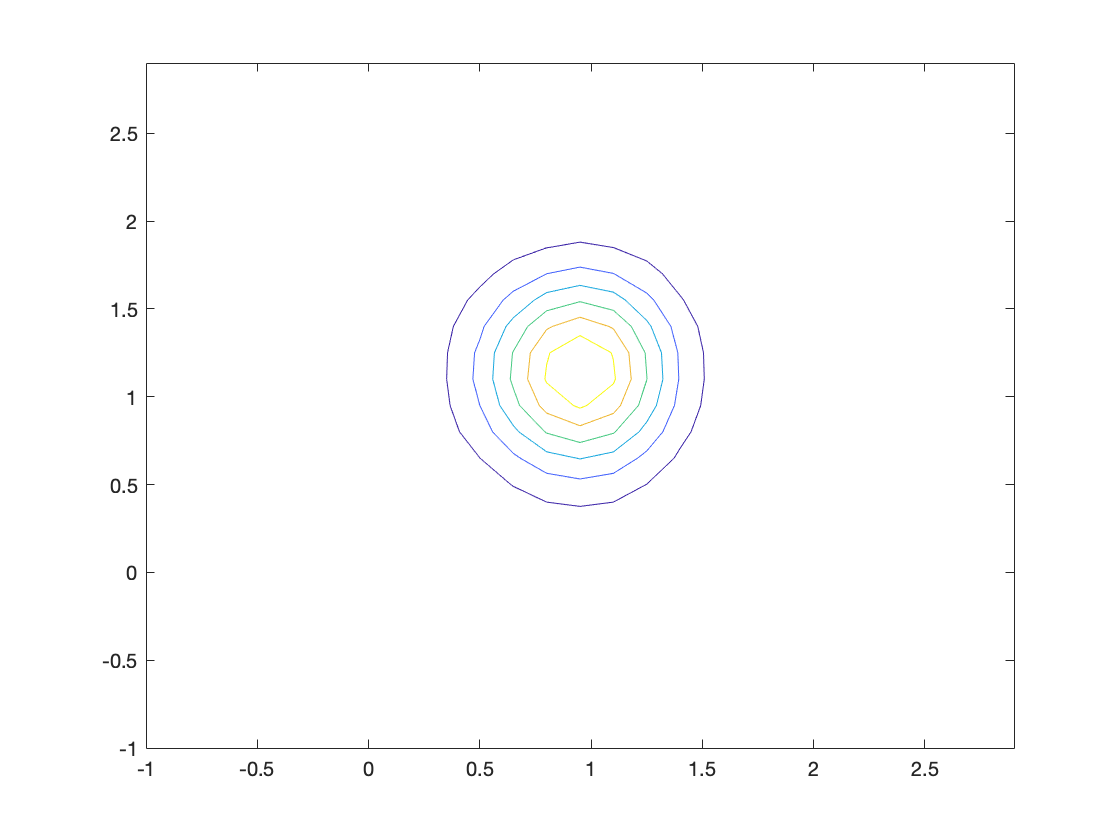}
\end{minipage}}
\subfigure{
\begin{minipage}[b]{0.31\linewidth}
\includegraphics[width=1.15\linewidth]{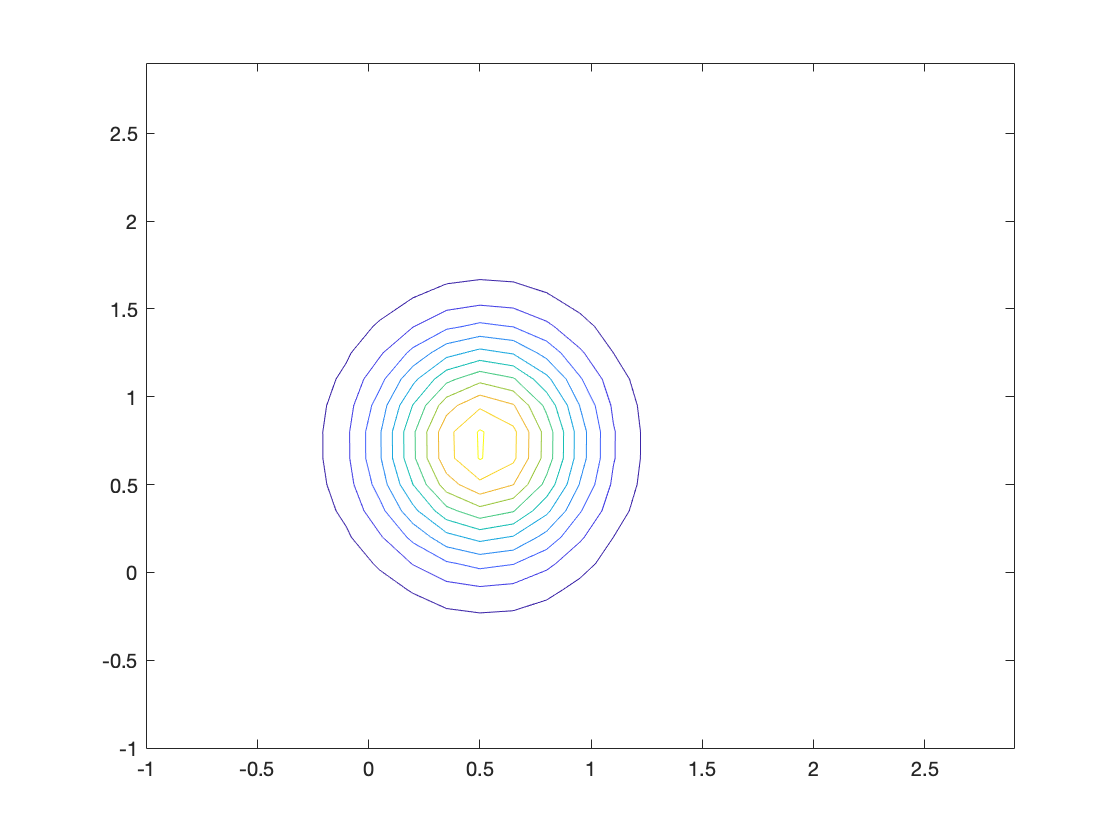}
\includegraphics[width=1.15\linewidth]{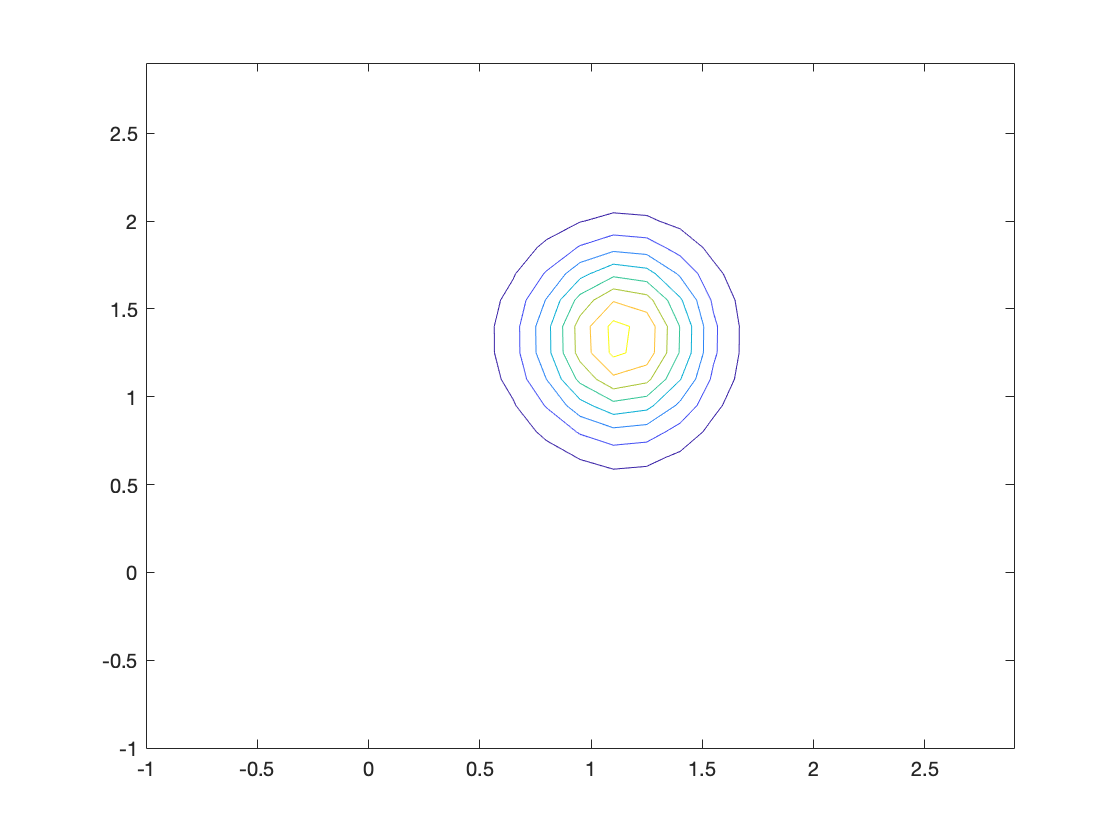}
\end{minipage}}
\subfigure{
\begin{minipage}[b]{0.31\linewidth}
\includegraphics[width=1.15\linewidth]{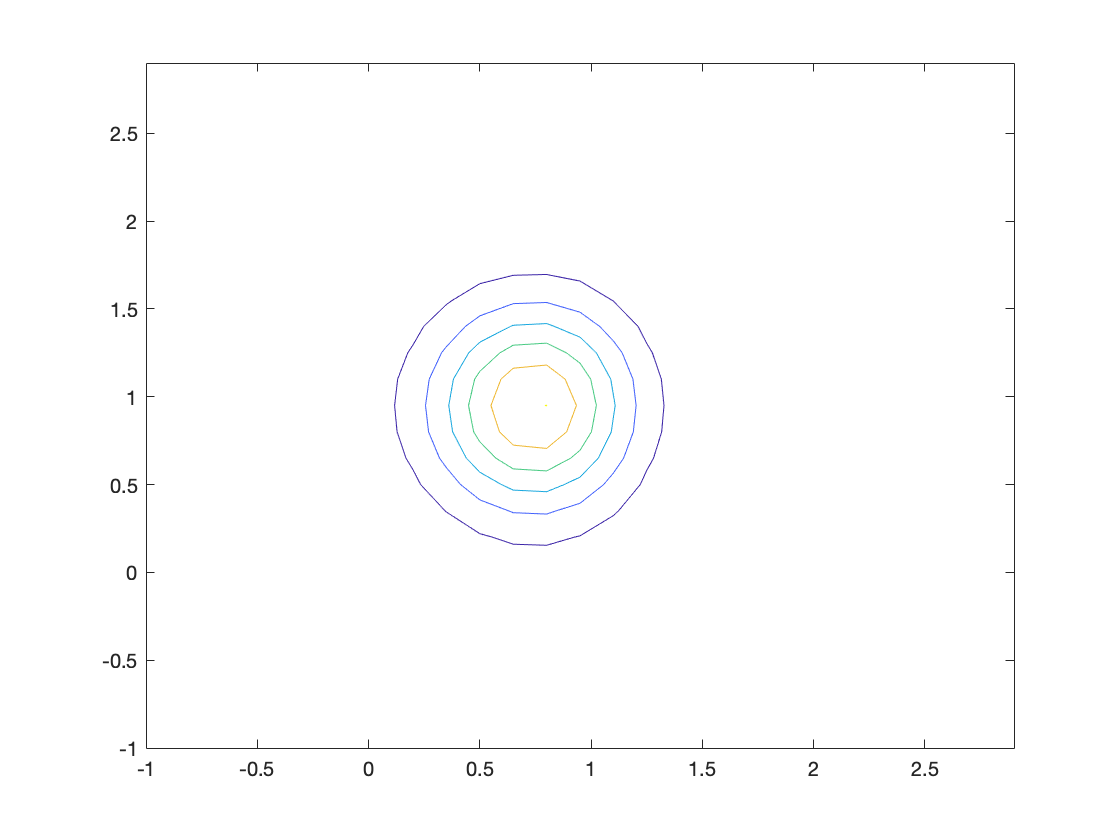}
\includegraphics[width=1.15\linewidth]{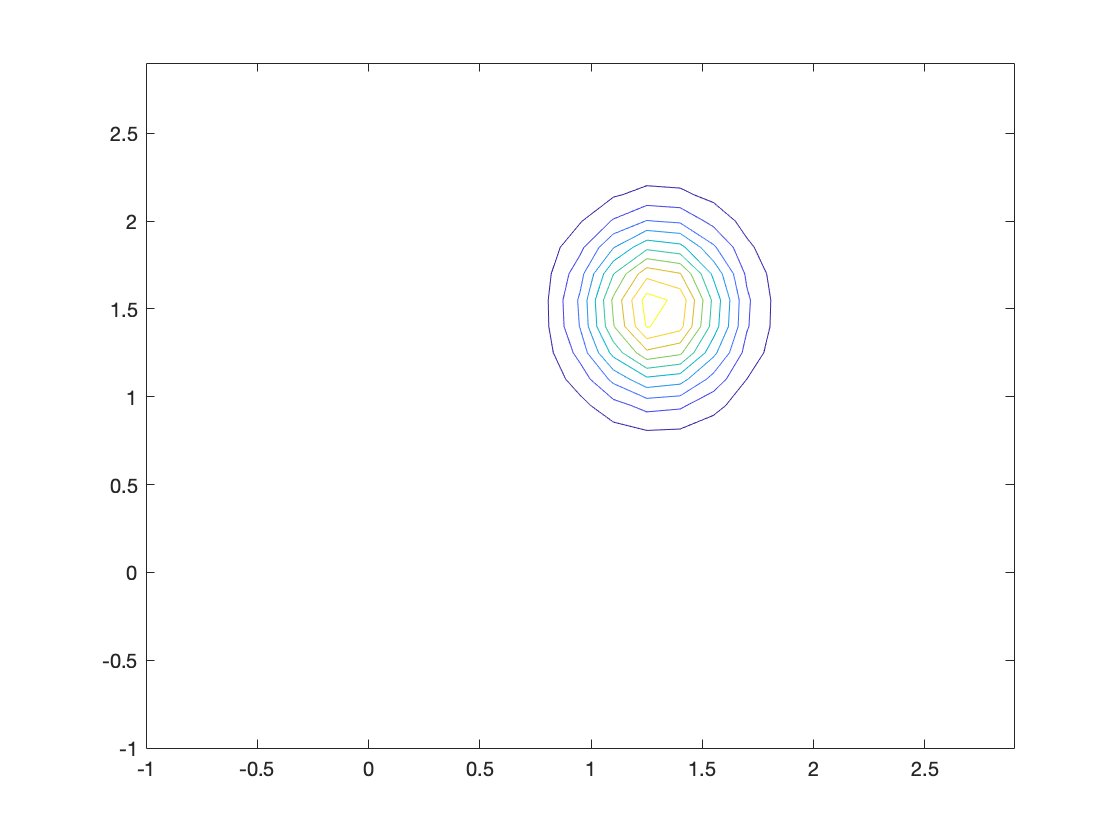}
\end{minipage}}

\centering 
\caption{Example \ref{Ex6}: contour plots of $\rho$ at $t=0, 0.2, 0.4,0.6,0.8,1$.}
\label{contour-den1}
\end{figure}

\end{example}

For the next set of examples, we choose the initial or terminal distributions as the normalization of the Laplace distribution. 
We use $a_0,b_0,c_0,r_0$ or  $a_1,b_1,c_1,r_1$ to indicate the parameters of the Laplace type distribution given as:
\begin{equation}\label{2d-Laplace} \begin{split}
\widehat \mu & =\exp(-a_0|x_2-b_0|-c_0|x_1-d_0|)+r_0,\\
\widehat \nu & =\exp(-a_1|x_2-b_1|-c_1|x_1-d_1|)+r_1.
\end{split}\end{equation}

\begin{example}\label{Ex7}
Spatial domain $\mathcal O=[-1,3]\times[-1,3]$.  Initial and terminal densities are normalizations of the Laplace distributions
in \eqref{2d-Laplace} 
with parameters $a_0=a_1=5,b_0=0.5,b_1=1.5,c_0=c_1=5,d_0=0.6,d_1=1.6,r_0=r_1=0.001.$
Contour plots of the density evolution are in Fig. \ref{ccontour-den1}.

\begin{figure}
\centering
\subfigure {
\begin{minipage}[b]{0.31\linewidth}
\includegraphics[width=1.15\linewidth]{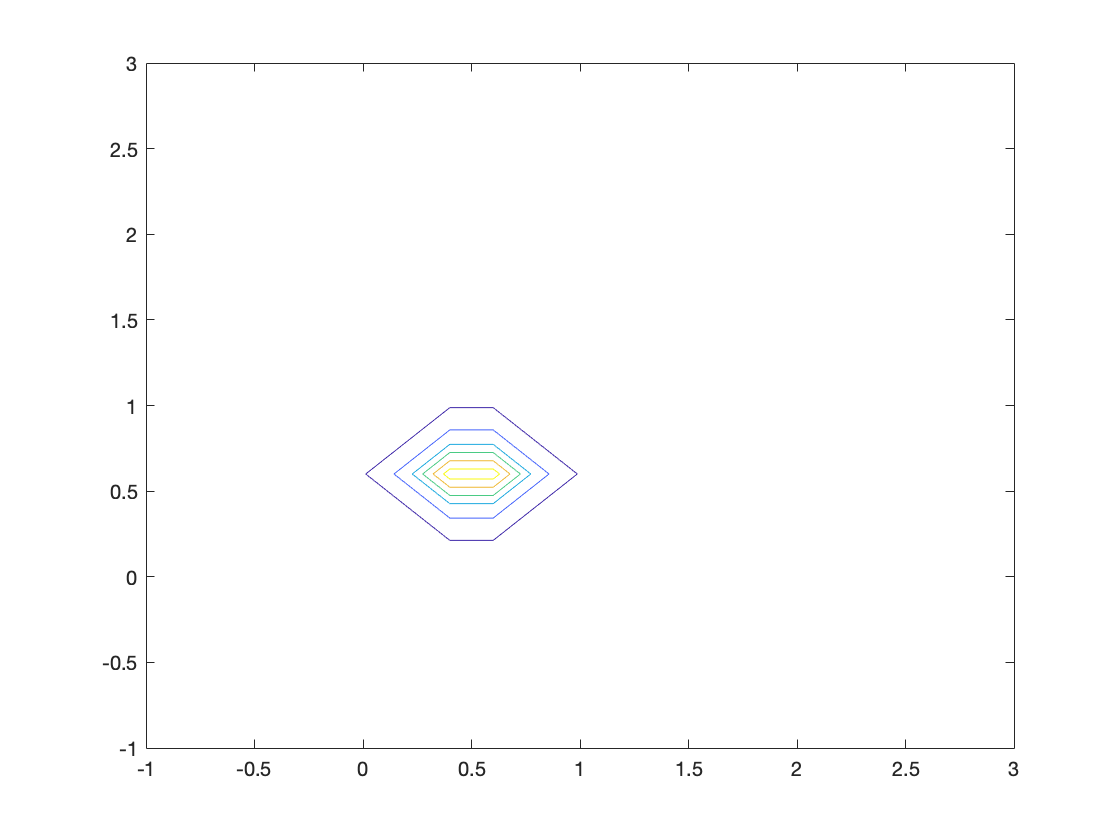}
\includegraphics[width=1.15\linewidth]{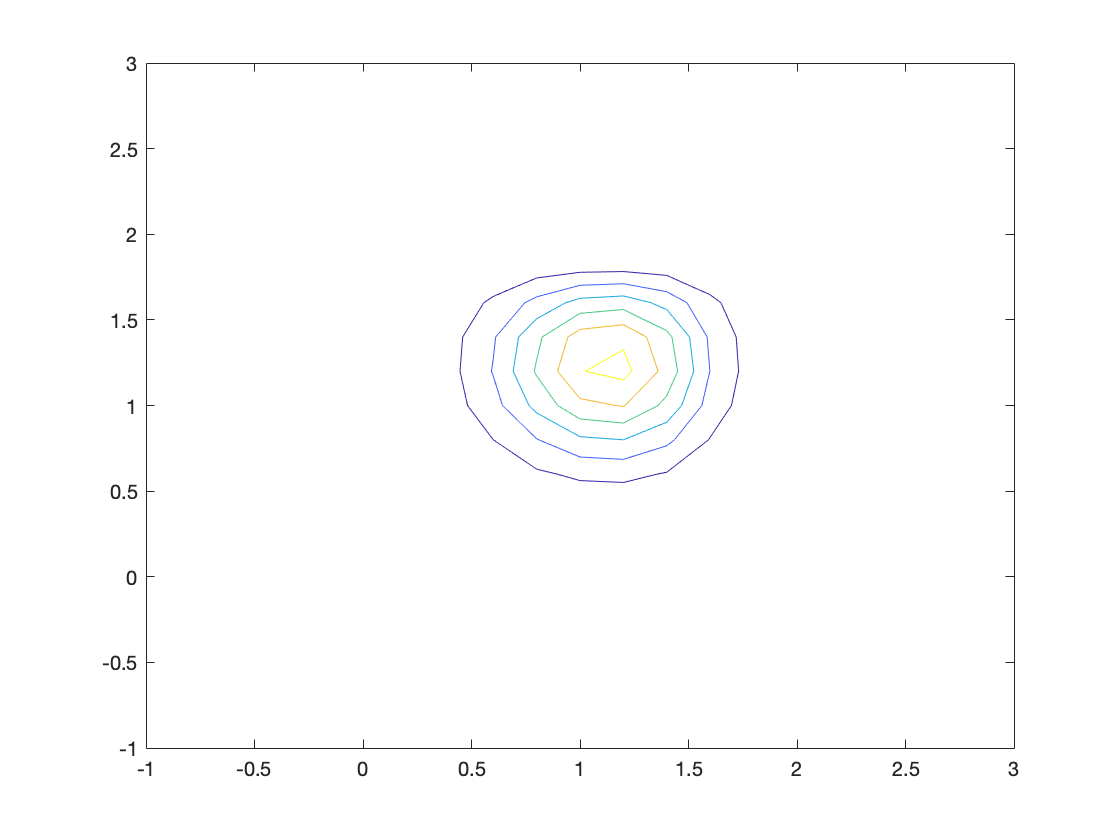}
\end{minipage}}
\subfigure{
\begin{minipage}[b]{0.31\linewidth}
\includegraphics[width=1.15\linewidth]{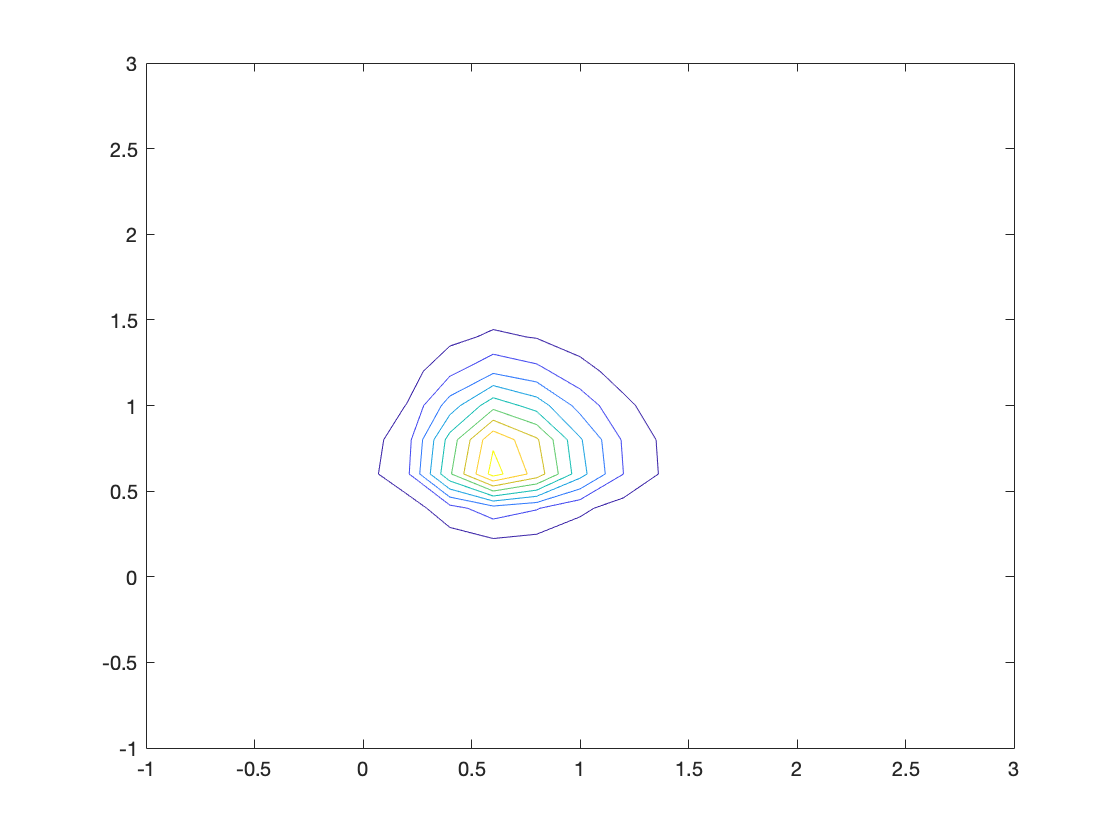}
\includegraphics[width=1.15\linewidth]{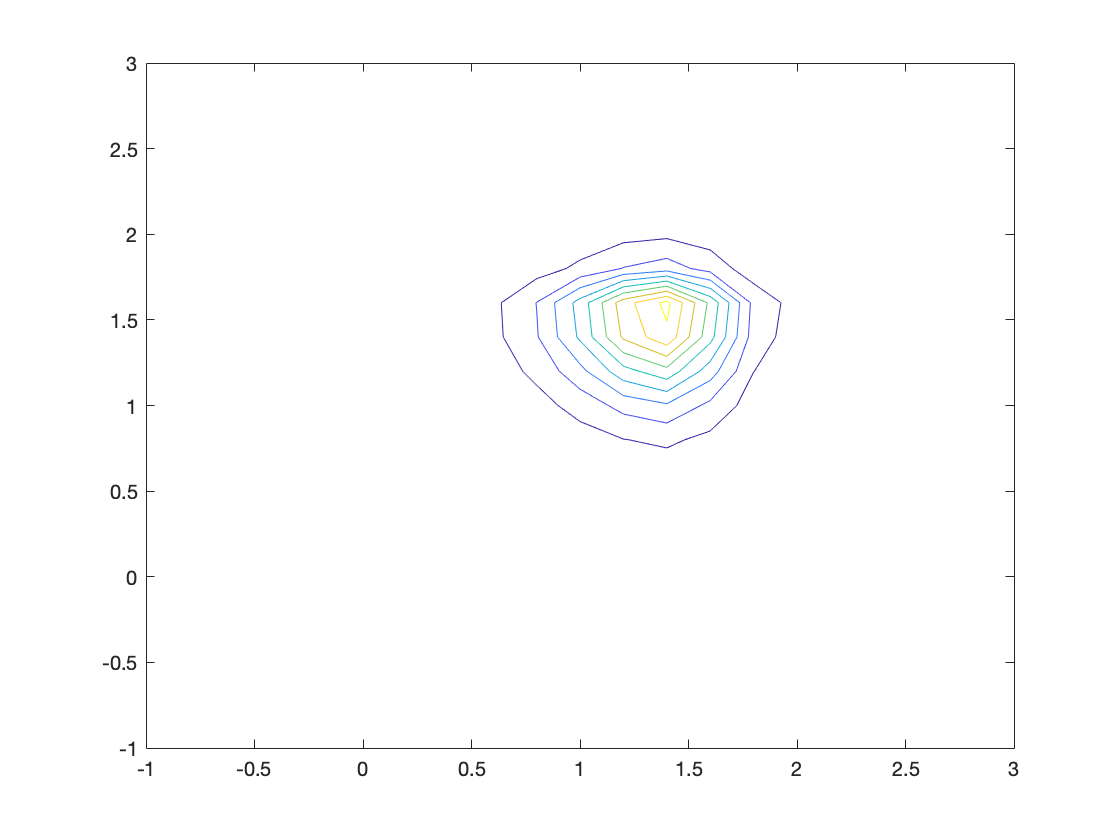}
\end{minipage}}
\subfigure{
\begin{minipage}[b]{0.31\linewidth}
\includegraphics[width=1.15\linewidth]{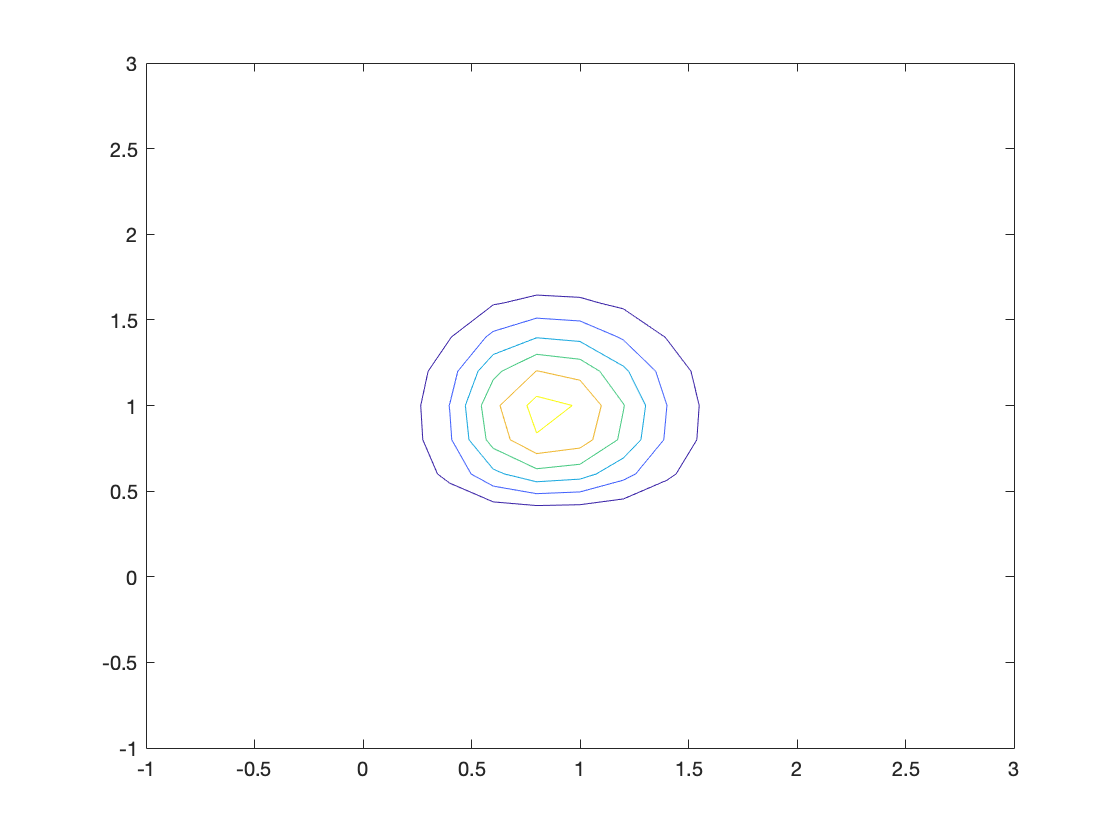}
\includegraphics[width=1.15\linewidth]{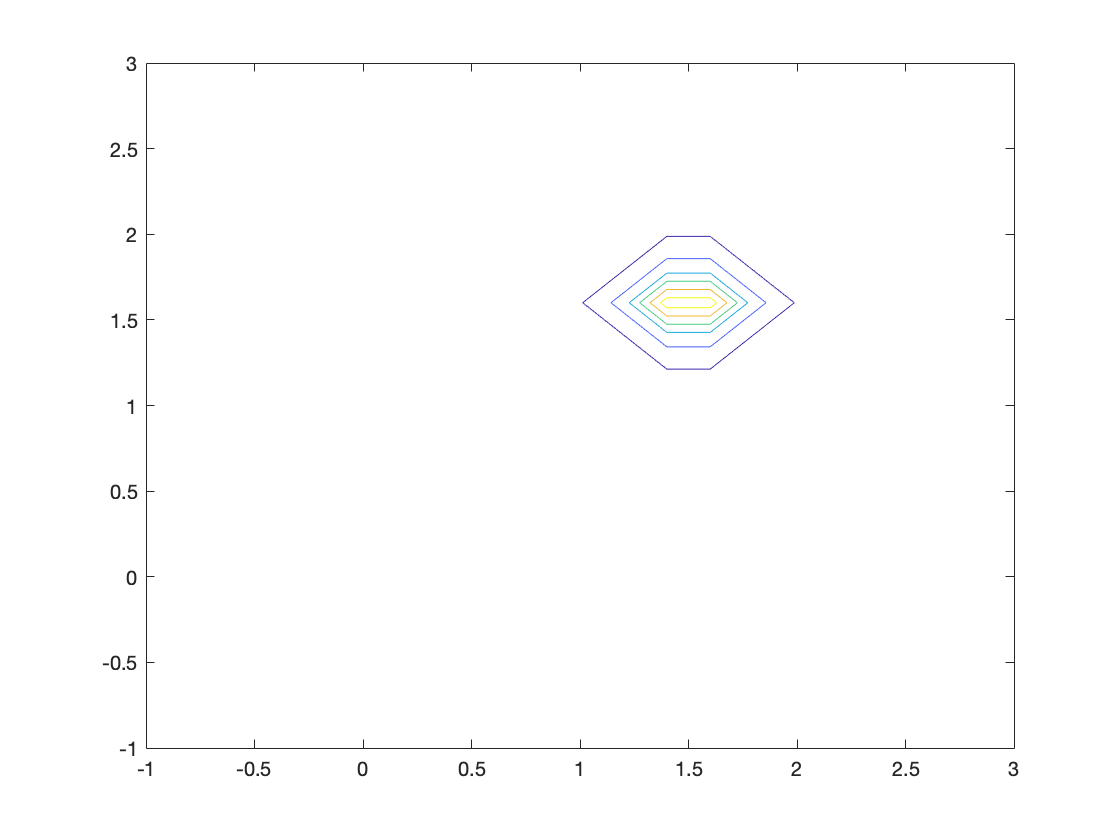}
\end{minipage}}

\centering 
\caption{Example \ref{Ex7}: contour plots of  $\rho$ at times $t=0, 0.2, 0.4,0.6,0.8,1$.}
\label{ccontour-den1}
\end{figure}

\end{example}

\begin{example}\label{Ex8}
Spatial domain $\mathcal O=[-1,3]\times[-1,3].$  Initial density is the uniform distribution.
Terminal density is the normalization of the Laplace distribution $\hat \nu$
with parameters $a_1=10,b_1=1.5,c_1=10,d=1.6,r=0.01.$ The contour plots of the density evolution are presented in 
Fig. \ref{dcontour-den1}.

\begin{figure}
\centering
\subfigure {
\begin{minipage}[b]{0.31\linewidth}
\includegraphics[width=1.15\linewidth]{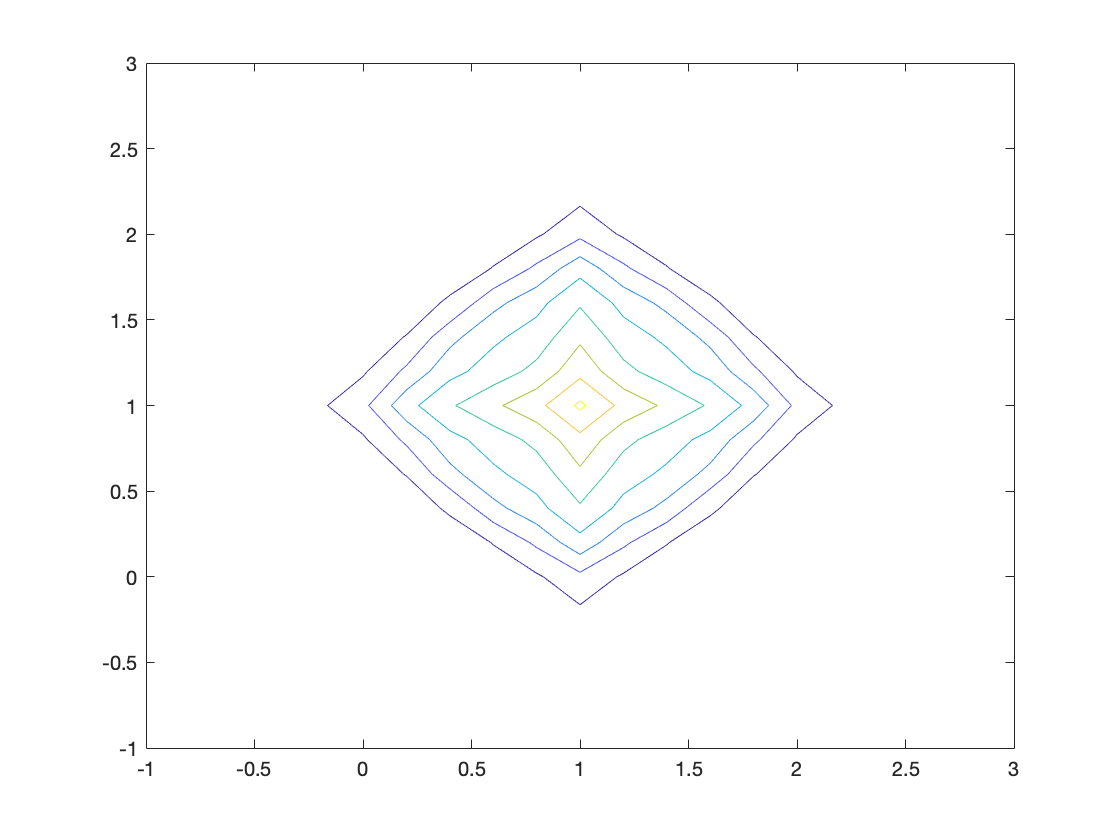}
\includegraphics[width=1.15\linewidth]{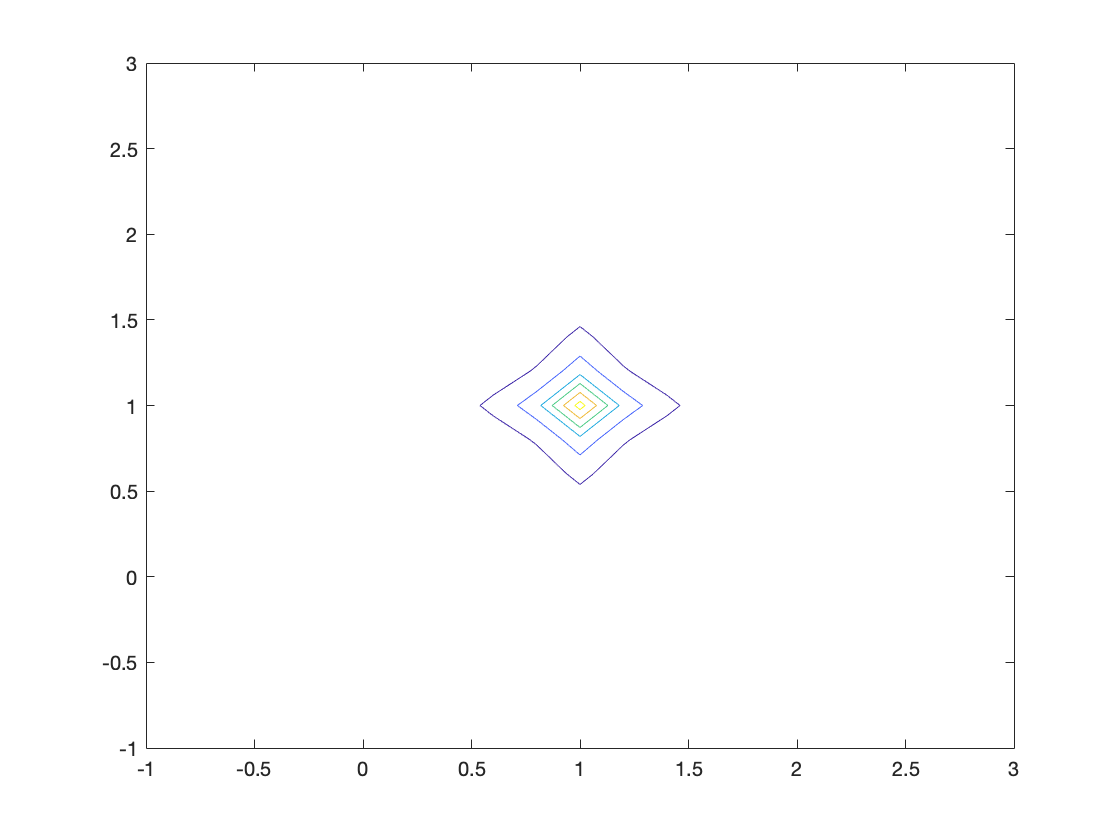}
\end{minipage}}
\subfigure{
\begin{minipage}[b]{0.31\linewidth}
\includegraphics[width=1.15\linewidth]{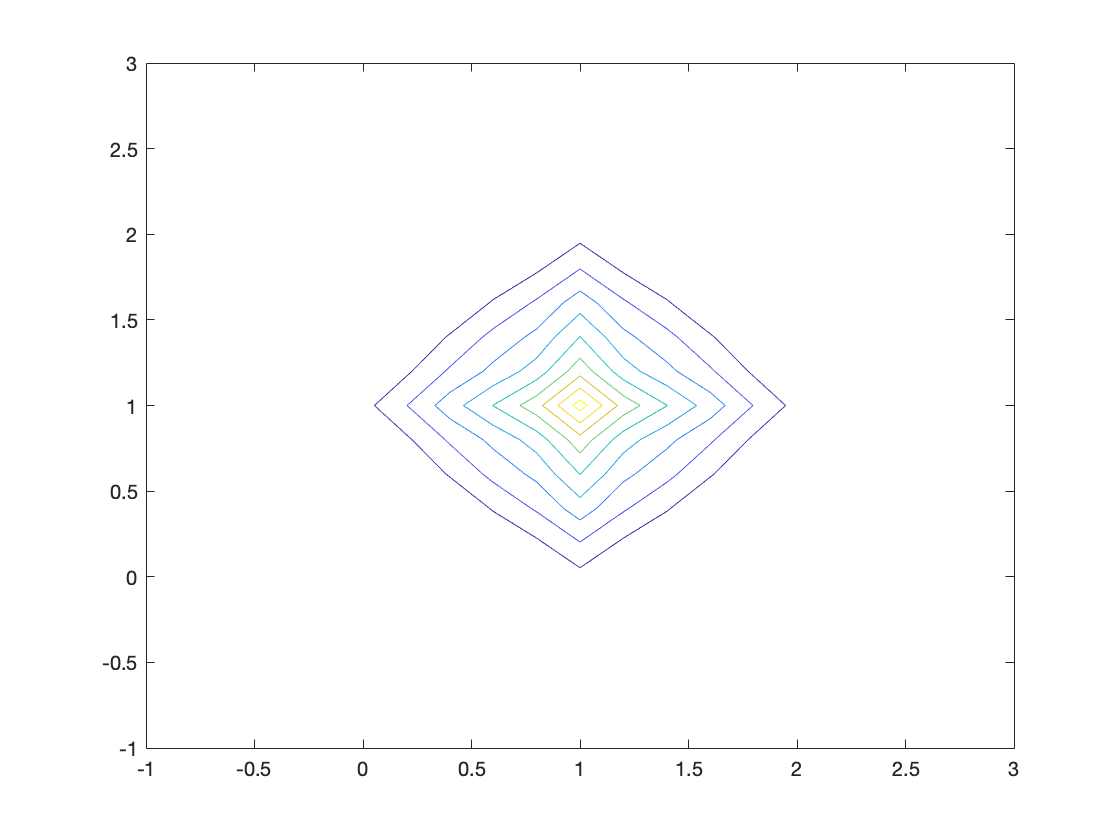}
\includegraphics[width=1.15\linewidth]{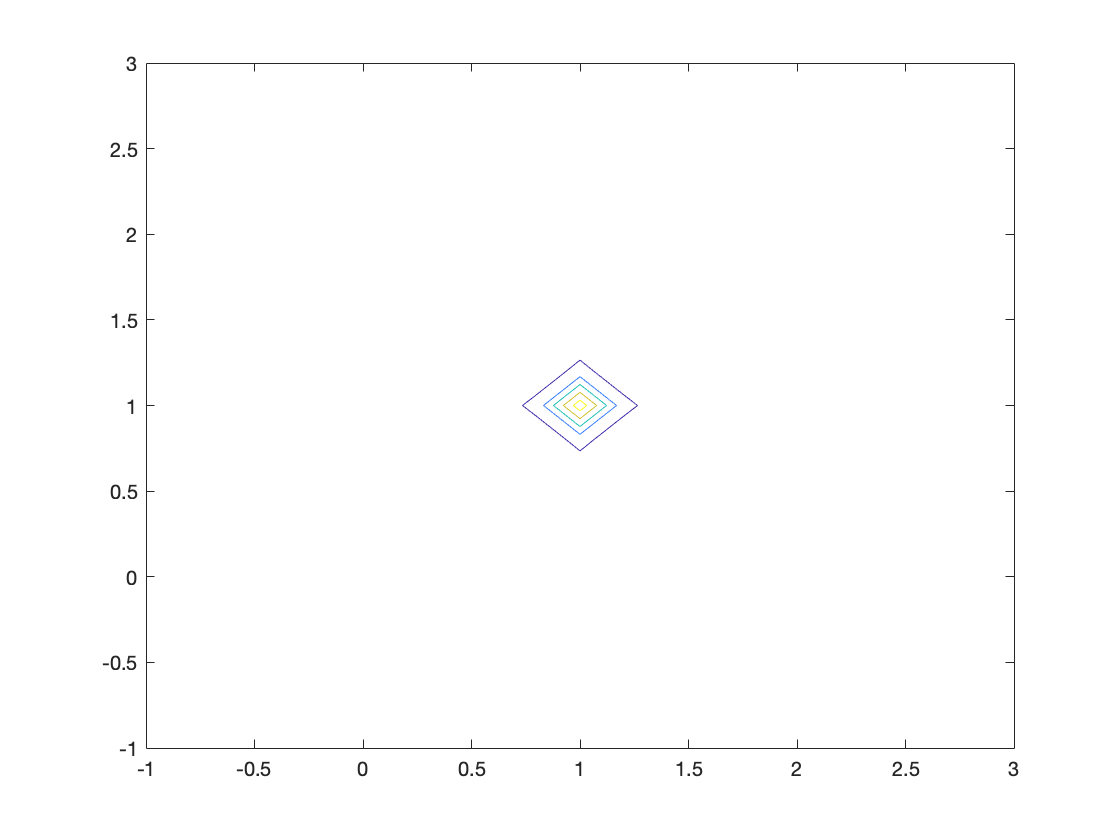}
\end{minipage}}
\subfigure{
\begin{minipage}[b]{0.31\linewidth}
\includegraphics[width=1.15\linewidth]{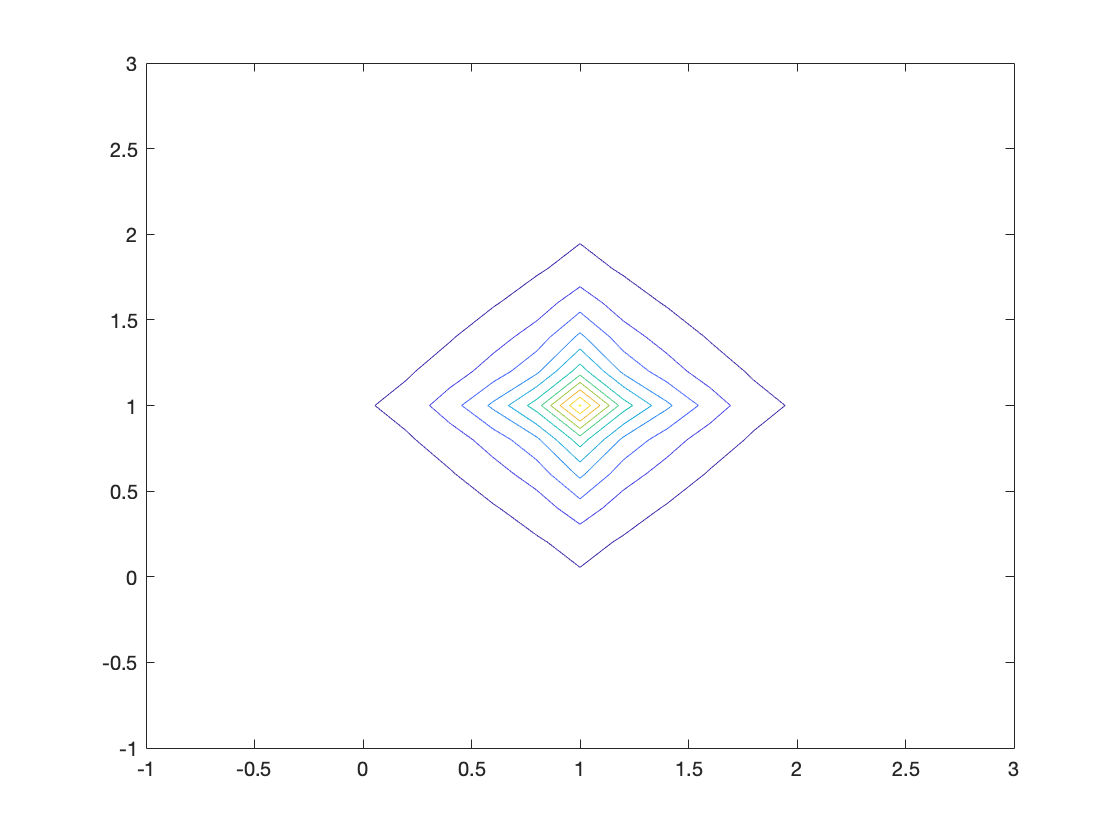}
\includegraphics[width=1.15\linewidth]{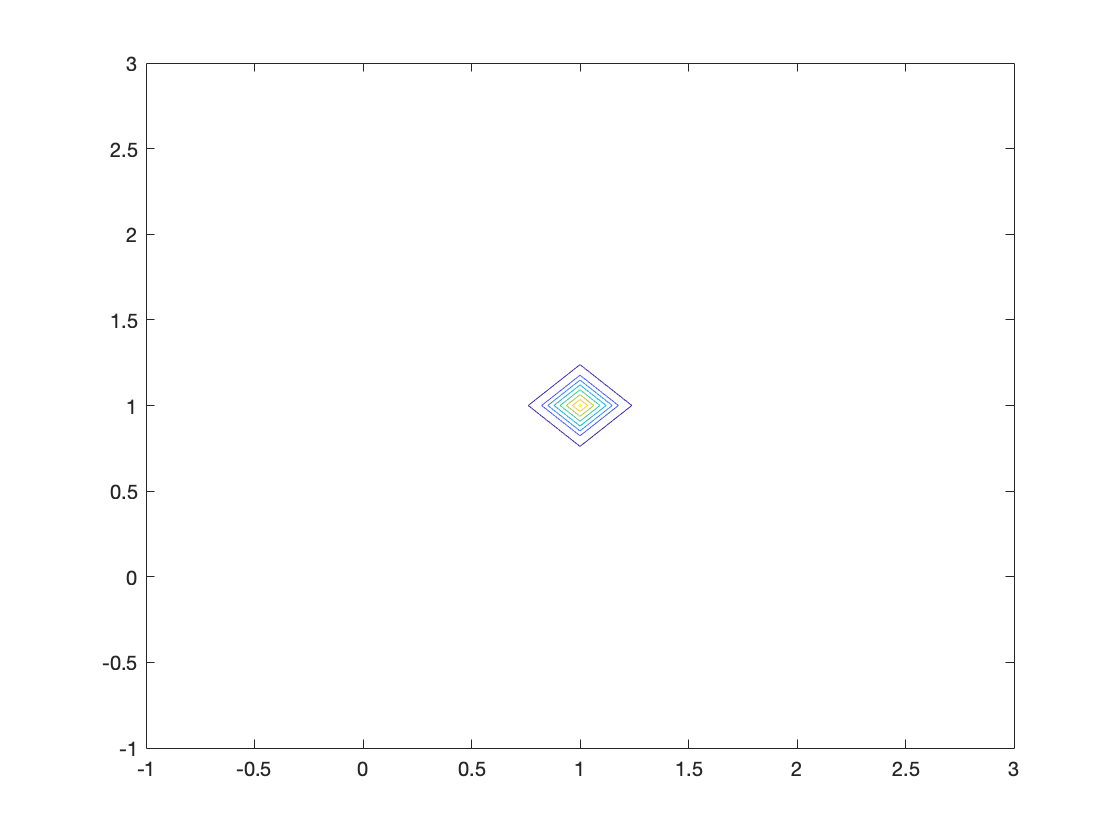}
\end{minipage}}

\centering 
\caption{Example \ref{Ex8}: contour plots of $\rho$ at times $t=0.1, 0.3, 0.5,0.7,0.9,1$.}
\label{dcontour-den1}
\end{figure}

\end{example}

\begin{example}\label{Ex9}
Spatial domain $\mathcal O=[-1,3]\times[-1,3].$
Initial density is the normalization of   
\begin{align*}
\mu &=(x_1+1)^2(x_1-3)^2+(x_2+1)^2(x_2-3)^2.
\end{align*}
Terminal distribution is the normalization of   
$\hat \nu$ in \eqref{2d-Laplace} with parameters $a_1=10,b_1=1.5,c_1=10,d_1=1.6,r_1=0.01.$
The contour plots of the density evolution are presented in 
Fig. \ref{econtour-den1}.

\begin{figure}
	\centering
	\subfigure {
		\begin{minipage}[b]{0.31\linewidth}
			\includegraphics[width=1.15\linewidth]{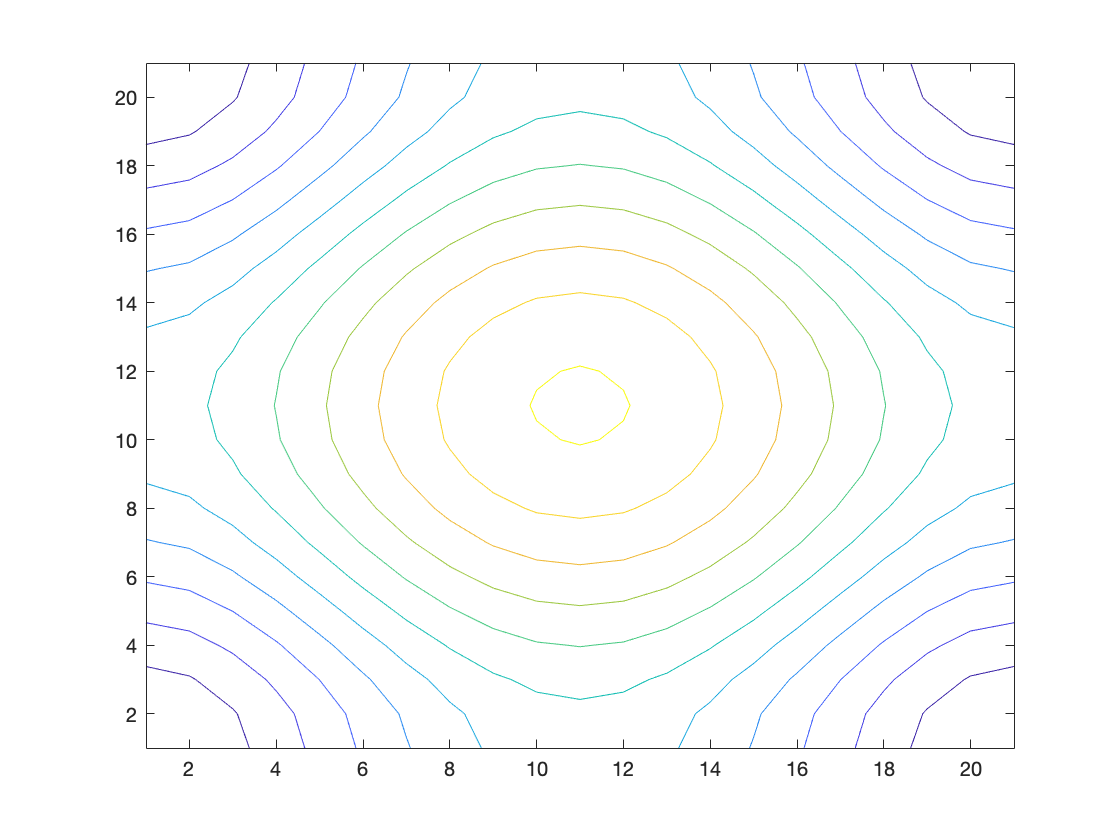}
			\includegraphics[width=1.15\linewidth]{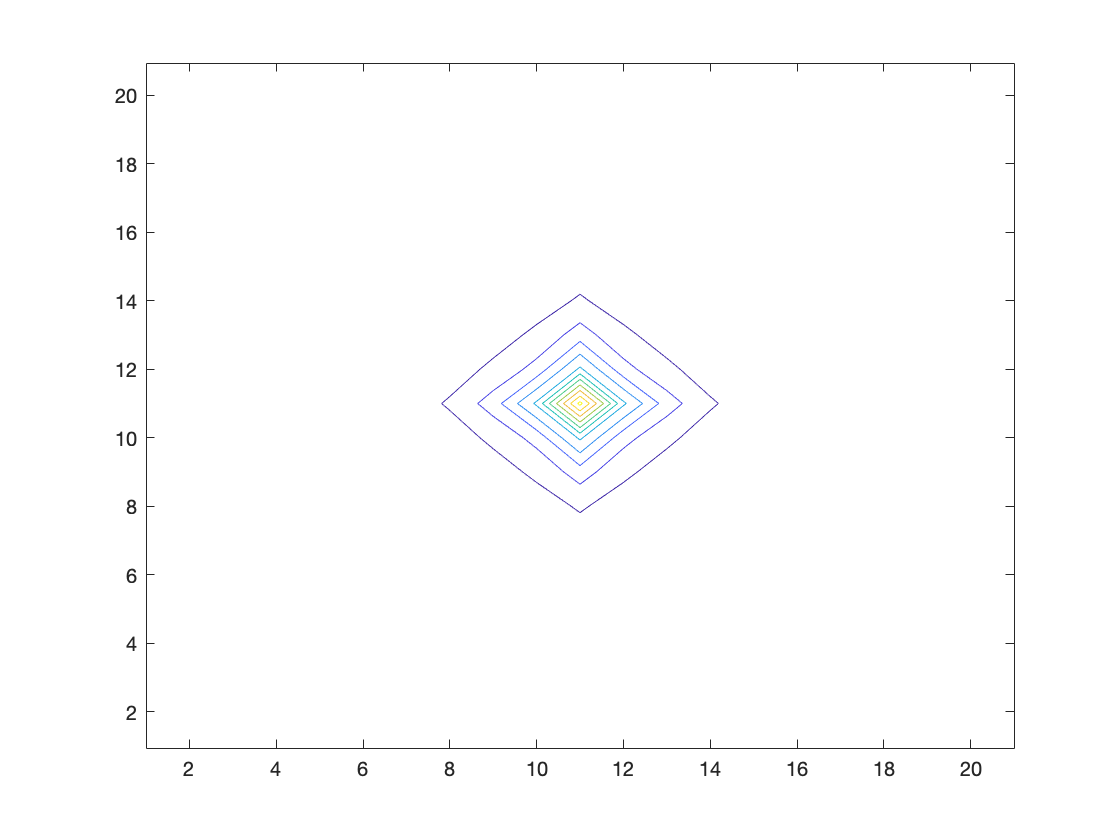}
	\end{minipage}}
	\subfigure{
		\begin{minipage}[b]{0.31\linewidth}
			\includegraphics[width=1.15\linewidth]{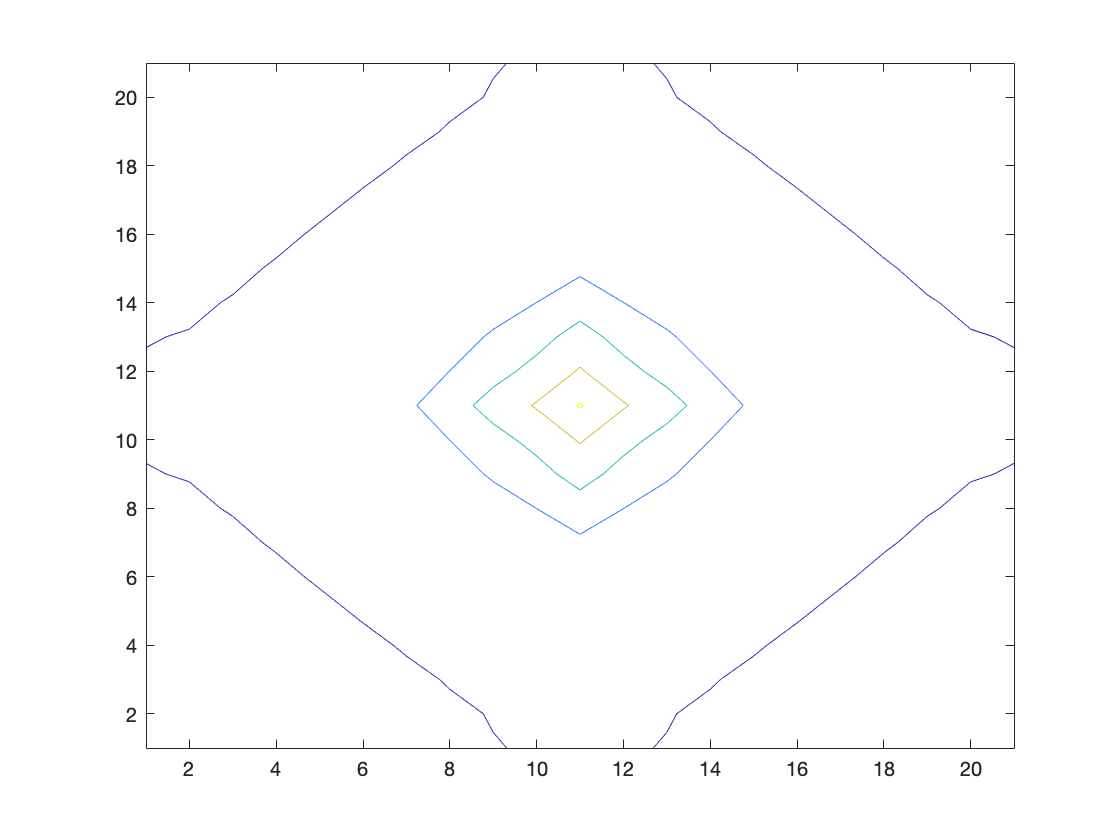}
			\includegraphics[width=1.15\linewidth]{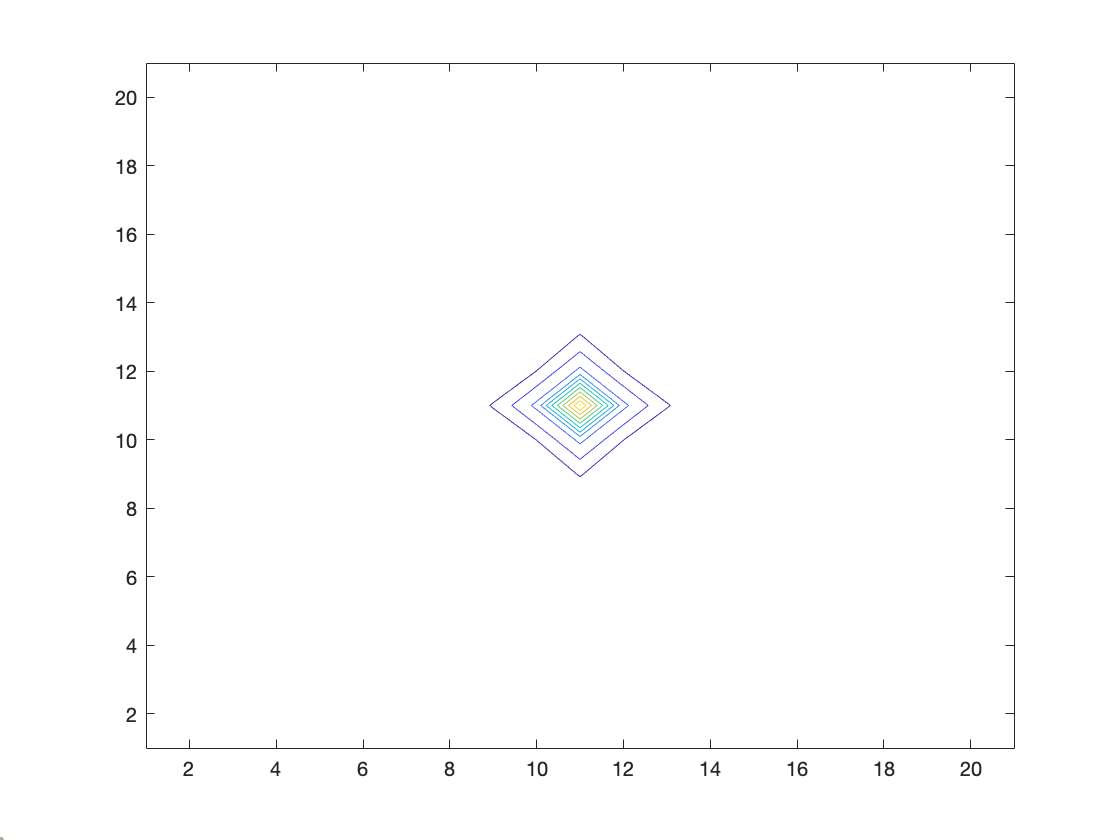}
	\end{minipage}}
	\subfigure{
		\begin{minipage}[b]{0.31\linewidth}
			\includegraphics[width=1.15\linewidth]{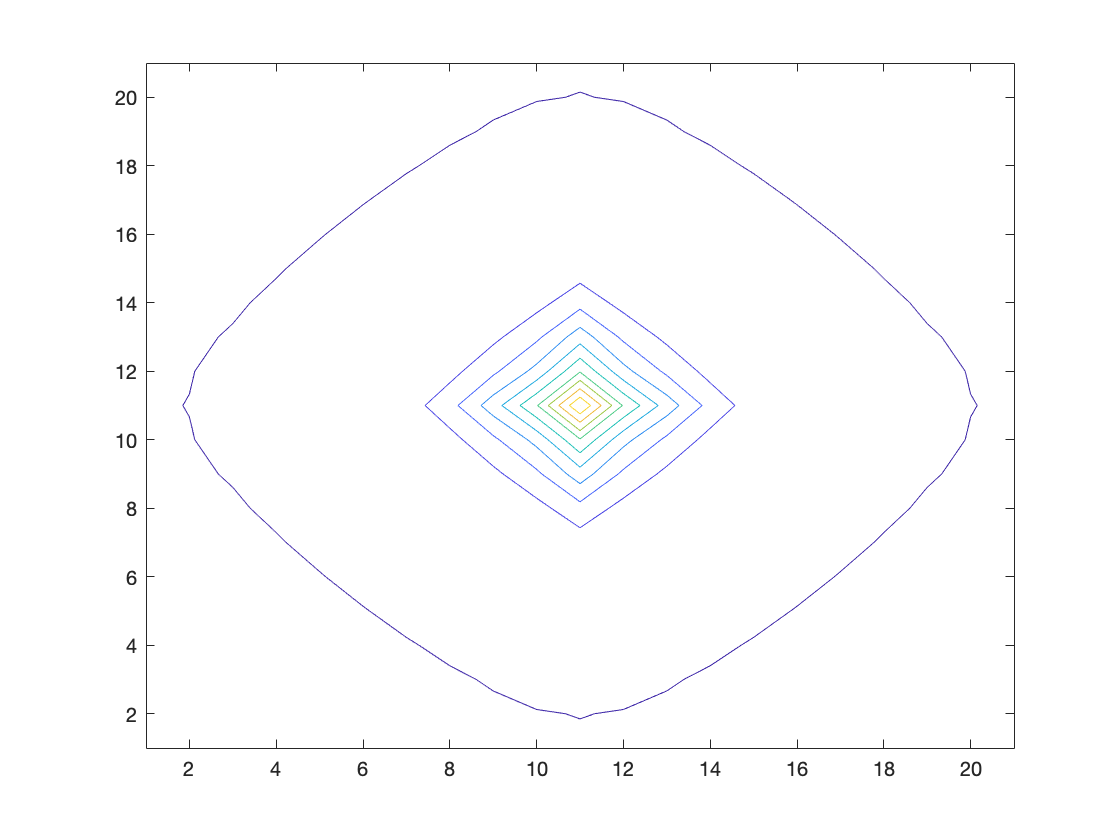}
			\includegraphics[width=1.15\linewidth]{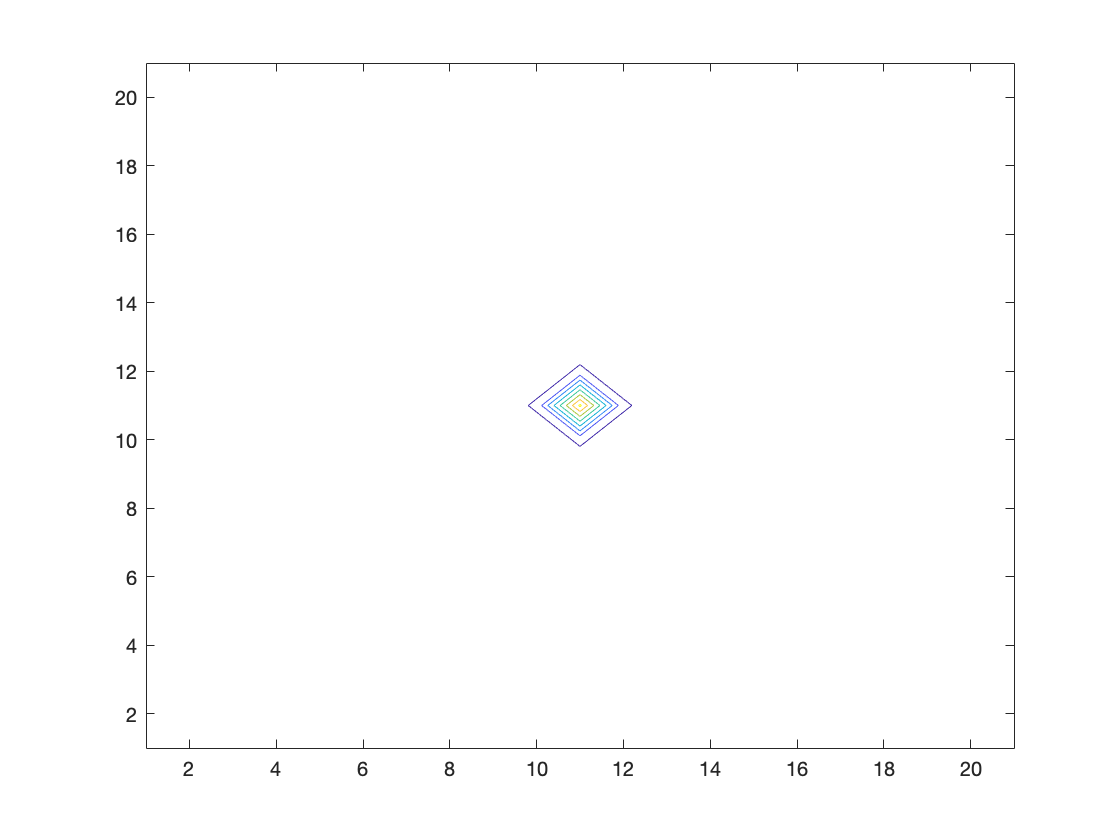}
	\end{minipage}}
	
	\centering 
	\caption{Example \ref{Ex9}: contour plots of $\rho$ at times $t=0, 0.2, 0.4,0.6,0.8,1$.}
	\label{econtour-den1}
\end{figure}

\end{example}

\begin{example}\label{Ex10}
Spatial domain $\mathcal O=[x_L,x_R]\times[x_L,x_R],$ $x_L=-1,x_R=3.$ 
The initial density and terminal distributions are normalized Gaussian densities with parameters
$a_0=a_1=50,b_0=0.5,b_1=1.5,c_0=c_1=50,d_0=0.3,d_1=1.3, r_1=r_2=0.001.$
The contour plot of the density evolution is presented in 
Fig. \ref{rhot-den}.

\begin{figure}
\centering
\subfigure {
\begin{minipage}[b]{0.31\linewidth}
\includegraphics[width=1.15\linewidth]{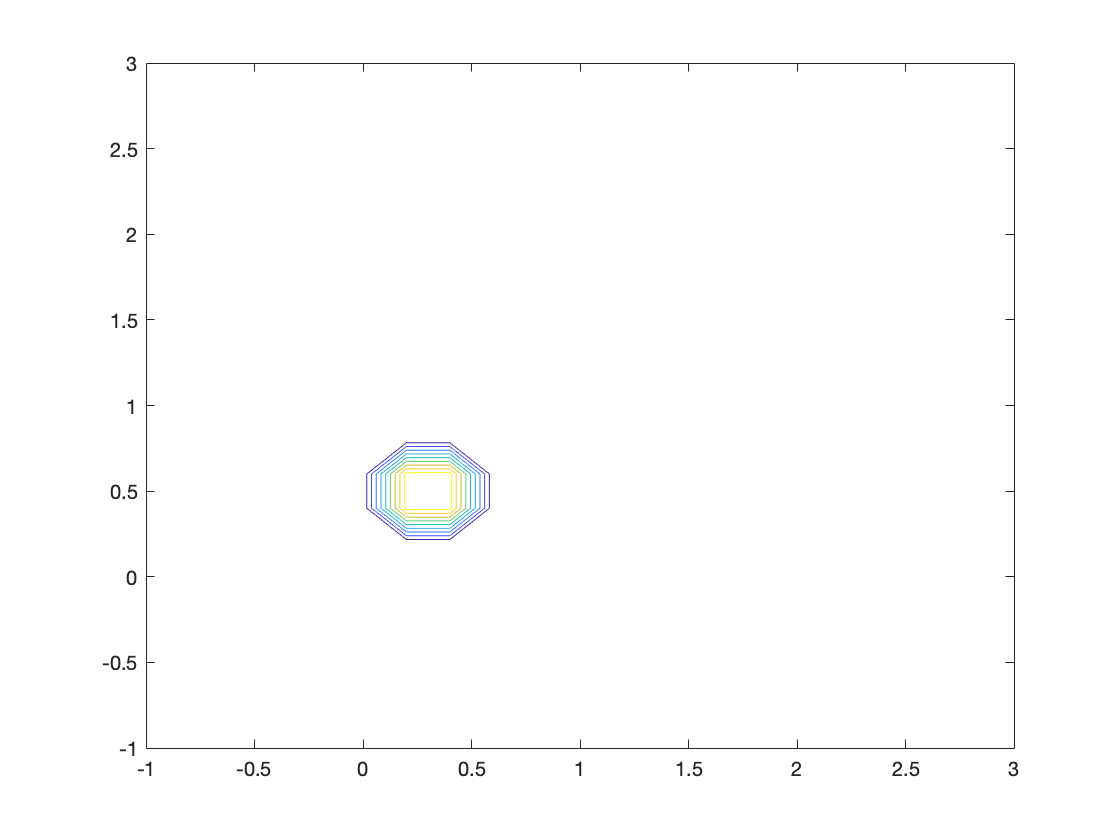}
\includegraphics[width=1.15\linewidth]{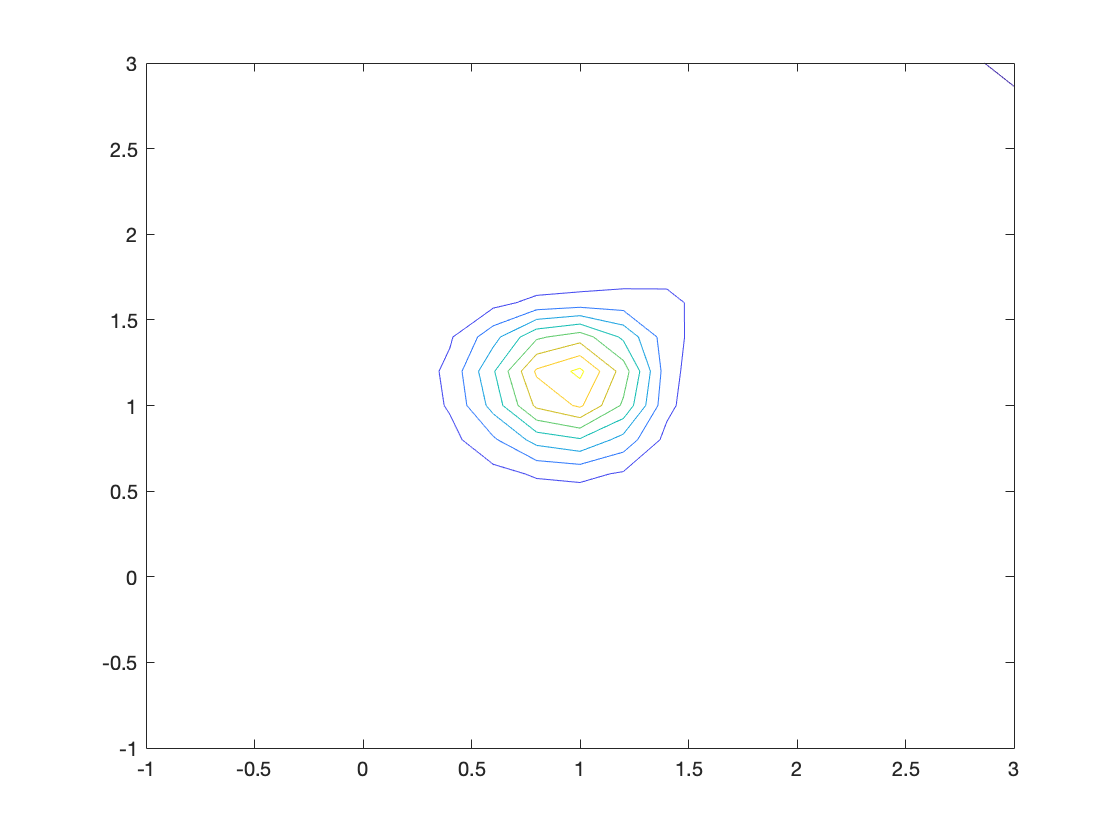}
\end{minipage}}
\subfigure{
\begin{minipage}[b]{0.31\linewidth}
\includegraphics[width=1.15\linewidth]{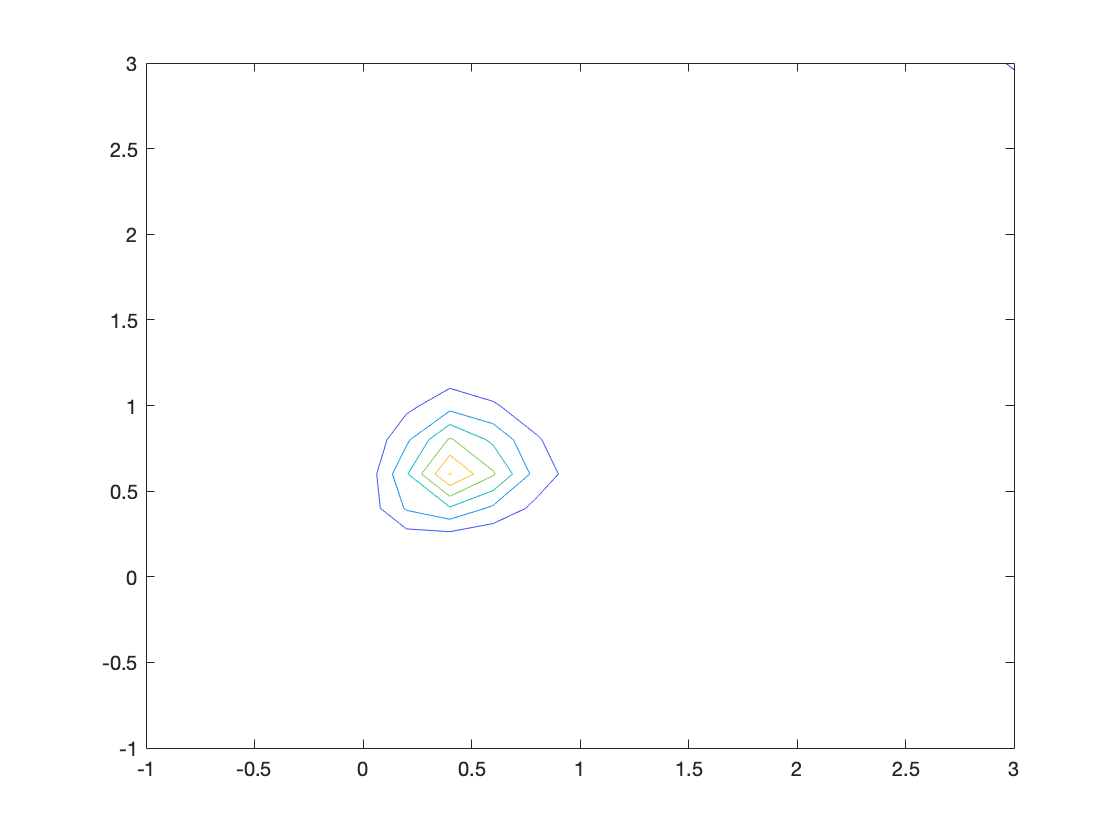}
\includegraphics[width=1.15\linewidth]{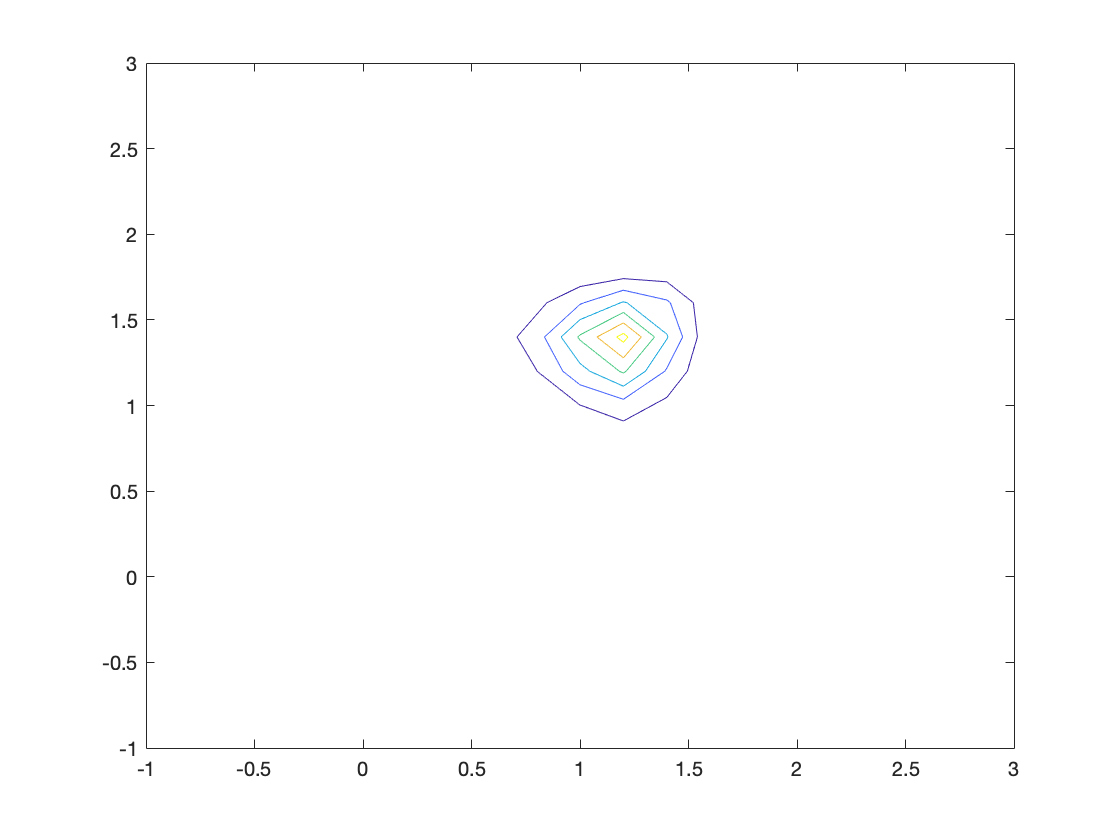}
\end{minipage}}
\subfigure{
\begin{minipage}[b]{0.31\linewidth}
\includegraphics[width=1.15\linewidth]{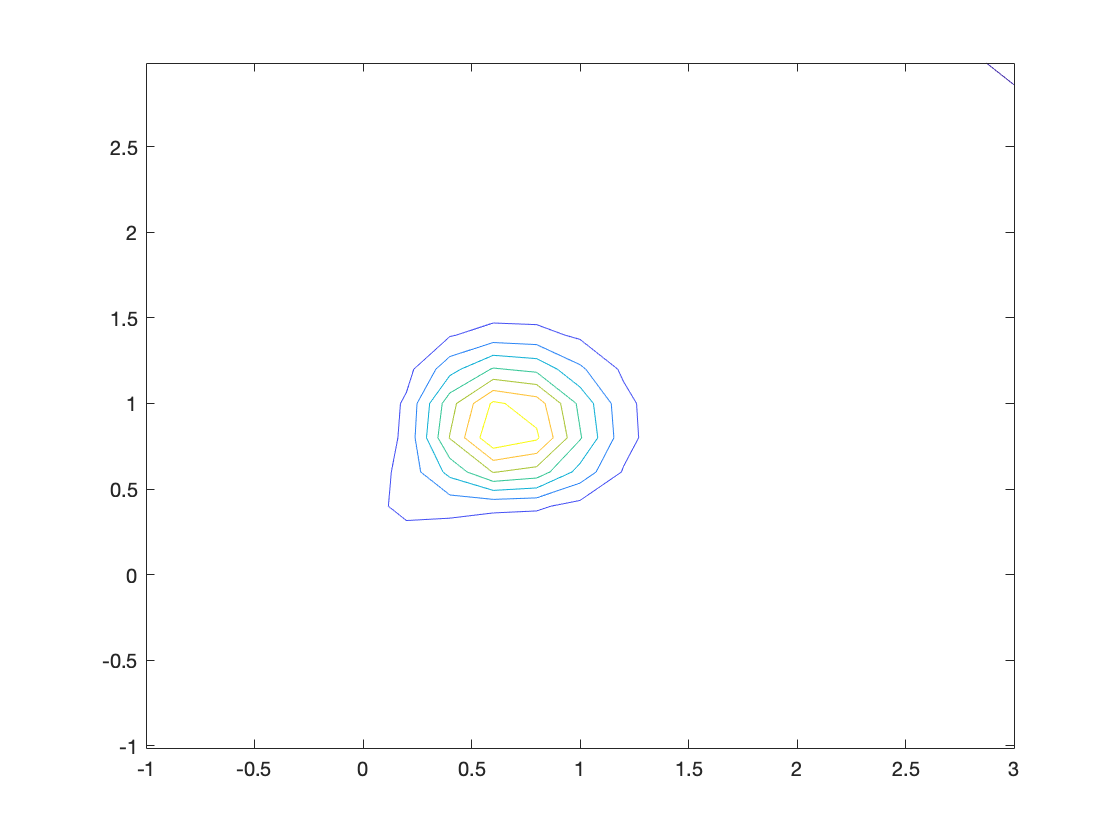}
\includegraphics[width=1.15\linewidth]{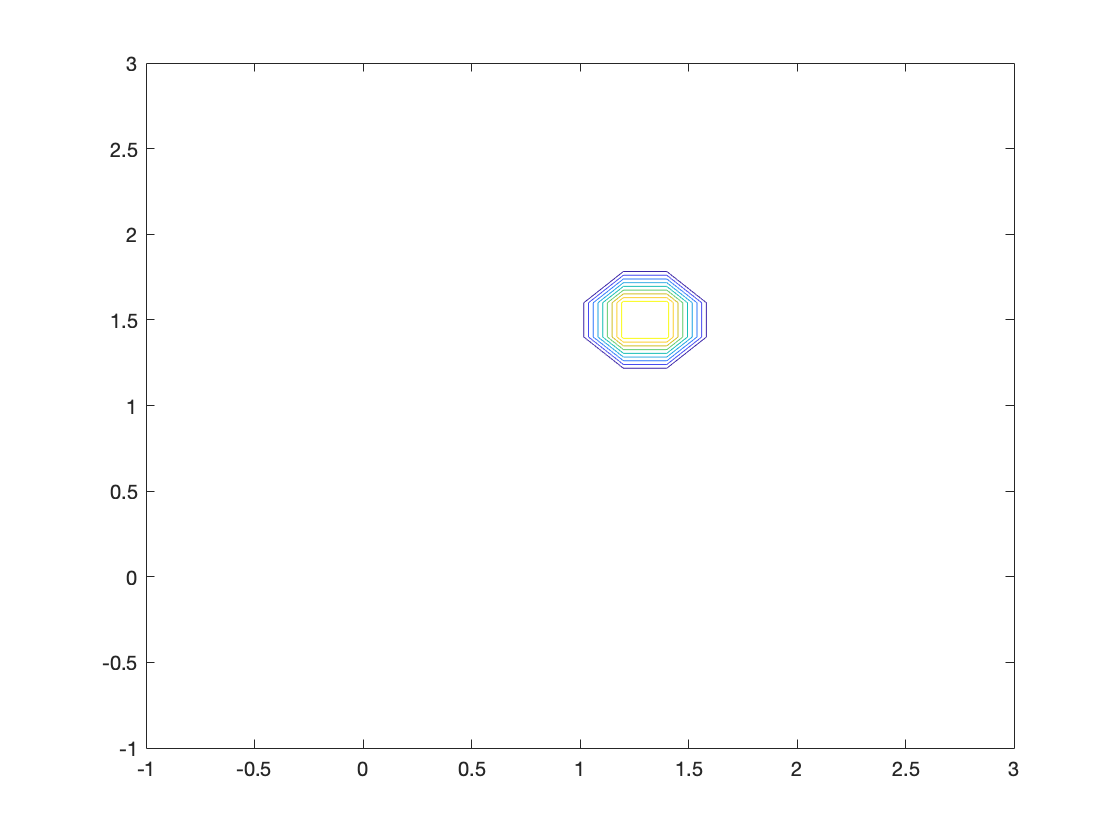}
\end{minipage}}

\centering 
\caption{Example \ref{Ex10}: contour plots of $\rho$ at times $t=0, 0.2, 0.4,0.6,0.8,1$.}
\label{rhot-den}
\end{figure}

\end{example}

\section{Conclusions}

In this paper, we proposed a new algorithm for the geodesic equation with $L^2$-Wasserstein metric on probability set. 
Our algorithm is based on the Benamou-Brenier fluid-mechanics formulation of the OT problem.  Namely, 
we view the geodesic equation as a 
boundary value problem with prescribed initial and terminal probability densities.
To solve the boundary value problem, we adopted the multiple shooting method and
used Newton's method to solve the resulting nonlinear system.  
We further adopted a continuation strategy in order to
enhance our ability to provide good initial guesses for Newton's method. 
Finally, we presented several numerical experiments on challenging problems, to display the 
effectiveness of our algorithm. 

There are many interesting questions that remain to be tackled.  Surely adaptive techniques in space and time
would be very desirable, especially if one wants to extend our numerical method to the
Wasserstein geodesic equations in higher dimension. 
The concern of truncating the spatial domain to a finite computational domain has not been
addressed in our work either, but this is clearly a problem of paramount importance and will require a 
careful theoretical estimation of decay rates of the densities involved.
We expect to tackle some of these issues in future work.

\bibliographystyle{plain}
\bibliography{bib}

\begin{thebibliography}{10}

\bibitem{AscherMattheijRussell:SolnBVPs}
U.~M. Ascher, R.~M. Mattheij, and R.~D. Russell.
\newblock {\em Numerical solution of boundary value problems for ordinary
  differential equations}.
\newblock Prentice Hall Series in Computational Mathematics. Prentice Hall,
  Inc., Englewood Cliffs, NJ, 1988.

\bibitem{BB99}
J.~D. Benamou and Y.~Brenier.
\newblock A numerical method for the optimal time-continuous mass transport
  problem and related problems.
\newblock In {\em Monge {A}mp\`ere equation: applications to geometry and
  optimization ({D}eerfield {B}each, {FL}, 1997)}, volume 226 of {\em Contemp.
  Math.}, pages 1--11. Amer. Math. Soc., Providence, RI, 1999.

\bibitem{BB00}
J.~D. Benamou and Y.~Brenier.
\newblock A computational fluid mechanics solution to the {M}onge-{K}antorovich
  mass transfer problem.
\newblock {\em Numer. Math.}, 84(3):375--393, 2000.

\bibitem{BFO14}
J.~D. Benamou, B.~D. Froese, and A.~M. Oberman.
\newblock Numerical solution of the optimal transportation problem using the
  {M}onge-{A}mp\`ere equation.
\newblock {\em J. Comput. Phys.}, 260:107--126, 2014.

\bibitem{MR1177479}
L.~A. Caffarelli.
\newblock Boundary regularity of maps with convex potentials.
\newblock {\em Comm. Pure Appl. Math.}, 45(9):1141--1151, 1992.

\bibitem{CHYGT18}
Y.~Chen, E.~Haber, K.~Yamamoto, T.~T. Georgiou, and A.~Tannenbaum.
\newblock An efficient algorithm for matrix-valued and vector-valued optimal
  mass transport.
\newblock {\em J. Sci. Comput.}, 77(1):79--100, 2018.

\bibitem{CDLZ19}
S.~Chow, L.~Dieci, W.~Li, and H.~Zhou.
\newblock Entropy dissipation semi-discretization schemes for {F}okker-{P}lanck
  equations.
\newblock {\em J. Dynam. Differential Equations}, 31(2):765--792, 2019.

\bibitem{CLZ20}
S.~Chow, W.~Li, and H.~Zhou.
\newblock Wasserstein {H}amiltonian flows.
\newblock {\em J. Differential Equations}, 268(3):1205--1219, 2020.

\bibitem{CDZ20}
J.~Cui, L.~Dieci, and H.~Zhou.
\newblock Time discretizations of {W}asserstein-{H}amiltonian flows.

\bibitem{Cuturi}
M.~Cuturi.
\newblock Sinkhorn distances: Lightspeed computation of optimal transport.
\newblock In C.~J.~C. Burges, L.~Bottou, M.~Welling, Z.~Ghahramani, and K.~Q.
  Weinberger, editors, {\em Advances in Neural Information Processing Systems},
  volume~26. Curran Associates, Inc., 2013.

\bibitem{LucaJD}
L.~Dieci and J.~D. Walsh, III.
\newblock The boundary method for semi-discrete optimal transport partitions
  and {W}asserstein distance computation.
\newblock {\em J. Comput. Appl. Math.}, 353:318--344, 2019.

\bibitem{Froese}
B.~D. Froese.
\newblock A numerical method for the elliptic {M}onge-{A}mp\`ere equation with
  transport boundary conditions.
\newblock {\em SIAM J. Sci. Comput.}, 34(3):A1432--A1459, 2012.

\bibitem{MR1440931}
W.~Gangbo and R.~J. McCann.
\newblock The geometry of optimal transportation.
\newblock {\em Acta Math.}, 177(2):113--161, 1996.

\bibitem{Givoli}
D.~Givoli.
\newblock {\em Numerical methods for problems in infinite domains}, volume~33
  of {\em Studies in Applied Mechanics}.
\newblock Elsevier Scientific Publishing Co., Amsterdam, 1992.

\bibitem{GLSY16}
X.~Gu, F.~Luo, J.~Sun, and S.~Yau.
\newblock Variational principles for {M}inkowski type problems, discrete
  optimal transport, and discrete {M}onge-{A}mpere equations.
\newblock {\em Asian J. Math.}, 20(2):383--398, 2016.

\bibitem{Kan04}
L.~V. Kantorovich.
\newblock On a problem of {M}onge.
\newblock {\em Zap. Nauchn. Sem. S.-Peterburg. Otdel. Mat. Inst. Steklov.
  (POMI)}, 312(Teor. Predst. Din. Sist. Komb. i Algoritm. Metody. 11):15--16,
  2004.

\bibitem{Keller:NumSoltpbvps}
H.~B. Keller.
\newblock {\em Numerical solution of two point boundary value problems}.
\newblock Society for Industrial and Applied Mathematics, Philadelphia, Pa.,
  1976.
\newblock Regional Conference Series in Applied Mathematics, No. 24.

\bibitem{LYO18}
W.~Li, P.~Yin, and S.~Osher.
\newblock Computations of optimal transport distance with {F}isher information
  regularization.
\newblock {\em J. Sci. Comput.}, 75(3):1581--1595, 2018.

\bibitem{MBMOV16}
L.~M\'etivier, R.~Brossie, Q.~M\'erigot, E.~Oudet, and J.~Virieux.
\newblock Measuring the misfit between seismograms using an optimal transport
  distance: application to full waveform inversion.
\newblock {\em J. Funct. Anal.}, 205(1):345--377, 2016.

\bibitem{ObermanRuan}
A.~M. {Oberman} and Y.~{Ruan}.
\newblock {An efficient linear programming method for Optimal Transportation}.
\newblock {\em arXiv e-prints}, page arXiv:1509.03668, September 2015.

\bibitem{OlikerPrussner}
V.~I. Oliker and L.~D. Prussner.
\newblock On the numerical solution of the equation {$(\partial^2z/\partial
  x^2)(\partial^2z/\partial y^2)-((\partial^2z/\partial x\partial y))^2=f$} and
  its discretizations. {I}.
\newblock {\em Numer. Math.}, 54(3):271--293, 1988.

\bibitem{PPO14}
N.~Papadakis, G.~Peyr\'{e}, and E.~Oudet.
\newblock Optimal transport with proximal splitting.
\newblock {\em SIAM J. Imaging Sci.}, 7(1):212--238, 2014.

\bibitem{PBTIT15}
C.~R. Prins, R.~Beltman, J.~H.~M. ten Thije~Boonkkamp, W.~L. Ijzerman, and
  T.~W. Tukker.
\newblock A least-squares method for optimal transport using the
  {M}onge-{A}mp\`ere equation.
\newblock {\em SIAM J. Sci. Comput.}, 37(6):B937--B961, 2015.

\bibitem{RCLO18}
E.~K. Ryu, Y.~Chen, W.~Li, and S.~Osher.
\newblock Vector and matrix optimal mass transport: theory, algorithm, and
  applications.
\newblock {\em SIAM J. Sci. Comput.}, 40(5):A3675--A3698, 2018.

\bibitem{San15}
F.~Santambrogio.
\newblock {\em Optimal transport for applied mathematicians}, volume~87 of {\em
  Progress in Nonlinear Differential Equations and their Applications}.
\newblock Birkh\"{a}user/Springer, Cham, 2015.
\newblock Calculus of variations, PDEs, and modeling.

\bibitem{TWK18}
E.~Tenetov, G.~Wolansky, and R.~Kimmel.
\newblock Fast entropic regularized optimal transport using semidiscrete cost
  approximation.
\newblock {\em SIAM J. Sci. Comput.}, 40(5):A3400--A3422, 2018.

\bibitem{Vil09}
C.~Villani.
\newblock {\em Optimal transport}, volume 338 of {\em Grundlehren der
  Mathematischen Wissenschaften [Fundamental Principles of Mathematical
  Sciences]}.
\newblock Springer-Verlag, Berlin, 2009.
\newblock Old and new.

\bibitem{WBBC16}
H.~Weller, P.~Browne, C.~Budd, and M.~Cullen.
\newblock Mesh adaptation on the sphere using optimal transport and the
  numerical solution of a {M}onge-{A}mp\`ere type equation.
\newblock {\em J. Comput. Phys.}, 308:102--123, 2016.

\end{thebibliography}

\end{document}